\documentclass{amsart}

\usepackage[T1]{fontenc}
\usepackage{enumerate, amsmath, amsfonts, amssymb, amsthm, mathrsfs, wasysym, graphics, graphicx, xcolor, url, hyperref, hypcap, shuffle, xargs, multicol, overpic, pdflscape, multirow, hvfloat, minibox, accents, array, multido, xifthen, a4wide, ae, aecompl, blkarray, pifont, mathtools, etoolbox, dsfont}
\usepackage{marginnote}
\hypersetup{colorlinks=true, citecolor=darkblue, linkcolor=darkblue}
\usepackage[all]{xy}
\usepackage[bottom]{footmisc}
\usepackage{tikz}
\usepackage{tikz-qtree}
\usetikzlibrary{trees, decorations, decorations.markings, shapes, arrows, matrix, calc, fit, intersections, patterns, angles}
\usepackage[external]{forest}
\usepackage{varwidth}
\graphicspath{{figures/}{figures/nodes/}}
\makeatletter\def\input@path{{figures/}}\makeatother
\usepackage{caption}
\usepackage{placeins} % place floats in the section they are defined
\usepackage{afterpage} % float placements
\usepackage[noabbrev,capitalise]{cleveref}
\usepackage[export]{adjustbox}
\usepackage{ulem}

\newtheorem{theorem}{Theorem}%[section]
\newtheorem{corollary}[theorem]{Corollary}
\newtheorem{proposition}[theorem]{Proposition}
\newtheorem{lemma}[theorem]{Lemma}

\newtheorem*{theorem*}{Theorem}%[section]

\theoremstyle{definition}
\newtheorem{definition}[theorem]{Definition}

\newtheorem{example}[theorem]{Example}
\newtheorem{remark}[theorem]{Remark}

\crefname{notation}{Notation}{Notations}
\crefname{problem}{Problem}{Problems}

\newcommand{\R}{\mathbb{R}} % reals
\newcommand{\HH}{\mathbb{H}} % hyperplane
\renewcommand{\b}[1]{{\boldsymbol{#1}}} % bold letters
\renewcommand{\c}[1]{\mathcal{#1}} % cal letters
\newcommand{\set}[2]{\left\{ #1 \;\middle|\; #2 \right\}} % set notation
\newcommand{\bigset}[2]{\big\{ #1 \;\big|\; #2 \big\}} % big set notation
\newcommand{\ssm}{\smallsetminus} % small set minus
\newcommand{\dotprod}[2]{\left\langle \, #1 \; \middle| \; #2 \, \right\rangle} % dot product
\newcommand{\one}{\b{1}} % the all one vector
\newcommand{\eqdef}{\mbox{\,\raisebox{0.2ex}{\scriptsize\ensuremath{\mathrm:}}\ensuremath{=}\,}} % :=
\newcommand{\simplex}{\b{\triangle}} % simplex
\DeclareMathOperator{\conv}{conv} % convex hull
\DeclareMathOperator{\rank}{rank} % rank

\newcommand{\ie}{\textit{i.e.}~} % id est
\newcommand{\eg}{\textit{e.g.}~} % exempli gratia
\newcommand{\aka}{\textit{a.k.a.}~} % also known as
\definecolor{darkblue}{rgb}{0,0,0.7} % darkblue color
\definecolor{green}{RGB}{57,181,74} % darkblue color
\definecolor{violet}{RGB}{147,39,143} % darkblue color
\newcommand{\darkblue}{\color{darkblue}} % darkblue command
\newcommand{\defn}[1]{\textsl{\darkblue #1}} % emphasis of a definition
\newcommand{\OEIS}[1]{{\rm \href{http://oeis.org/#1}{\texttt{#1}}}}
\usepackage{todonotes}

\newcommand{\meet}{\wedge} % meet
\newcommandx{\projDown}[1][1={}]{\smash{\pi_\downarrow^{#1}}} % down projection map
\newcommandx{\projUp}[1][1={}]{\smash{\pi^\uparrow_{#1}}} % up projection map
\newcommandx{\PT}[1][1=T]{\mathbb{#1}} % painted tree
\newcommandx{\LS}[1][1=S]{\mathbb{#1}} % lighted shade
\DeclareMathOperator{\shadow}{Sh} % shadow
\newcommand{\PTGF}{\mathcal{PT}} % painted tree generating function
\newcommand{\CGF}{\mathcal{C}} % Catalan generating function
\newcommand{\SGF}{\mathcal{S}} % Schroder generating function
\newcommand{\tSGF}{\mathcal{S}_\ast} % twisted Schroder generating function
\newcommand{\LSGF}{\mathcal{LS}} % lighted shade generating function
\newcommand{\surjections}[2]{\mathsf{S}(#1,#2)} % number of surjections
\newcommandx{\Fan}[1][1=n]{\mathcal{F}(#1)} % fan
\newcommand{\polytope}[1]{\mathds{#1}} % font polytopes
\newcommandx{\Perm}[1][1=d]{\polytope{P}\mathrm{erm}(#1)} % permutahedron
\newcommandx{\Asso}[1][1=d]{\polytope{A}\mathrm{sso}(#1)} % associahedron
\newcommandx{\Multiplihedron}[2][1=m, 2=n]{\polytope{M}\mathrm{ul}(#1, #2)} % multiplihedron
\newcommandx{\HP}[2][1=m, 2=n]{\polytope{H}\mathrm{och}(#1, #2)} % Hochschild polytope
\newcommandx{\multiwords}[2][1=m, 2=n]{\textsl{W}(#1, #2)} % multiwords
\DeclareMathOperator{\cube}{Cube} % cube

\title{Hochschild polytopes}

\thanks{VP was partially supported by the Spanish grant PID2022-137283NB-C21 of MCIN/AEI/10.13039/501100011033 / FEDER, UE, by Departament de Recerca i Universitats de la Generalitat de Catalunya (2021 SGR 00697), by the French grant CHARMS (ANR-19-CE40-0017), and by the French--Austrian project PAGCAP (ANR-21-CE48-0020 \& FWF I 5788). DP was partially supported by the Danish National Research Foundation grant DNRF157.}

\author{Vincent Pilaud}
\address{Universitat de Barcelona}
\email{vincent.pilaud@ub.edu}
\urladdr{\url{https://www.ub.edu/comb/vincentpilaud/}}

\author{Daria Poliakova}
\address{Universität Hamburg}
\email{polydarya@gmail.com}
\begin{document}

\begin{abstract}
The $(m,n)$-multiplihedron is a polytope whose faces correspond to $m$-painted \linebreak \mbox{$n$-trees}, and whose oriented skeleton is the Hasse diagram of the rotation lattice on binary $m$-painted $n$-trees.
Deleting certain inequalities from the facet description of the $(m,n)$-multipli\-hedron, we construct the $(m,n)$-Hochschild polytope whose faces correspond to $m$-lighted \mbox{$n$-sha}\-des, and whose oriented skeleton is the Hasse diagram of the rotation lattice on unary $m$-lighted \mbox{$n$-shades}.
Moreover, there is a natural shadow map from $m$-painted $n$-trees to \mbox{$m$-lighted} $n$-shades, which  turns out to define a meet semilattice morphism of rotation lattices.
In particular, when $m=1$, our Hochschild polytope is a deformed permutahedron whose oriented skeleton is the Hasse diagram of the Hochschild lattice.
\end{abstract}

\vspace*{-1.8cm}
\maketitle

%\vspace*{2cm}

\begin{figure}[h]
	\centerline{\includegraphics{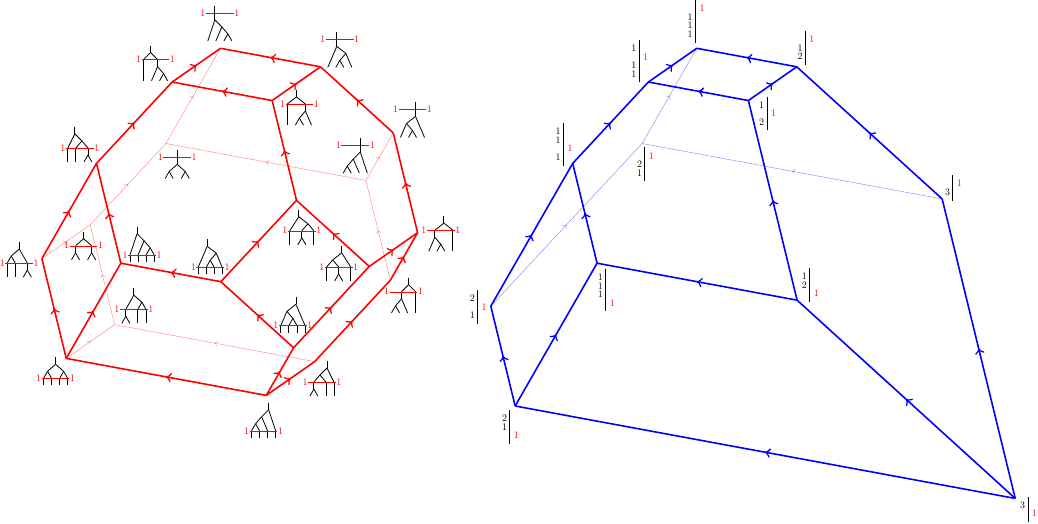}}
%	\centerline{\input{multiplihedronFreehedronLabeledOriented13}}
	\caption{The multiplihedron~$\Multiplihedron[1][3]$ (left) and the Hochschild polytope~$\HP[1][3]$ (right).}
	\label{fig:multiplihedronFreehedronLabeledOriented13}
\end{figure}

%\pagebreak
\vspace*{-.3cm}
\tableofcontents
\vspace*{-.9cm}
%\vspace*{-.3cm}

%%%%%%%%%%%%%%%%%%%%%%%%%%%%%%%%%%%%%%

\section*{Introduction}

We present a remake of the famous combinatorial, geometric, and algebraic interplay between permutations and binary trees.
In the original story, the central character is the surjective map from permutations to binary trees (given by successive binary search tree insertions~\cite{Tonks, HivertNovelliThibon-algebraBinarySearchTrees}).
This map enables us to construct the Tamari lattice~\cite{Tamari} as a lattice quotient of the weak order~\cite{Reading-CambrianLattices}, the sylvester fan as a quotient fan of the braid fan~\cite{Reading-HopfAlgebras}, Loday's associahedron~\cite{ShniderSternberg,Loday} as a removahedron of the permutahedron, and the Loday--Ronco Hopf algebra~\cite{LodayRonco} as a Hopf subalgebra of the Malvenuto--Reutenauer Hopf algebra~\cite{MalvenutoReutenauer}.
Many variations of this saga have been further investigated, notably for other lattice quotients of the weak order~\cite{Reading-HopfAlgebras, ChatelPilaud, PilaudPons-permutrees, Pilaud-brickAlgebra, PilaudSantos-quotientopes, Pilaud-arcDiagramAlgebra} and for generalized associahedra arizing from finite type cluster algebras~\cite{FominZelevinsky-ClusterAlgebrasI, FominZelevinsky-ClusterAlgebrasII, Reading-CambrianLattices, ReadingSpeyer, HohlwegLangeThomas, HohlwegPilaudStella}.
See \cite{PilaudSantosZiegler} for a recent survey on this topic.

%In the present remake, permutations are replaced by binary $m$-painted $n$-trees (binary trees on~$n$ nodes with $m$ horizontal labeled edge cuts), while binary trees are replaced by unary $m$-lighted $n$-shades (compositions of~$n$ with $m$ labels inside their gaps).
In the present remake, permutations are replaced by binary $m$-painted $n$-trees, while binary trees are replaced by unary $m$-lighted $n$-shades.
While their precise definitions are delayed to \cref{sec:combinatorics}, these combinatorial objects are already illustrated in \cref{fig:multiplihedronFreehedronLabeledOriented13}  for~$m = 1$ and~$n = 3$.
The $m$-painted $n$-trees already appeared in \cite[Sect.~3.1]{ChapotonPilaud}, inspired from the case $m = 1$ studied in~\cite{Stasheff-HSpaces, SaneblidzeUmble-diagonals, Forcey-multiplihedra, ForceyLauveSottile, MauWoodward, ArdilaDoker}.
They are mixtures between the permutations of~$[m]$ and the binary trees with~$n$ nodes (here, mixture is meant in the precise sense of shuffle~\cite{ChapotonPilaud}, which is very different from other interpolations of permutations and binary trees, notably permutrees~\cite{PilaudPons-permutrees}).
The $m$-lighted $n$-shades are introduced in this paper, inspired from the case $m = 1$ studied in~\cite{AbadCrainicDherin, Poliakova, Chapoton-Dyck, Combe, Muhle}.
Here again, the central character is a natural surjective map from the former to the latter.
Namely, the shadow map sends an $m$-painted $n$-tree to the \mbox{$m$-lighted} $n$-shade obtained by collecting the arity sequence along the right branch.
In other words, this map records the shadow projected on the right of the tree when the sun sets on the left of the tree.

We first use this map for lattice purposes.
It was proved in~\cite{ChapotonPilaud} that the right rotation digraph on binary $m$-painted $n$-trees (a mixture of the simple transposition digraph on permutations and the right rotation digraph on binary trees) defines a lattice.
We consider here also the right rotation digraph on unary $m$-lighted $n$-shades.
In contrast to the rotation graph on binary \mbox{$m$-painted} $n$-trees, the rotation graph on unary $m$-lighted $n$-shades is regular (each node has $m+n-1$ incoming plus outgoing neighbors).
We prove that it defines as well a lattice by showing that the shadow map is a meet semilattice morphism (but not a lattice morphism).
When~$m = 0$, this gives an unusual meet semilattice morphism from the Tamari lattice to the boolean lattice (distinct from the usual lattice morphism given by the canopy map).
When~$m = 1$, this gives a connection, reminiscent of~\cite{Poliakova}, between the painted tree rotation lattice and the Hochschild lattice introduced in~\cite{Chapoton-Dyck} and studied in~\cite{Combe, Muhle}.
The Hochschild lattice has nice lattice properties: it was proved to be congruence uniform and extremal in~\cite{Combe}, its Galois graph, its canonical join complex and its core label order were described in~\cite{Muhle}, and its Coxeter polynomial was conjectured to be a product of cyclotomic polynomials~\cite[Appendix]{Combe}.
For $m > 1$, computational experiments indicate that the $m$-lighted $n$-shade rotation lattice is still constructible by interval doubling (hence semi-distributive and congruence uniform), but it is not extremal and its Coxeter polynomial is not a product of cyclotomic polynomials.
However, its subposet induced by unary $m$-lighted $n$-shades where the labels of the lights are ordered seems to enjoy all these nice properties.
The lattice theory of the $m$-lighted $n$-shade right rotations certainly deserves to be investigated further.

We then use the shadow map for polytopal purposes.
It was proved in~\cite{ChapotonPilaud} that the coarsening poset on all $m$-painted $n$-trees is isomorphic to the face lattice of a polytope, called the $(m,n)$-multiplihedron~$\Multiplihedron$.
This polytope is a deformed permutahedron (\aka polymatroid~\cite{Edmonds}, or generalized permutahedron~\cite{Postnikov, PostnikovReinerWilliams}) obtained as the shuffle product~\cite{ChapotonPilaud} of an $m$-permutahedron with an $n$-associahedron of J.-L.~Loday~\cite{ShniderSternberg,Loday}.
Oriented in a suitable direction, the skeleton of the $(m,n)$-multiplihedron is isomorphic to the right rotation digraph on binary $m$-painted $n$-trees~\cite{ChapotonPilaud}.
Similarly, we show here that the coarsening poset on all $m$-lighted $n$-shades is isomorphic to the face lattice of a polytope, called the $(m,n)$-Hochschild polytope~$\HP$.
We obtain this polytope by deleting some inequalities in the facet description of the $(m,n)$-multiplihedron.
We also work out the vertex description of the $(m,n)$-Hochschild polytope and its decomposition as Minkowski sum of faces of the standard simplex.
We obtain a deformed permutahedron whose oriented skeleton is isomorphic to the right rotation digraph on unary $m$-lighted $n$-shades.
When~$m = 0$, the $(0,n)$-multiplihedron is the \mbox{$n$-associahedron} and the $(0,n)$-Hochschild polytope is a skew cube (which is not a parallelotope).
When~$m = 1$, the $(1,n)$-multiplihedron is the classical multiplihedron introduced and studied in~\cite{Stasheff-HSpaces, SaneblidzeUmble-diagonals, Forcey-multiplihedra, ForceyLauveSottile, MauWoodward, ArdilaDoker}, and the $(1,n)$-Hochschild polytope is a deformed permutahedron realizing the Hochschild lattice~\cite{Chapoton-Dyck, Combe, Muhle}.
Let us insist here on the fact that our Hochschild polytope provides a much stronger geometric realization of the Hochschild lattice than the two already existing ones.
Namely, the Hochschild lattice is known to be realized
\begin{itemize}
\item on the one hand, as the standard orientation of a graph drawn on the boundary of an hypercube (see \cref{sec:cubicRealizations}), but this graph is not the skeleton of a convex polytope,
\smallskip
\item on the other hand, as an orientation of the skeleton of a convex polytope called freehedron and obtained as a truncation of the standard simplex~\cite{Saneblidze} (or equivalently as the Minkowski sum of the faces of the standard simplex corresponding to all initial and final intervals), but this orientation cannot be obtained as a Morse orientation given by a linear functional (see \cref{exm:badFreehedron}).
\end{itemize}
Finding a deformed permutahedron whose skeleton oriented in the standard linear direction is isomorphic to the Hasse diagram of the Hochschild lattice was an open question raised by F.~Chapoton.

The aficionados of the permutahedron--associahedron saga probably wonder about properties of the singletons of the shadow map (\ie a unary $m$-lighted $n$-shade whose shadow fiber consists of a single binary $m$-painted $n$-tree).
Interestingly, these singletons are counted by binomial transforms of Fibonacci numbers.
Moreover, the facet defining inequalities of~$\Multiplihedron$ that are preserved in~$\HP$ are precisely those that contain a common vertex of~$\Multiplihedron$ and~$\HP$.
This property was essential in the original realization of the Cambrian fans of~\cite{ReadingSpeyer} as generalized associahedra~\cite{HohlwegLangeThomas}.

Somewhat independently, we also show that the right rotation digraph on unary $m$-lighted \mbox{$n$-shades} can also be realized on the boundary of an hypercube, generalizing the existing cubic coordinates for the Hochschild lattice~\cite{Combe}.
Cubic coordinates are well known for many famous lattices (they are called Lehmer codes for weak Bruhat lattices~\cite{Lehmer}, and bracket vectors for Tamari lattices~\cite{HuangTamari}).
In \cite{SaneblidzeUmble-diagonals}, a stronger notion of cubic subdivisions was used to construct combinatorial diagonals for the corresponding polytopes.
%When available, cubic coordinates also provide an elegant alternative proof of the lattice property.

We conclude this introduction by a glance at the algebraic motivation for painted trees and lighted shades, coming from homological algebra.
The family of multiplihedra controls the notion of $A_\infty$-morphisms.
If $A$ and $B$ are two $A_\infty$-algebras and $f : A \to B$ is an $A_\infty$-morphism, then each face of the multiplihedron encodes an operation $A^{ \otimes n} \to B$, with the cellular differential taking care of the relations.
Equivalently, one can view the faces of the multiplihedron as encoding the operations $A^{ \otimes n-1} \otimes M \to N$, where $A$ is an $A_\infty$-algebra and $M$ and $N$ are $A_\infty$-modules over~$A$.
Now if one assumes $A$ strictly associative (DG instead of $A_\infty$), there are fewer such operations.
A universal basis for such operations was constructed in \cite[Thm.~6.4]{AbadCrainicDherin} in the form of short forest-tree-forest triples, and it was observed in \cite[Sect.~5]{Poliakova} that these objects are nothing else but the faces of the freehedra of~\cite{Saneblidze}.
This gave the case $m=1$ of the shadow map \cite[Construction 2]{Poliakova}.

\enlargethispage{-.4cm}
The paper is organized as follows.
In \cref{sec:combinatorics}, we survey the $m$-painted $n$-trees from~\cite{ChapotonPilaud} and introduce the $m$-lighted $n$-shades, and we consider the shadow map sending the former to the latter.
In \cref{sec:polytopes}, we recall the descriptions of the $(m,n)$-multiplihedron, realizing the \mbox{$m$-painted} $n$-tree coarsening lattice, from which we derive the construction of the $(m,n)$-Hochschild polytope, realizing the $m$-lighted $n$-shade coarsening lattice.
Finally, we discuss in \cref{sec:cubicRealizations} the cubic coordinates for $m$-painted $n$-trees and $m$-lighted $n$-shades.

%%%%%%%%%%%%%%%%%%%%%%%%%%%%%%%%%%%%%%

\section{Painted trees and lighted shades}
\label{sec:combinatorics}

In this section, we first recall the combinatorics of $m$-painted $n$-trees (\cref{subsec:paintedTrees}) and introduce that of \mbox{$m$-lighted} $n$-shades (\cref{subsec:lightedShades}).
We then analyse the natural shadow map from $m$-painted $n$-trees to $m$-lighted $n$-shades (\cref{subsec:shadow}), with a particular focus on its singletons (\cref{subsec:singletons}).

%%%%%%%%%

\subsection{$m$-painted $n$-trees}
\label{subsec:paintedTrees}

We start with the combinatorics of $m$-painted $n$-trees already studied in detail in \cite[Sect.~3.1]{ChapotonPilaud}.
It was inspired from the case $m = 1$ studied in~\cite{Stasheff-HSpaces, SaneblidzeUmble-diagonals, Forcey-multiplihedra, ForceyLauveSottile, MauWoodward, ArdilaDoker}.

\begin{definition}
An \defn{$n$-tree} is a rooted plane tree with $n+1$ leaves.
\end{definition}

As usual, we orient such a tree towards its root and label its nodes in inorder.
Namely, each node with $\ell$ subtrees is labeled by an $(\ell-1)$-subset~$\{x_1, \dots, x_{\ell-1}\}$ of~$[n]$ such that all labels in  its $i$th subtree are larger than~$x_{i-1}$ and smaller than~$x_i$ (where by convention~$x_0 = 0$ and~$x_\ell = n+1$). Note in particular that unary nodes receive an empty label.
Note that we use node to refer to internal nodes, excluding the leaves.

\begin{definition}[{\cite[Def.~104]{ChapotonPilaud}}]
A \defn{cut} of an $n$-tree~$T$ is a subset~$c$ of nodes of~$T$ containing precisely one node along the path from the root to any leaf of~$T$.
A cut~$c$ is \defn{below} a cut~$c'$ if the unique node of $c$ is after the unique node of~$c'$ along any path from the root to a leaf of~$T$ (note that we draw trees growing downward).
\end{definition}

\begin{definition}[{\cite[Def.~105]{ChapotonPilaud}}]
\label{def:paintedTrees}
An \defn{$m$-painted $n$-tree}~$\PT \eqdef (T, C, \mu)$ is an $n$-tree~$T$ together with a sequence~$C \eqdef (c_1, \dots, c_k)$ of $k$ cuts of~$T$ and an ordered partition~$\mu$ of~$[m]$ into~$k$ parts for some~$k \in [m]$, such that
\begin{itemize}
\item $c_i$ is below~$c_{i+1}$ for all~$i \in [k-1]$,
\item $\bigcup C \eqdef c_1 \cup \dots \cup c_k$ contains all unary nodes of~$T$.
\end{itemize}
\end{definition}

We represent an $m$-painted $n$-tree~$\PT \eqdef (T, C, \mu)$ as a downward growing tree~$T$, where the cuts of~$C$ are red horizontal lines, labeled by the corresponding parts of~$\mu$. As there is no ambiguity, we write $12$ for the set~$\{1,2\}$. See \cref{fig:paintedTrees,fig:deletionsPaintedTrees,fig:rotationsPaintedTrees} for illustrations.

\afterpage{
\begin{figure}[t]
	\centerline{\includegraphics[scale=.9]{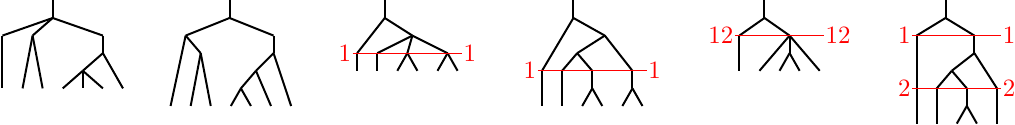}}
%	\centerline{
%		\paintedTree[1]{[[[]][[][]][[[[][][]][]]]]}{}
%		\hspace{-.5cm}
%		\paintedTree[1]{[[[][[][]]][[[[[][]][]][]]]]}{}
%		\hspace{-.5cm}
%		\paintedTree[1]{[[1 []][[1 []][1 [][]][1 [][]]]]}{1}
%		\hspace{-.5cm}
%		\paintedTree[1]{[[1 []][[[1 []][1 [[][]]]][1 [[][]]]]]}{1}
%		\hspace{-.5cm}
%		\paintedTree[1]{[ [12 []][12 [][[][]][]]]}{12}
%		\hspace{-.5cm}
%		\paintedTree[1]{[ [1 [2 []]][1 [[[2 [[]]][2 [[][]]]][2 []]]]]}{1,2}
%	}
	\caption{Some $m$-painted $n$-trees with~$m + n = 6$.}
	\label{fig:paintedTrees}
\end{figure}
}

We now associate to each $m$-painted $n$-tree a preposet (\ie a reflexive and transitive binary relation) on~$[m+n]$.
These preposets will be helpful in several places.

\begin{definition}
\label{def:preorderPaintedTree}
Consider an $m$-painted $n$-tree~$\PT \eqdef (T, C, \mu)$.
Orient~$T$ towards its root, label each node~$x$ of~$T$ by the union~$N(x)$ of the part in~$\mu$ corresponding to the cut of~$C$ passing through~$x$ (empty set if~$x$ is in no cut of~$C$) and the inorder label of~$x$ in the tree~$T$ shifted by~$m$, and finally merge all nodes contained in each cut.
We then define~$\preccurlyeq_{\PT}$ as the preposet on~$[m+n]$ where~$i \preccurlyeq_{\PT} j$ if there is a (possibly empty) oriented path from the node containing~$i$ to the node containing~$j$ in the resulting oriented graph.
See \cref{fig:preposetsPaintedTrees}.
\afterpage{
\begin{figure}[t]
	\centerline{\includegraphics[scale=.9]{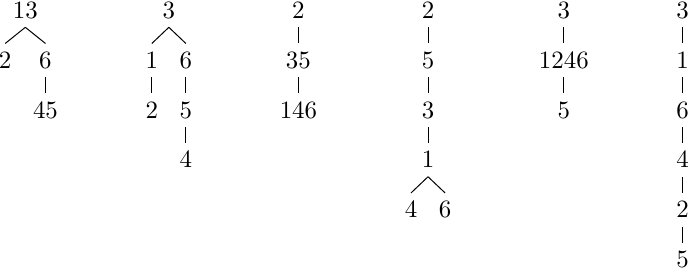}}
%	\centerline{
%		\tree{[.13 [.2 ] [.6 [.45 ] ] ]}
%		\hspace{.3cm}
%		\tree{[.3 [.1 [.2 ] ] [.6 [.5 [.4 ] ] ] ]}
%		\hspace{.3cm}
%		\tree{[.2 [.35 [.146 ] ] ]}
%		\hspace{.3cm}
%		\tree{[.2 [.5 [.3 [.1 [.4 ] [.6 ] ] ] ] ]}
%		\hspace{.3cm}
%		\tree{[.3 [.1246 [.5 ] ] ]}
%		\hspace{.3cm}
%		\tree{[.3 [.1 [.6 [.4 [.2 [.5 ] ] ] ] ] ]}
%	}
	\caption{The preposets~$\preccurlyeq_{\PT}$ associated to the $m$-painted $n$-trees~$\PT$ of \cref{fig:paintedTrees}.}
	\label{fig:preposetsPaintedTrees}
\end{figure}
}
\end{definition}

We now use these preposets to define the coarsening poset on $m$-painted $n$-trees.

\begin{definition}[{\cite[Def.~108]{ChapotonPilaud}}]
The \defn{$m$-painted $n$-tree coarsening poset} is the poset on $m$-painted $n$-trees ordered by coarsening of their corresponding preposets, that is, $\PT \le \PT'$ if~${\preccurlyeq_{\PT}} \subseteq {\preccurlyeq_{\PT'}}$.
\end{definition}

\begin{remark}
Alternatively~\cite[Prop.~111]{ChapotonPilaud}, we could describe the cover relations of the \mbox{$m$-painted} $n$-tree coarsening poset combinatorially by three types of operations, as was done in \cite[Def.~106]{ChapotonPilaud} and illustrated in \cref{fig:deletionsPaintedTrees}.
Namely, to obtain the elements covering an $m$-painted $n$-tree, one can
\begin{enumerate}[(i)]
\item contract an edge whose child is contained in no cut,
\item contract all edges from a parent in no cut to its children all in the same cut,
\item merge two consecutive cuts with no node in between them.
\end{enumerate}
\afterpage{
\begin{figure}[t]
	\centerline{\includegraphics[scale=.9]{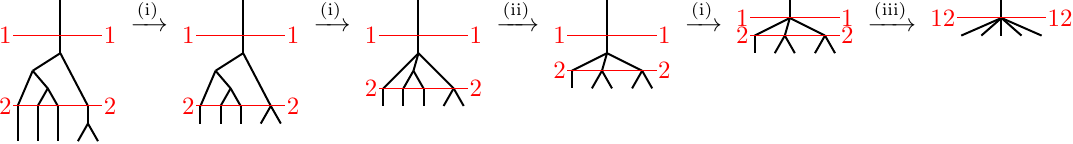}}
%	\centerline{
%		\paintedTree[1]{[1 [[[2 []][[2 []][2 []]]][2 [[][]]]]]}{1,2}
%		\hspace{-.3cm}\raisebox{-.7cm}{$\xrightarrow{\text{(i)}}$}\hspace{-.3cm}
%		\paintedTree[1]{[1 [[[2 []][[2 []][2 []]]][2 [][]]]]}{1,2}
%		\hspace{-.3cm}\raisebox{-.7cm}{$\xrightarrow{\text{(i)}}$}\hspace{-.3cm}
%		\paintedTree[1]{[1 [[2 []][[2 []][2 []]][2 [][]]]]}{1,2}
%		\hspace{-.3cm}\raisebox{-.7cm}{$\xrightarrow{\text{(ii)}}$}\hspace{-.3cm}
%		\paintedTree[1]{[1 [[2 []][2 [][]][2 [][]]]]}{1,2}
%		\hspace{-.3cm}\raisebox{-.7cm}{$\xrightarrow{\text{(i)}}$}\hspace{-.3cm}
%		\paintedTree[1]{[1 [2 []][2 [][]][2 [][]]]}{1,2}
%		\hspace{-.3cm}\raisebox{-.7cm}{$\xrightarrow{\text{(iii)}}$}\hspace{-.3cm}
%		\paintedTree[1]{[12 [][][][][]]}{12}
%	}
	\caption{Coarsenings of some $2$-painted $4$-trees. Each coarsening is labeled by its type.}
	\label{fig:deletionsPaintedTrees}
\end{figure}
}
\end{remark}

In the following statement, we denote by~$|T|$ the number of nodes of a tree~$T$ (including unary nodes), and define~$|C| \eqdef k$ and~$|\bigcup C| \eqdef |c_1 \cup \dots \cup c_k|$ for~$C = (c_1, \dots, c_k)$.

\begin{proposition}[{\cite[Props.~107 \& 116]{ChapotonPilaud}}]
The $m$-painted $n$-tree coarsening poset is a join semilattice ranked by ${m+n-|T|-|C|+|\bigcup C|}$.
We call \defn{$m$-painted $n$-tree coarsening lattice} the lattice obtained by adding a bottom element.
\end{proposition}

We now define another lattice structure, but on minimal $m$-painted $n$-trees.
See \cref{fig:multiplihedronFreehedronLabeledLattice13,fig:multiplihedronFreehedronLabeledLattice3,fig:multiplihedronFreehedronLabeledLattice4}.

\begin{definition}[{\cite[Def.~112]{ChapotonPilaud}}]
An $m$-painted $n$-tree~$\PT \eqdef (T, C, \mu)$ is \defn{binary} if it has rank~$0$, meaning that all nodes in~$\bigcup C$ are unary, while all nodes not in~$\bigcup C$ are binary.
The \defn{binary $m$-painted $n$-tree right rotation digraph} is the directed graph on binary $m$-painted $n$-trees with three types of edges, illustrated in \cref{fig:rotationsPaintedTreesGeneric,fig:rotationsPaintedTrees}:
\begin{enumerate}[(i)]
\item right rotate an edge of~$T$ joining two binary nodes,
\item move a cut of~$C$ past a binary node of~$T$ just above it,
\item exchange the labels of two consecutive cuts of~$C$ with no node of~$T$ in between them, passing the small label above the large label.
\end{enumerate}
\afterpage{
\begin{figure}[t]
	\centerline{\includegraphics[scale=.9]{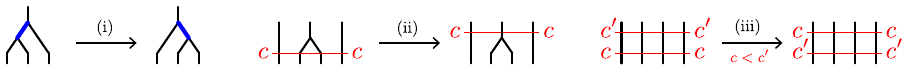}}
	\caption{Rotations of binary $m$-painted $n$-trees. Each rotation is labeled by its type. In type~(i), the rotated edge is colored in blue.}
	\label{fig:rotationsPaintedTreesGeneric}
\end{figure}
\begin{figure}[t]
	\centerline{\includegraphics[scale=.9]{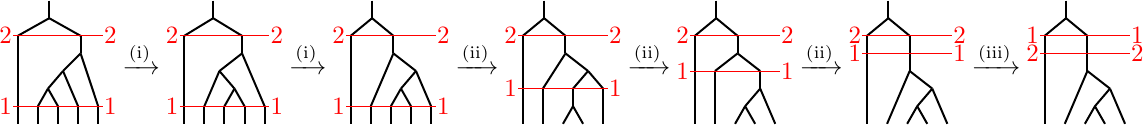}}
%	\centerline{
%		\paintedTree[1]{[ [2 [1 []]][2 [[[[1 []][1 []]][1 []]][1 []]]]]}{2,1}
%		\hspace{-.3cm}\raisebox{-1.2cm}{$\xrightarrow{\text{(i)}}$}\hspace{-.3cm}
%		\paintedTree[1]{[ [2 [1 []]][2 [[[1 []][[1 []][1 []]]][1 []]]]]}{2,1}
%		\hspace{-.3cm}\raisebox{-1.2cm}{$\xrightarrow{\text{(i)}}$}\hspace{-.3cm}
%		\paintedTree[1]{[ [2 [1 []]][2 [[1 []][[[1 []][1 []]][1 []]]]]]}{2,1}
%		\hspace{-.3cm}\raisebox{-1.2cm}{$\xrightarrow{\text{(ii)}}$}\hspace{-.3cm}
%		\paintedTree[1]{[ [2 [1 []]][2 [[1 []][[1 [[][]]][1 []]]]]]}{2,1}
%		\hspace{-.3cm}\raisebox{-1.2cm}{$\xrightarrow{\text{(ii)}}$}\hspace{-.3cm}
%		\paintedTree[1]{[ [2 [1 []]][2 [[1 []][1 [[[][]][]]]]]]}{2,1}
%		\hspace{-.3cm}\raisebox{-1.2cm}{$\xrightarrow{\text{(ii)}}$}\hspace{-.3cm}
%		\paintedTree[1]{[ [2 [1 []]][2 [1 [[][[[][]][]]]]]]}{2,1}
%		\hspace{-.3cm}\raisebox{-1.2cm}{$\xrightarrow{\text{(iii)}}$}\hspace{-.3cm}
%		\paintedTree[1]{[ [1 [2 []]][1 [2 [[][[[][]][]]]]]]}{1,2}
%	}
	\caption{Rotations of some binary $2$-painted $4$-trees. Each rotation is labeled by its type.}
	\label{fig:rotationsPaintedTrees}
\end{figure}
}
\end{definition}

\begin{remark}
\label{rem:rotationsPaintedTrees}
We could alternatively describe the right rotations on binary $m$-painted $n$-trees using their poset of \cref{def:preorderPaintedTree}.
Namely, there is an edge~$\PT \to \PT'$ if and only if there are linear extensions~$\sigma$ of~$\preccurlyeq_{\PT}$ and~$\sigma'$ of~$\preccurlyeq_{\PT'}$ such that $\sigma \lessdot \sigma'$ is a cover relation in the weak order.
The equivalence between these two perspectives follows from \cref{prop:fanPaintedTrees,prop:graphPaintedTrees} \mbox{(see \cite[Props.~118 \& 119]{ChapotonPilaud})}.
\end{remark}

\begin{proposition}[{\cite[Def.~119]{ChapotonPilaud}}]
The binary $m$-painted $n$-tree right rotation digraph is the Hasse diagram of a lattice.
\end{proposition}

\begin{example}
When~$m = 0$, the $0$-painted $n$-tree rotation lattice is the Tamari lattice~\cite{Tamari, HuangTamari}.
When~$m = 1$, the $1$-painted $n$-tree rotation lattice is the multiplihedron lattice introduced in~\cite{ChapotonPilaud}.
\afterpage{
\begin{figure}[t]
	\centerline{\includegraphics{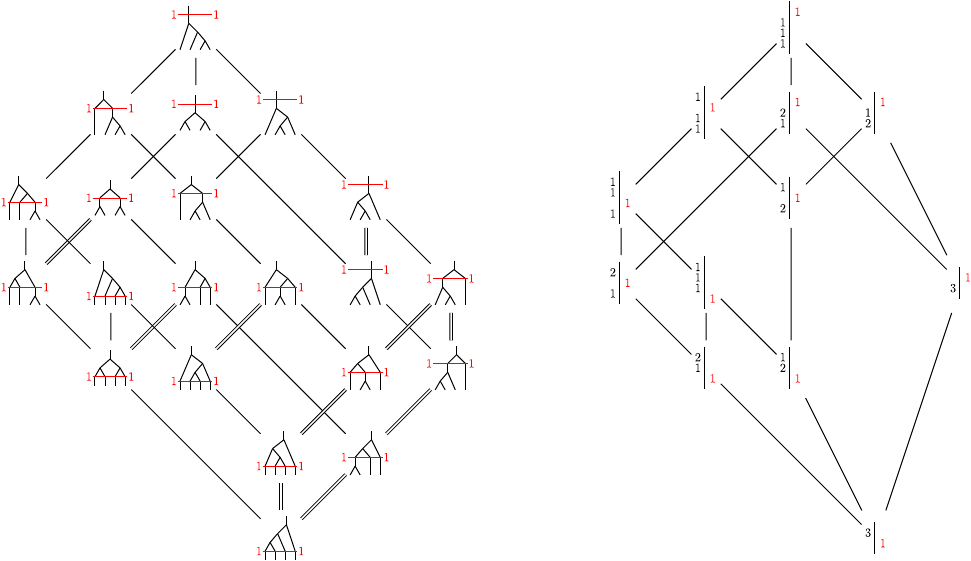}}
%	\centerline{\input{multiplihedronFreehedronLabeledLatticeDoubled13}}
	\caption{The $1$-painted $3$-tree (left) and $1$-lighted $3$-shade (right) rotation lattices.}
	\label{fig:multiplihedronFreehedronLabeledLattice13}
\end{figure}
}
\end{example}

\begin{remark}
Note that the $m$-painted $n$-tree rotation lattice is meet semidistributive, but not join semidistributive when~$m \ge 1$.
% Note that the problem does not come from the permutation part:
%sage: [[(m, n, LatticePoset(poset_skeleton(multiplihedron(m,n))).is_meet_semidistributive()) for n in range(6)] for m in range(3)]                                                                        
%[[(0, 0, True),
%  (0, 1, True),
%  (0, 2, True),
%  (0, 3, True),
%  (0, 4, True),
%  (0, 5, True)],
% [(1, 0, True),
%  (1, 1, True),
%  (1, 2, True),
%  (1, 3, True),
%  (1, 4, True),
%  (1, 5, True)],
% [(2, 0, True),
%  (2, 1, True),
%  (2, 2, True),
%  (2, 3, True),
%  (2, 4, True),
%  (2, 5, True)]]
%sage: LatticePoset(poset_skeleton(multiplihedron(1,3))).is_semidistributive()                                                                                                                             
%False
\end{remark}

\enlargethispage{.2cm}
We conclude this recollection on $m$-painted $n$-trees by some enumerative observations.
See also \cref{table:verticesMultiplihedra,,table:facetsMultiplihedra,,table:facesMultiplihedra} in \cref{subsec:tablesMultiplihedra}.

\begin{proposition}[{\cite[Prop.~126]{ChapotonPilaud}}]
\label{prop:numberBinaryPaintedTrees}
The number of binary $m$-painted $n$-trees is
\[
m! \, [y^{n+1}] \, \CGF^{(m+1)}(y),
\]
where~$[y^{n+1}]$ selects the coefficient of~$y^{n+1}$, and $\CGF^{(i)}(y)$ is defined for~$i \ge 1$ by
\[
\CGF^{(1)}(y) \eqdef \CGF(y)
\qquad\text{and}\qquad
\CGF^{(i+1)}(y) \eqdef \CGF \big( \CGF^{(i)}(y) \big),
\]
where
\[
\CGF(y) = \frac{1-\sqrt{1-4y}}{2}
\]
is the Catalan generating function.
See \cref{table:verticesMultiplihedra} in \cref{subsec:tablesMultiplihedra} for the first few numbers.
\end{proposition}

\begin{proposition}[{\cite[Prop.~127]{ChapotonPilaud}}]
\label{prop:numberShortPaintedTrees}
The number of rank $m+n-2$ (that is, corank~$1$) $m$-painted $n$-trees~is
\[
 \binom{n+1}{2} - 1 + 2^{m+n} - 2^n.
\]
See \cref{table:facetsMultiplihedra} in \cref{subsec:tablesMultiplihedra} for the first few numbers.
\end{proposition}

\begin{proposition}[{\cite[Prop.~128]{ChapotonPilaud}}]
\label{prop:numberPaintedTrees}
The generating function~$\smash{\PTGF(x,y,z) \eqdef \!\!\!\! \sum\limits_{m,n,p} \!\!\! PT(m,n,p) \, x^m y^n z^p}$ of the number of rank~$p$ $m$-painted $n$-trees is given by
\[
\PTGF(x,y,z) = \sum_m x^m \sum_{k = 0}^m \SGF \big( \tSGF^{(k)}(y,z), z \big) \, \surjections{m}{k} \, z^{m-k},
\]
where~$\surjections{m}{k}$ is the number of surjections from~$[m]$ to~$[k]$,
\[
\SGF(y,z) = \frac{1+yz-\sqrt{1-4y-2yz+y^2z^2}}{2(z+1)}
\]
is the Schr\"oder generating function, and~$\tSGF^{(i)}(y,z)$ is defined for~$i \ge 0$~by
\[
\tSGF^{(0)}(y,z) \eqdef y,
\quad
\tSGF^{(1)}(y,z) \eqdef (1+z) \, \SGF(y,z) - yz
\quad\text{and}\quad
\tSGF^{(i+1)}(y,z) \eqdef \tSGF^{(i)} \big( \tSGF^{(1)}(y,z), z \big).
\]
See \cref{table:facesMultiplihedra} in \cref{subsec:tablesMultiplihedra} for the first few numbers.
\end{proposition}

\begin{example}
When~$m = 0$, the number of $0$-painted $n$-trees of rank $0$, rank $n-2$ and arbitrary rank are respectively given by the classical Catalan numbers (\OEIS{A000108}), the interval numbers (\OEIS{A000096}) and the Schr\"oder numbers (\OEIS{A001003}).
\end{example}

%%%%%%%%%

\begin{figure}[p]
	\centerline{\includegraphics[scale=1.3]{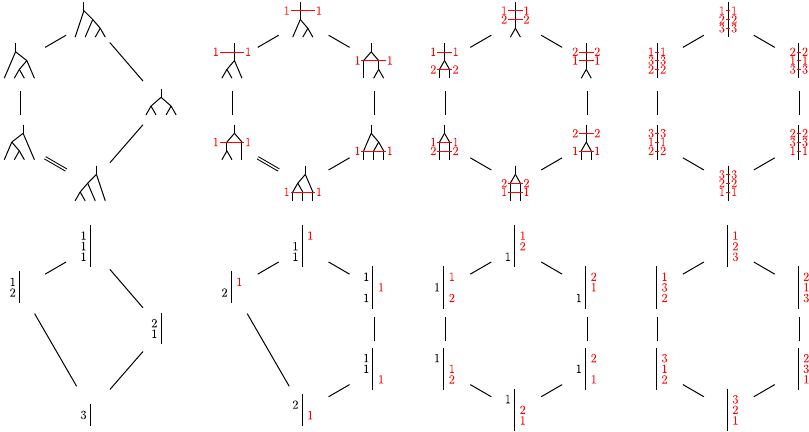}}
%	\centerline{\input{multiplihedronFreehedronLabeledLatticeDoubled3}}
	\caption{The $m$-painted $n$-tree rotation lattice (top) and the $m$-lighted $n$-shade rotation lattice (bottom) for $(m,n) = (0,3)$, $(1,2)$, $(2,1)$, and~$(3,0)$ (left to right).}
	\label{fig:multiplihedronFreehedronLabeledLattice3}
\end{figure}

\begin{figure}[p]
	\vspace{.7cm}
	\centerline{\includegraphics[scale=1.1]{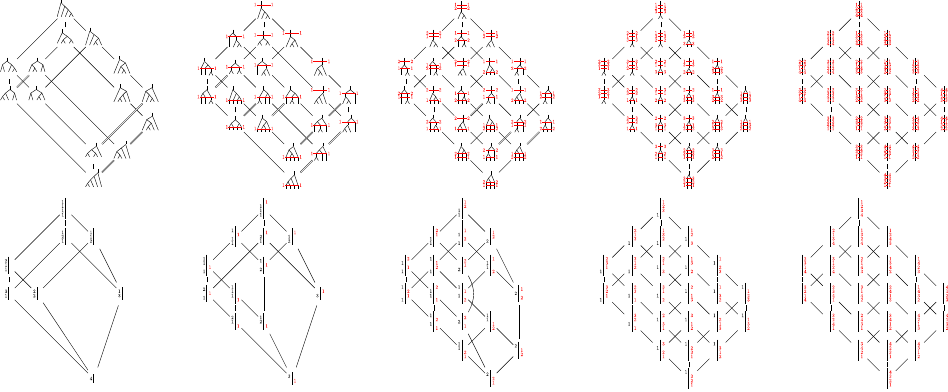}}
%	\centerline{\scalebox{.5}{\input{multiplihedronFreehedronLabeledLatticeDoubled4}}}
	\caption{The $m$-painted $n$-tree rotation lattice (top) and the $m$-lighted $n$-shade rotation lattice (bottom) for $(m,n) = (0,4)$, $(1,3)$, $(2,2)$, $(3,1)$, and~$(4,0)$ (left to right).}
	\label{fig:multiplihedronFreehedronLabeledLattice4}
\end{figure}

%\begin{figure}[h]
%	\centerline{\includegraphics[scale=.9]{multiplihedronFreehedronLabeledLattice04}}}
%%	\centerline{\input{multiplihedronFreehedronLabeledLattice04}}
%	\caption{The $0$-painted $4$-tree (left) and $0$-lighted $4$-shade (right) rotation lattices.}
%\end{figure}
%
%\begin{figure}[h]
%	\centerline{\includegraphics[scale=.9]{multiplihedronFreehedronLabeledLattice13}}}
%%	\centerline{\input{multiplihedronFreehedronLabeledLattice13}}
%	\caption{The $1$-painted $3$-tree (left) and $1$-lighted $3$-shade (right) rotation lattices.}
%\end{figure}
%
%\begin{figure}[h]
%	\centerline{\includegraphics[scale=.9]{multiplihedronFreehedronLabeledLattice22}}}
%%	\centerline{\input{multiplihedronFreehedronLabeledLattice22}}
%	\caption{The $2$-painted $2$-tree (left) and $2$-lighted $2$-shade (right) rotation lattices.}
%\end{figure}
%
%\begin{figure}[h]
%	\centerline{\includegraphics[scale=.9]{multiplihedronFreehedronLabeledLattice31}}}
%%	\centerline{\input{multiplihedronFreehedronLabeledLattice31}}
%	\caption{The $3$-painted $1$-tree (left) and $3$-lighted $1$-shade (right) rotation lattices.}
%\end{figure}
%
%\begin{figure}[h]
%	\centerline{\includegraphics[scale=.9]{multiplihedronFreehedronLabeledLattice40}}}
%%	\centerline{\input{multiplihedronFreehedronLabeledLattice40}}
%	\caption{The $4$-painted $0$-tree (left) and $4$-lighted $0$-shade (right) rotation lattices.}
%\end{figure}

%%%%%%%%%

\subsection{$m$-lighted $n$-shades}
\label{subsec:lightedShades}

We now introduce the main new characters of this paper, which will later appear as certain shadows of $m$-painted $n$-trees.

\begin{definition}
\label{def:lightedShades}
An \defn{$n$-shade} is a sequence of (possibly empty) tuples of integers, whose total sum is~$n$.
%We call \defn{length} the number of tuples in this sequence.
%\end{definition}
%
%%\begin{definition}
%%A \defn{cut} in an $n$-shade~$S$ is just the position of one of the tuples of~$S$.
%%A cut~$C$ is \defn{below} a cut~$C'$ if the position~$C$ is larger than the position~$C'$.
%%\end{definition}
%
%\begin{definition}
An \defn{$m$-lighted $n$-shade}~$\LS \eqdef (S, C, \mu)$ is an $n$-shade~$S$ together with a set~$C$ of $k$ distinguished positions in~$S$, containing all positions of empty tuples of~$S$, and an ordered partition~$\mu$ of~$[m]$ into~$k$ parts for some~$k \in [m]$.
\end{definition}

\begin{remark}
Alternatively, we could define an $m$-lighted $n$-shade as a pair~$(S,C)$ of sequences of the same length, where~$S$ contains (possibly empty) tuples of integers and has total sum~$n$, while $C$ contains (possibly empty) disjoint subsets of~$[m]$ whose union is~$[m]$, and~$c_i$ is nonempty when~$s_i$ is the empty tuple.
We preferred the version of \cref{def:lightedShades} to be more parallel to \cref{def:paintedTrees}.
\end{remark}

We represent an $m$-lighted $n$-shade~$\LS \eqdef (S, C, \mu)$ as a vertical line, with the tuples of the sequence~$S$ in black on the left, and the cuts of~$C$ in red on the right, all from top to bottom. As there is no ambiguity, we write $12$ for the tuple $(1,2)$ or the set~$\{1,2\}$. See \cref{fig:lightedShades,fig:deletionsLightedShades,fig:rotationsLightedShades} for illustrations.

\begin{figure}[h]
	\centerline{\includegraphics[scale=.9]{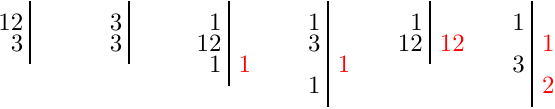}}
%	\centerline{
%		\lightedShade[1]{[{12} [3 []]]}{}
%		\hspace{-.5cm}
%		\lightedShade[1]{[3 [3 []]]}{}
%		\hspace{-.5cm}
%		\lightedShade[1]{[1 [{12} [1, tier=1 []]]]}{1}
%		\hspace{-.5cm}
%		\lightedShade[1]{[1 [3 [, tier=1 [1 []]]]]}{1}
%		\hspace{-.5cm}
%		\lightedShade[1]{[1 [{12}, tier=12 []]]}{12}
%		\hspace{-.5cm}
%		\lightedShade[1]{[1 [, tier=1 [3 [, tier=2 []]]]]}{1,2}
%	}
	\caption{Some $m$-lighted $n$-shades with~$m + n = 6$.}
	\label{fig:lightedShades}
\end{figure}

We now associate to each $m$-lighted $n$-shade a preposet on~$[m+n]$.
These preposets will be helpful in several places.

\begin{definition}
\label{def:preorderLightedShade}
Consider an $m$-lighted $n$-shade~$\LS \eqdef (S, C, \mu)$.
Define the \defn{range}~$r(s)$ of a tuple~$s$ of~$S$ as the interval~$]u,v]$ where~$u$ (resp.~$v$) is $m$ plus the sum of all entries of all tuples of~$S$ which appear strictly (resp.~weakly) earlier than~$s$.
Define the \defn{preceeding sum}~$ps(x)$ of an entry~$x$ in a tuple of~$S$ as $m$ plus the sum of all entries that appear weakly before~$x$ in~$S$ (meaning either all entries in a strictly earlier tuple of~$S$, or the weakly earlier entries in the same tuple as~$x$).
For each tuple~$s = (s_1, \dots, s_\ell)$ of~$S$, let~$N(s)$ denote the union of the part of~$\mu$ corresponding to the cut of~$C$ passing through~$s$ (empty set if~$s$ is in no cut of~$C$) and the set~$\{ps(s_1), \dots, ps(s_\ell)\}$.
Define also~$M(s) \eqdef r(s) \ssm N(s)$.
Finally, consider the directed graph with one node~$N(s)$ for each tuple~$s$ of~$S$ and one singleton node~$\{i\}$ for each~$i \in \bigcup_{s \in S} M(s)$, where~$N(s)$ has incoming arcs from the singletons~$\{i\}$ for~$i \in M(s)$ and the node~$N(s')$ of the next tuple~$s'$ of~$S$ after~$s$.
We then define~$\preccurlyeq_{\LS}$ as the preposet on~$[m+n]$ where~$i \preccurlyeq_{\LS} j$ if there is a (possibly empty) oriented path from the node containing~$i$ to the node containing~$j$ in the resulting oriented graph.
%given by the relations
%\begin{itemize}
%\item $i \preccurlyeq_{\LS} j$ if~$i,j \in [m]$ and~$i$ appears weakly after~$j$ in~$\mu$,
%\item $k \preccurlyeq_{\LS} ps(y)$ if~$x$ and~$y$ are elements of tuples of~$S$ such that the tuple of~$x$ appears weakly after the tuple of~$y$, and~$ps(x)-x < k \le ps(x)$,
%\item $i \preccurlyeq_{\LS} ps(x)$ if~$i \in [m]$ and~$x$ is an element of a tuple of~$S$ which appears weakly before the cut containing~$i$,
%\item $k \preccurlyeq_{\LS} i$ if~$i \in [m]$ and~$ps(x)-x < k \le ps(x)$ for some element~$x$ of a tuple of~$S$ which appears weakly after the cut containing~$i$.
%\end{itemize}
See \cref{fig:preposetsLightedShades}.
\begin{figure}[h]
	\centerline{\includegraphics[scale=.9]{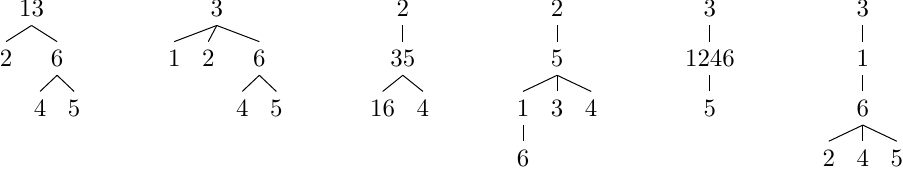}}
%	\centerline{
%		\tree{[.13 [.2 ] [.6 [.4 ] [.5 ] ] ]}
%		\hspace{.3cm}
%		\tree{[.3 [.1 ] [.2 ] [.6 [.4 ] [.5 ] ] ]}
%		\hspace{.3cm}
%		\tree{[.2 [.35 [.16 ] [.4 ] ] ]}
%		\hspace{.3cm}
%		\tree{[.2 [.5 [.1 [.6 ] ] [.3 ] [.4 ] ] ]}
%		\hspace{.3cm}
%		\tree{[.3 [.1246 [.5 ] ] ]}
%		\hspace{.3cm}
%		\tree{[.3 [.1 [.6 [.2 ] [.4 ] [.5 ] ] ] ]}
%	}
	\caption{The preposets~$\preccurlyeq_{\LS}$ associated to the $m$-lighted $n$-shades~$\LS$ of \cref{fig:lightedShades}.}
	\label{fig:preposetsLightedShades}
\end{figure}
\end{definition}

\begin{remark}
\label{rem:caterpillar}
Define the Hasse diagram of a preposet~$\preccurlyeq$ on~$X$ to be the Hasse diagram of the poset~${\preccurlyeq} / {\equiv}$ on the classes of the equivalence relation~${\equiv} \eqdef \set{(x,y) \in X \times X}{x \preccurlyeq y \text{ and } y \preccurlyeq x}$ defined by~$\preccurlyeq$.
The description of \cref{def:preorderLightedShade} implies that the Hasse diagram of the preposet~$\preccurlyeq_{\LS}$ of an $m$-lighted $n$-shade~$\LS$ is always a tree, in contrast to the preposet~$\preccurlyeq_{\PT}$ of an $m$-painted $n$-tree~$\PT$.
In fact, the Hasse diagram of~$\preccurlyeq_{\LS}$ is a \defn{caterpillar tree}, \ie a path containing the nodes~$N(s)$ for all tuples~$s$ of~$S$, to which are attached the nodes of~$\bigcup_{s \in S} M(s)$.
%More precisely, the Hasse diagram of~$\preccurlyeq_{\LS}$ is a caterpillar forest, whose path contains one node~$\{ps(x_1), \dots, ps(x_k)\}$ for each tuple~$(x_1, \dots, x_k)$ of~$\LS$.
\end{remark}

We now use these preposets to define the coarsening poset on $m$-lighted $n$-shades.

\begin{definition}
The \defn{$m$-lighted $n$-shade coarsening poset} is the poset on $m$-lighted $n$-shades defined by coarsening of their corresponding preposets, that is, $\LS \le \LS'$ if~${\preccurlyeq_{\LS}} \subseteq {\preccurlyeq_{\LS'}}$.
\end{definition}

\begin{remark}
\label{rem:coarseningLightedShades}
Alternatively, we could describe the cover relations of the $m$-lighted $n$-shades coarsening poset combinatorially by two types of operations, as illustrated in \cref{fig:deletionsLightedShades}.
Namely, to obtain the elements covering an $m$-lighted $n$-shade, one can
\begin{enumerate}[(i)]
\item concatenate two consecutive (possibly empty) tuples, and merge their (possibly empty) cuts,
\item replace one of the integers~$x$ inside a tuple by two integers~$y,z$ with~$x = y + z$ and~$y \ge 1$ and~$z \ge 1$.
\end{enumerate}
\begin{figure}[t]
	\centerline{\includegraphics[scale=.9]{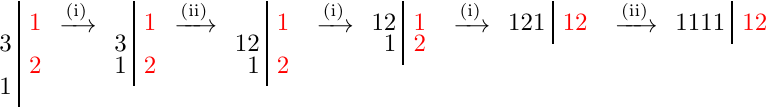}}
%	\centerline{
%		\lightedShade[1]{[, tier=1 [3 [, tier=2 [1 []]]]]}{1,2}
%		\hspace{-.3cm}\raisebox{-.7cm}{$\xrightarrow{\text{(i)}}$}\hspace{-.3cm}
%		\lightedShade[1]{[, tier=1 [3 [1, tier=2 []]]]}{1,2}
%		\hspace{-.3cm}\raisebox{-.7cm}{$\xrightarrow{\text{(ii)}}$}\hspace{-.3cm}
%		\lightedShade[1]{[, tier=1 [{12} [1, tier=2 []]]]}{1,2}
%		\hspace{-.3cm}\raisebox{-.7cm}{$\xrightarrow{\text{(i)}}$}\hspace{-.3cm}
%		\lightedShade[1]{[{12}, tier=1 [1, tier=2 []]]}{1,2}
%		\hspace{-.3cm}\raisebox{-.7cm}{$\xrightarrow{\text{(i)}}$}\hspace{-.3cm}
%		\lightedShade[1]{[{121}, tier=12 []]}{12}
%		\hspace{-.3cm}\raisebox{-.7cm}{$\xrightarrow{\text{(ii)}}$}\hspace{-.3cm}
%		\lightedShade[1]{[{1111}, tier=12 []]}{12}
%	}
	\caption{Coarsenings of some $2$-lighted $4$-shades. Each coarsening is labeled by its type.}
	\label{fig:deletionsLightedShades}
\end{figure}
\end{remark}

\enlargethispage{.1cm}
For a sequence~$S \eqdef (s_1, \dots, s_\ell)$ of tuples, we define~$|S| \eqdef \ell$ and~$\|S\| \eqdef \sum_{i \in [\ell]} |s_i|$, where~$|s_i|$ is the length of the tuple~$s_i$.

\begin{proposition}
\label{prop:lightedShadeCoarseningLattice}
The $m$-lighted $n$-shade coarsening poset is a join semilattice ranked by~${m - |S| + \|S\|}$.
We call \defn{$m$-lighted $n$-shade coarsening lattice} the lattice obtained by adding a bottom element.
\end{proposition}

\begin{proof}
For the rank, if~$\LS \eqdef (S,C,\mu)$ and~$\LS' \eqdef (S',C',\mu')$ are obtained by one of the two operations of \cref{rem:coarseningLightedShades}, then we have
\begin{enumerate}[(i)]
\item $|S'| = |S|-1$ and~$\|S'\| = \|S\|$ when we concatenate two consecutive tuples,
\item $|S'| = |S|$ and~$\|S'\| = \|S\|+1$ when we refine an integer into two inside one of the tuples.
\end{enumerate}
In both situations, we get~$\rank(\LS') = \rank(\LS)+1$.
Finally, the join semilattice property will follow from~\cref{prop:fanLightedShades2}.
\end{proof}

We now define another lattice structure, but on minimal $m$-lighted $n$-shades.
\mbox{See \cref{fig:multiplihedronFreehedronLabeledLattice13,fig:multiplihedronFreehedronLabeledLattice3,fig:multiplihedronFreehedronLabeledLattice4}}.

\begin{definition}
\label{def:rotationLightedShades}
An $m$-lighted $n$-shade~$\LS \eqdef (S, C, \mu)$ is \defn{unary} if it has rank~$0$, meaning that all tuples in~$\bigcup C$ are empty tuples, while all tuples not in~$\bigcup C$ are singletons.
The \defn{unary $m$-lighted $n$-shade right rotation digraph} is the directed graph on unary $m$-lighted $n$-shades with three types of edges, illustrated in \cref{fig:rotationsLigthedShadesGeneric,fig:rotationsLightedShades}:
\begin{enumerate}[(i)]
\item replace a singleton~$(r)$ of~$S$ by two singletons~$(s), (t)$ with~$r = s + t$ and~$s \ge 1$ and~$t \ge 1$,
\item exchange a singleton of~$S$ with a cut of~$C$ below it,
\item exchange the labels of two consecutive cuts of~$C$ with no singleton in between them, passing the small label above the large label.
\end{enumerate}
\begin{figure}[t]
	\centerline{\includegraphics[scale=.9]{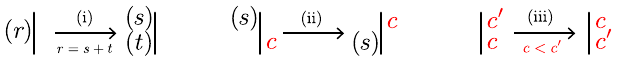}}
	\caption{Rotations of unary $m$-lighted $n$-shades. Each rotation is labeled by its type.}
	\label{fig:rotationsLigthedShadesGeneric}
\end{figure}
\begin{figure}[t]
	\centerline{\includegraphics[scale=.9]{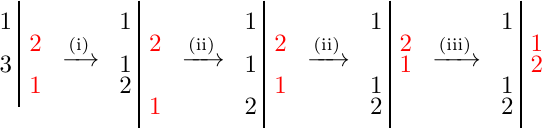}}
%	\centerline{
%		\lightedShade[1]{[1 [, tier=2 [3 [, tier=1 []]]]]}{1,2}
%		\hspace{-.3cm}\raisebox{-1.2cm}{$\xrightarrow{\text{(i)}}$}\hspace{-.3cm}
%		\lightedShade[1]{[1 [, tier=2 [1 [2 [, tier=1 []]]]]]}{1,2}
%		\hspace{-.3cm}\raisebox{-1.2cm}{$\xrightarrow{\text{(ii)}}$}\hspace{-.3cm}
%		\lightedShade[1]{[1 [, tier=2 [1 [, tier=1 [2 []]]]]]}{1,2}
%		\hspace{-.3cm}\raisebox{-1.2cm}{$\xrightarrow{\text{(ii)}}$}\hspace{-.3cm}
%		\lightedShade[1]{[1 [, tier=2 [, tier=1 [1 [2 []]]]]]}{1,2}
%		\hspace{-.3cm}\raisebox{-1.2cm}{$\xrightarrow{\text{(iii)}}$}\hspace{-.3cm}
%		\lightedShade[1]{[1 [, tier=1 [, tier=2 [1 [2 []]]]]]}{2,1}
%	}
	\caption{Rotations of some unary $2$-lighted $4$-shades. Each rotation is labeled by its type.}
	\label{fig:rotationsLightedShades}
\end{figure}
\end{definition}

\begin{remark}
\label{rem:rotationsLightedShades}
As in \cref{rem:rotationsPaintedTrees}, we could alternatively describe the right rotations on unary $m$-lighted $n$-shades using their poset of \cref{def:preorderLightedShade}.
Namely, there is an edge joining~$\LS$ to~$\LS'$ if and only if there are linear extensions~$\sigma$ of~$\preccurlyeq_{\LS}$ and~$\sigma'$ of~$\preccurlyeq_{\LS'}$ such that $\sigma \lessdot \sigma'$ is a cover relation in the weak order.
The equivalence between these two perspectives follows from \cref{prop:fanLightedShades,prop:graphLightedShades}.
\end{remark}

\begin{remark}
\label{rem:rotationGraphLightedShadesRegular}
From~\cref{def:rotationLightedShades}, we observe that any unary $m$-lighted $n$-shade~$\LS$ with singleton tuples~$s_1, \dots, s_r$ admits~$m+r-1+\sum_{i \in [r]} (s_i-1) = m+n-1$ (left or right) rotations.
In other words, the (undirected) rotation graph is regular of degree~$m+n-1$.
Note that this can also be seen as a consequence of \cref{def:preorderLightedShade} as the Hasse diagram of~$\preccurlyeq_{\LS}$ is a tree.
\end{remark}

The next statement will follow from \cref{prop:shadowMapSemilatticeMapRotation}.

\begin{proposition}
\label{prop:lightedShadeRotationLattice}
The unary $m$-lighted $n$-shade right rotation digraph is the Hasse diagram of a lattice.
\end{proposition}

\begin{example}
When~$m \! = \! 0$, the $0$-lighted $n$-shade rotation lattice is the boolean lattice.
When~${m \! = \! 1}$, the $1$-lighted $n$-shade rotation lattice is the Hochschild lattice studied in~\cite{Chapoton-Dyck, Combe, Muhle}.
\end{example}

\begin{remark}
\label{rem:latticeProperties}
Computational experiments indicate that the $m$-lighted $n$-shade rotation lattice is constructible by interval doubling (hence semidistributive and congruence uniform).
However, in contrast to the case when~$m \le 1$, it is not extremal (see~\cite{Muhle} for context), and its Coxeter polynomial is not a product of cyclotomic polynomials (see \cite{Chapoton-PFCIC} and \cite[Appendix]{Combe} for context).
Nevertheless, its subposet induced by unary $m$-lighted $n$-shades where the labels of the lights are ordered (see also \cref{def:multiwords}) seems to enjoy all these nice properties.
The lattice properties of the $m$-lighted $n$-shade rotation lattice and its subposet deserve to be investigated~further.
\end{remark}

\enlargethispage{.4cm}
We conclude this section on $m$-lighted $n$-shades by some enumerative observations.
See also \cref{table:verticesHochschildPolytope,,table:facetsHochschildPolytope,,table:facesHochschildPolytope} in \cref{subsec:tablesHochschildPolytope}.

\begin{proposition}
\label{prop:numberUnaryLightedShades}
The number of unary $m$-lighted $n$-shades is
\[
m! \sum_{\ell = 1}^n \binom{m+\ell}{m} \binom{n-1}{\ell-1}.
\]
See \cref{table:verticesHochschildPolytope} in \cref{subsec:tablesHochschildPolytope} for the first few numbers.
\end{proposition}

\begin{proof}
The number of unary $m$-lighted $n$-shades with~$\ell$ singletons is given by~$m! \binom{m+\ell}{\ell} \binom{n-1}{\ell-1}$.
Namely, choose the order of the cuts (hence $m!$ choices), the position of the~$m$ cuts and~$\ell$ singletons (hence $\binom{m+\ell}{\ell}$ choices), and the values of the $\ell$ singletons (hence $\binom{n-1}{\ell-1}$ choices).
\end{proof}

\begin{proposition}
\label{prop:numberShortLightedShades}
The number of rank $m+n-2$ (that is, corank~$1$) $m$-lighted $n$-shades is
\[
(2^m+1)(n+1) - 4 + \delta_{n=0}.
\]
See \cref{table:facetsHochschildPolytope} in \cref{subsec:tablesHochschildPolytope} for the first few numbers.
\end{proposition}

\begin{proof}
Consider an $m$-lighted $n$-shade~$\LS \eqdef (S, C, \mu)$ with~$S \eqdef (s_1, \dots, s_\ell)$
According to \cref{prop:lightedShadeCoarseningLattice}, $\LS$ has rank~$m+n-2$ if and only if~$n-2 = \|S\|-|S| = \sum_{i \in [\ell]} |s_i|-1$, or equivalently if and only if one of the following holds:
\begin{itemize}
\item either $\ell = 1$ and~$s_1 = 1^{i-1} 2 1^{n-1-i}$ for some~$i \in [n-1]$ (hence~$n-1$ choices), 
\item or~$\ell = 2$ and~$s_1 = 1^i$ while~$s_2 = 1^{n-i}$ for some~$i \in [n-1]$ and the $m$ labels are allocated arbitrarily on the two positions (hence~$2^m(n-1)$ choices),
\item or~$\ell = 2$ and~$s_1 = \varnothing$ while~$s_2 = 1^n$ and the $m$ labels are allocated on the two tuples, with at least one label on the first tuple (hence~$2^m-1$ choices),
\item or~$\ell = 2$ and~$s_1 = 1^n$ while~$s_2 = \varnothing$ and the $m$ labels are allocated on the two tuples, with at least one label on the second tuple (hence~$2^m-1$ choices).
\end{itemize}
We obtain that there are~$n-1+2^m(n-1)+2(2^m-1) = (2^m+1)(n+1) - 4$ rank~${m+n-2}$ \mbox{$m$-lighted} $n$-shades.
This analysis does not precisely apply when~$n = 0$ which requires the correction term~$\delta_{n=0}$.
\end{proof}

\begin{proposition}
\label{prop:numberLightedShades}
The generating function~$\smash{\LSGF(x,y,z) \eqdef \sum\limits_{m,n,p} LS(m,n,p) \, x^m y^n z^p}$ of the number of rank~$p$ $m$-lighted $n$-shades is given by
\[
\LSGF(x,y,z) = \sum_m x^m \sum_{k = 0}^m \frac{ \big( 1-y \big)^k \big( 1-y(z+1) \big)}{\big( 1-y(z+2) \big)^{k+1}} \, \surjections{m}{k} \, z^{m-k},
\]
where~$\surjections{m}{k}$ is the number of surjections from~$[m]$ to~$[k]$.
See \cref{table:facesHochschildPolytope} in \cref{subsec:tablesHochschildPolytope} for the first few numbers.
\end{proposition}

\begin{proof}
Denote by
\[
\tau^\ge(y,z) = \frac{1}{1-yz/(1-y)} = \frac{1-y}{1-y(z+1)}
\qquad\text{(resp.}\quad
\tau^>(y,z) = \frac{yz}{1-y(z+1)}
\text{ )}
\]
the generating function of all (resp.~nonempty) tuples of integers, where~$y$ counts the sum, and $z$ counts the length.
For fixed integers~$0 \le k \le m$, we claim that the generating function~$\sum_{n,p} LS(m,n,p) \, y^n \, z^p$  of rank~$p$ $m$-lighted $n$-shades with~$k$ cuts is given by
\[
\big( \tau^\ge(y,z)/z \big)^k \Big( \frac{1}{1-\tau^>(y,z)/z} \Big)^{k+1} \, \surjections{m}{k} \, z^m = \frac{ \big( 1-y \big)^k \big( 1-y(z+1) \big)}{\big( 1-y(z+2) \big)^{k+1}} \, \surjections{m}{k} \, z^{m-k}
\]
Indeed, to construct a rank~$p$ $m$-lighted shade with~$k$ cuts, we need to choose
\begin{itemize}
\item $k$ possibly empty tuples for the cuts, hence $k$ factors~$\tau^\ge(y,z)/z$,
\item $k+1$ possibly empty sequences of nonempty tuples in between the cuts (including before the first cut and after the last cut), hence $k+1$ factors $\frac{1}{1-\tau^>(y,z)/z}$,
\item an ordered partition of~$[m]$ into $k$ parts, hence a factor~$\surjections{m}{k}$.
\end{itemize}
As the rank of an $m$-lighted $n$-shade~$\LS \eqdef (S,C,\mu)$ with~$S \eqdef (s_1, \dots, s_\ell)$ is given by~${m - |S| + \|S\|} = m + \sum_{i \in [\ell]} (|s_i| - 1)$, we divide both~$\tau^\ge(y,z)$ and~$\tau^>(y,z)$ by~$z$, and we finally multiply by~$z^m$.
\end{proof}

\begin{example}
When~$m = 0$, the number of $0$-painted $n$-trees of rank $0$, rank $n-2$ and arbitrary rank are respectively given by $2^{n-1}$ (\OEIS{A000079}), $2(n-1)$ (\OEIS{A005843}) and $3^{n-1}$ (\OEIS{A000244}).
When~$m = 1$, the number of $1$-painted $n$-trees of rank $0$, rank $n-1$ are respectively given by $2^{n-2}(n+3)$ (\OEIS{A045623}) and $3n-1$ (\OEIS{A016789}).
\end{example}

%%%%%%%%%

\subsection{Shadow map}
\label{subsec:shadow}

We now describe a natural shadow map sending an $m$-painted $n$-tree to an $m$-lighted $n$-shade.
Intuitively, the shadow is what you see on the right of the tree when the sun sets on its left.
For instance, the $m$-painted $n$-trees of \cref{fig:paintedTrees} are sent to the $m$-lighted $n$-shade of \cref{fig:lightedShades}.
We call \defn{right branch} of a tree~$T$ the path from the root to the rightmost leaf~of~$T$.

\begin{definition}
\label{def:shadowMap}
The \defn{shadow} of an $n$-tree~$T$ is the $n$-shade~$\shadow(T)$ obtained by replacing each node on the right branch of~$T$, whose subtrees~$T_1, \dots, T_p$ from left to right have~$n_1, \dots, n_p$ leaves, by the tuple~$(n_1, \dots, n_{p-1})$ (note that the number of leaves~$n_p$ of the rightmost child is omitted in this sequence).
%\begin{itemize}
%\item contracting all edges joining a child to a parent which does not lie on the right branch~of~$T$,
%\item replacing each node on the right branch of~$T$ by the tuple of the arities of its children except its rightmost.
%\end{itemize}
The \defn{shadow} of a cut~$c$ in~$T$ is the position~$\shadow(c)$ in~$\shadow(T)$ of the unique node of the right branch of~$T$ contained in~$c$.
For a sequence~$C = (c_1, \dots, c_k)$, define~$\shadow(C) \eqdef (\shadow(c_1), \dots, \shadow(c_k))$.
The \defn{shadow} of an $m$-painted $n$-tree~$\PT \eqdef (T, C, \mu)$ is the \mbox{$m$-lighted} $n$-shade~$\shadow(\PT) \eqdef (\shadow(S), \shadow(C), \mu)$.
\end{definition}

\begin{definition}
\label{def:shadowCongruence}
The \defn{shadow congruence} is the equivalence class on $m$-painted $n$-trees whose equivalence classes are the fibers of the shadow map. In other words, two $m$-painted $n$-trees are shadow congruent if they have the same shadow.
\end{definition}

\begin{proposition}
\label{prop:shadowPreposets}
For any $m$-painted $n$-tree~$\PT$ and any $m$-lighted $n$-shade~$\LS$, we have
\[
{\preccurlyeq_{\LS}} \subseteq {\preccurlyeq_{\PT}} \quad \iff \quad {\preccurlyeq_{\LS}} \subseteq {\preccurlyeq_{\shadow(\PT)}}.
\]
\end{proposition}

\begin{proof}
%Let~$\PT \eqdef (T, C, \mu)$ be an $m$-painted $n$-tree and~$\LS \eqdef \shadow(\PT) = (\shadow(S), \shadow(C), \mu)$.
Observe first that for a painted tree~$\PT \eqdef (T, C, \mu)$,
\begin{itemize}
\item the $i$th node~$n_i$ on the right branch of~$T$ corresponds to the $i$th tuple~$s_i$ of~$\shadow(T)$,
\item the label~$N(n_i)$ in \cref{def:preorderPaintedTree} coincides with the label~$N(s_i)$ in \cref{def:preorderLightedShade},
\item the inorder labels shifted by~$m$ of the descendants of $n_i$ which are not descendants of~$n_{i+1}$ precisely cover the range~$r(s_i)$ of~$s_i$.
\end{itemize}
We thus derive that~${\preccurlyeq_{\shadow(\PT)}} \subseteq {\preccurlyeq_{\PT}}$, which implies the backward direction.

For the forward direction, let~$\LS$ be such that~${\preccurlyeq_{\LS}} \subseteq {\preccurlyeq_{\PT}}$.
By \cref{def:preorderLightedShade}, for each entry~$x$ in a tuple of~$\LS$, the preceeding sum~$ps(x)$ is weakly larger in~$\preccurlyeq_{\LS}$ than all elements of~$]ps(x),m+n]$.
As~${\preccurlyeq_{\LS}} \subseteq {\preccurlyeq_{\PT}}$, this implies that, in the inorder labeling of~$\PT$ shifted by~$m$, the value~$p(x)$ appears inside a node on the right branch of~$\PT$.
Moreover, if~$x$ and~$y$ are two entries of the same tuple~$s$ of~$\LS$, then $ps(x) \equiv_{\LS} ps(y)$, so that~$p(x)$ and~$p(y)$ appear at the same node of~$\PT$.
We thus obtain that the set~$N(s)$ of any tuple~$s$ of~$\LS$ is a subset of a set~$N(t)$ of some tuple~$t$ of~$\shadow(\PT)$.
We conclude that~${{\preccurlyeq_{\LS}} \subseteq {\preccurlyeq_{\shadow(\PT)}}}$.
%Finally, according to the description of \cref{rem:coarseningLightedShades}, any coarsening of~$\shadow(\PT)$ to an $m$-lighted $n$-shade~$\LS$ would
%\begin{itemize}
%\item either concatenate two consecutive tuples~$s$ and~$s'$ of~$\shadow(\PT)$, resulting in the union of~$N(s)$ and~$N(s')$ in a larger equivalence class of~$\preccurlyeq_{\LS}$,
%\item or replace one of the integers~$x$ in a tuple of~$\shadow(\PT)$ by two integers~$y,z$ with~$x = y+z$, resulting in the addition of~$ps(y)$ to~$N(s)$ in~$\preccurlyeq_{\LS}$.
%\end{itemize}
%In both cases, the additional relations are not relations of~$\preccurlyeq_{\PT}$.
\end{proof}

We now study the behavior of the shadow map and shadow congruence of \cref{def:shadowMap,def:shadowCongruence} on the $m$-painted $n$-tree lattice, as illustrated in \cref{fig:multiplihedronFreehedronLabeledLattice13,fig:multiplihedronFreehedronLabeledLattice3,fig:multiplihedronFreehedronLabeledLattice4}.
In all these pictures, the cover relations between two $m$-painted $n$-trees in the same shadow class are represented by double~arrows.

Given two meet semilattices~$(M, \meet)$ and~$(M', \meet')$, a map~$f : M \to M'$ is a \defn{meet semilattice morphism} if~$f(x \meet y) = f(x) \meet' f(y)$ for all~$x,y \in M$.
The fibers of~$f$ are the classes of a \defn{meet semilattice congruence}~$\equiv_f$ on~$M$, and the image of~$f$ is then called the \defn{meet semilattice quotient} of~$M$ by~$\equiv_f$.
In other words, an equivalence relation~$\equiv$ on~$M$ is a meet semilattice congruence when~$x_1 \equiv x_2$ and~$y_1 \equiv y_2$ implies~$x_1 \meet y_1 \equiv x_2 \meet y_2$, and the quotient~$M / {\equiv}$ is the meet semilattice on the $\equiv$-equivalence classes, where for two $\equiv$-equivalence classes~$X$ and~$Y$, 
\begin{itemize}
\item the order relation is given by~$X \le Y$ if there exist representatives~$x \in X$ and~$y \in Y$ such that~$x \le y$, 
\item the meet~$X \meet Y$ is the $\equiv$-equivalence class of~$x \meet y$ for any representatives~$x \in X$ and~$y \in Y$.
\end{itemize}
The following classical criterion will be fundamental.
A proof of the similar criterion for lattice congruences can be found \eg in~\cite[Prop.~9-5.2]{Reading-PosetRegionsChapter}.
We adapt this proof here to meet semilattice congruences for the convenience of the reader.
Recall that~$X \subseteq M$ is \defn{order convex} if~$x \le y \le z$ and~$x,z \in X$ implies~$y \in X$.

\begin{proposition}
\label{prop:characterizationMeetSemilatticeCongruence}
An equivalence relation~$\equiv$ on a finite meet semilattice~$(M, \meet)$ is a meet semilattice congruence if and only if
\begin{enumerate}[(i)]
\item each $\equiv$-equivalence class is order convex and admits a unique minimal element,
\item the map~$\projDown : M \to M$ sending an element of~$M$ to the minimal element of its $\equiv$-equivalence class is order preserving.
\end{enumerate}
\end{proposition}

\begin{proof}
Assume first that~$\equiv$ is a meet semilattice congruence. Then
\begin{enumerate}[(i)]
\item Let~$X$ be a $\equiv$-equivalence class.
If~$x \le y \le z$ and~$x,z \in X$, then~$x = x \meet y \equiv z \meet y = y$ and thus~$X$ is order convex.
Since~$(M,\meet)$ is a finite meet semilattice, $X$ has at least one minimal element.
Assume for contradiction that it has two minimal elements~$x$ and~$y$.
Since~$x \equiv y$, we have~$x \meet y \equiv x \meet x = x$ so that~$x \le y$.
By symmetry, we obtain that~$x = y$.
\item Consider now~$x \le y$ in~$M$.
As~$\projDown(x) \equiv x$ and~$\projDown(y) \equiv y$, we have~$x = x \meet y \equiv \projDown(x) \meet \projDown(y)$.
As~$\projDown(x)$ is minimal in the class of~$x$, we thus obtain that~$\projDown(x) \le \projDown(x) \meet \projDown(y) \le \projDown(y)$.
\end{enumerate}

Conversely, assume that~$\equiv$ is an equivalence relation satisfying~(i) and~(ii).
For any~$x, y \in M$, we have~$\projDown(x) \le x$ and~$\projDown(y) \le y$ so that~$\projDown(x) \meet \projDown(y) \le x \meet y$, hence~$\projDown \big( \projDown(x) \meet \projDown(y) \big) \le \projDown(x \meet y)$.
Moreover, as~$x \meet y \le x$ and~$x \meet y \le y$, we have~$\projDown(x \meet y) \le \projDown(x)$ and~$\projDown(x \meet y) \le \projDown(y)$, so that~$\projDown(x \meet y) \le \projDown(x) \meet \projDown(y)$.
Combining these two inequalities, we obtain that
\[
\projDown\big( \projDown(x) \meet \projDown(y) \big) \le \projDown(x \meet y) \le \projDown(x) \meet \projDown(y).
\]
Since~$\projDown\big( \projDown(x) \meet \projDown(y) \big) \equiv \projDown(x) \meet \projDown(y)$, we obtain that~$x \meet y \equiv \projDown(x \meet y) \equiv \projDown(x) \meet \projDown(y)$ by order convexity.
Now for~$x_1,x_2,y_1,y_2 \in M$, if~$x_1 \equiv x_2$ and~$y_1 \equiv y_2$, we have~$\projDown(x_1) = \projDown(x_2)$ and~$\projDown(y_1) = \projDown(y_2)$ so that~$x_1 \meet y_1 \equiv \projDown(x_1) \meet \projDown(y_1) = \projDown(x_2) \meet \projDown(y_2) \equiv x_2 \meet y_2$.
\end{proof}

\enlargethispage{.3cm}
%Join semilattice morphisms, congruences and quotients are defined dually, and the dual characterization to \cref{prop:characterizationMeetSemilatticeCongruence} holds.
We will now apply this characterization to the shadow congruence on the binary $m$-painted $n$-tree rotation lattice.
This will prove along the way that the binary $m$-painted $n$-tree rotation poset is a meet semilattice quotient, hence a meetsemilattice, hence a lattice as it is bounded.
This completes the proof of \cref{prop:lightedShadeRotationLattice}.

\begin{proposition}
\label{prop:shadowMapSemilatticeMapRotation}
The shadow map is a surjective meet semilattice morphism from the binary \mbox{$m$-painted} $n$-tree rotation lattice to the unary $m$-lighted $n$-shade rotation lattice.
\end{proposition}

\begin{proof}
%We first describe the fibers of the shadow map from binary $m$-painted $n$-trees to unary $m$-lighted $n$-shades.
The shadow fiber of a unary $m$-lighted $n$-shade~$\LS \eqdef (S, C, \mu)$ is the set of all binary \mbox{$m$-painted} $n$-trees obtained by replacing each entry~$s$ of~$S$ by a binary tree on $s$ leaves cut by all lines of~$C$ below~$s$.
In other words, this fiber is the Cartesian product of the $m_s$-painted $s$-tree right rotation digraphs for all entries~$s$ of~$S$ where~$m_s$ is the number of lines of~$C$ below~$s$.
It is thus clearly an interval, whose minimal element is obtained by replacing each entry~$s$ of~$S$ by a left comb on~$s$ leaves, cut at the level of its leaves by all lines of~$C$ below~$s$.
Hence, the minimal element~$\projDown(\PT)$ in the shadow class of a binary $m$-painted $n$-tree~$\PT \eqdef (T, C, \mu)$ is obtained by replacing each left subtree of~$T$ by a comb, cut by the lines of~$C$ at the level of its leaves.
See \cref{fig:multiplihedronFreehedronLabeledLattice13,fig:multiplihedronFreehedronLabeledLattice3,fig:multiplihedronFreehedronLabeledLattice4} for illustrations.
%In other words, the unique minimal element~$\projDown(\PT)$ in the shadow class of an $m$-painted $n$-tree~$\PT \eqdef (T, C, \mu)$ is obtained by replacing each left subtree of~$T$ by a comb, cut by the lines of~$C$ at the level of its leaves.

Consider now two binary $m$-painted $n$-trees~$\PT \eqdef (T, C, \mu)$ and~$\PT' \eqdef (T', C', \mu')$ connected by a right rotation.
If this rotation does not affect the right branch of~$\PT$, then~$\PT$ and~$\PT'$ are shadow congruent, so that~$\projDown(\PT) = \projDown(\PT')$.
Assume now that this rotation affects the right branch.
There are three possible such flips:
\begin{enumerate}[(i)]
\item Assume first that we rotate an edge~$i \to j$ in~$\PT$ (with~$j$ on the right branch of~$\PT$) to an edge~$i \leftarrow j$ in~$\PT'$ (with both~$i$ and~$j$ on the right branch of~$\PT'$). Then~$\projDown(\PT)$ and~$\projDown(\PT')$ coincide except that~$\projDown(\PT)$ has a left comb at~$j$ (cut at the level of its leaves by all lines of~$C$ below~$j$) while~$\projDown(\PT')$ has a left comb at~$i$ and a left comb at~$j$ (both cut at the level of their leaves by all lines of~$C$ below~$j$). As the left comb is the rotation minimal binary tree, we can perform a sequence of right rotations in~$\projDown(\PT)$ to obtain~$\projDown(\PT')$. Note here that it is crucial that the cuts appear in the left combs of~$\projDown(\PT)$ and~$\projDown(\PT')$ at the level of their leaves so that these binary tree rotations are indeed painted tree rotations.
\item Assume now that we move a cut~$c$ past a binary node~$i$ (on the right branch) to pass from~$\PT$ to~$\PT'$. Then~$\projDown(\PT)$ and~$\projDown(\PT')$ coincide except that the left comb at node~$i$ of~$\projDown(\PT)$ is completely above~$c$ while the left comb at node~$i$ of~$\projDown(\PT')$ is completely below~$c$. Hence, we can successively move the cut~$c$ past each node of the left comb at node~$i$ of~$\projDown(\PT)$ to obtain~$\projDown(\PT')$.
\item Assume finally that we exchange the labels of two consecutive cuts with no node in between them to pass from~$\PT$ to~$\PT'$. Then we can exchange the labels of the same cuts to pass from~$\projDown(\PT)$ to~$\projDown(\PT')$, since they are still consecutive and still have no node between them.
\end{enumerate}
In all cases, we obtain that~$\projDown(\PT) < \projDown(\PT')$.
We conclude that~$\projDown$ is order preserving, so that the shadow map is a meet semilattice morphism by \cref{prop:characterizationMeetSemilatticeCongruence}.
\end{proof}

\begin{example}
When~$m = 0$, we obtain an unusual meet semilattice morphism from the Tamari lattice to the boolean lattice (distinct from the usual lattice morphism given by the canopy map).
When~$m = 1$, we obtain a meet semilattice morphism from the multiplihedron lattice to the Hochschild lattice, reminiscent of~\cite{Poliakova}.
\end{example}

\begin{remark}
Note that the shadow map is not a join semilattice morphism.
For instance,
\[
\includegraphics[scale=.9]{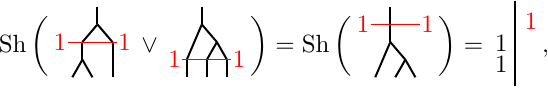}
%\shadow \bigg( \hspace{-.4cm} \raisebox{.8cm}{\paintedTree[1]{[[1[[][]]][1[]]]}{1}} \hspace{-.4cm} \join \hspace{-.4cm} \raisebox{.8cm}{\paintedTree[1]{[[1[]][[1[]][1[]]]]}{1}} \hspace{-.4cm} \bigg) = \shadow \bigg( \hspace{-.4cm} \raisebox{.8cm}{\paintedTree[1]{[1[[][[][]]]]}{1}} \hspace{-.4cm} \bigg) = \hspace{-.4cm} \raisebox{1cm}{\lightedShade[1]{[, tier=1 [1 [1 []]]]}{1}} \hspace{-.4cm},
\]
while
\[
\includegraphics[scale=.9]{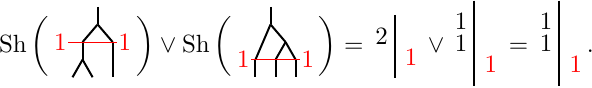}
%\shadow \bigg( \hspace{-.4cm} \raisebox{.8cm}{\paintedTree[1]{[[1[[][]]][1[]]]}{1}} \hspace{-.4cm} \bigg) \join \shadow \bigg( \hspace{-.4cm} \raisebox{.8cm}{\paintedTree[1]{[[1[]][[1[]][1[]]]]}{1}} \hspace{-.4cm} \bigg) = \hspace{-.4cm} \raisebox{.8cm}{\lightedShade[1]{[2 [, tier=1 []]]}{1}} \hspace{-.4cm} \join \hspace{-.4cm} \raisebox{1cm}{\lightedShade[1]{[1 [1 [, tier=1 []]]]}{1}} \hspace{-.4cm} = \hspace{-.4cm} \raisebox{1cm}{\lightedShade[1]{[1 [1 [, tier=1 []]]]}{1}} \hspace{-.4cm}.
\]
Note that there is already a counter-example with~$(m,n) = (0,3)$, see \cref{fig:multiplihedronFreehedronLabeledLattice3}\,(left).

If we tried to apply the (dual) characterization of \cref{prop:characterizationMeetSemilatticeCongruence}, we would observe that, even if the fiber of a unary $m$-lighted $n$-shade~$\LS \eqdef (S, C, \mu)$ has a unique maximal element (obtained by replacing each element~$x$ of~$S$ by a right comb on~$x$ leaves, cut at the level of its root by all lines of~$C$ below~$x$), the projection map~$\projUp$ is not order preserving.
\end{remark}

\begin{remark}
Note that we could also consider the left shadow map, given by the arity sequence on the left branch of the $m$-painted $n$-tree.
It defines a join semilattice morphism, which is not a meet semilattice morphism.
It would also be interesting to consider the map that records the arity sequence along the path from the root to the $i$th leaf.
And of course all intersections of the equivalence relations arising from these arity sequence maps.
\end{remark}

\begin{remark}
Note that it is crucial here that our orientation of the skeleton of the $(m,n)$-multiplihe\-dron gives advantage to the permutation part over the binary tree part (in other words, that we consider the shuffle of the $m$-permutahedron with the $n$-associahedron).
Indeed, as observed in the proof of \cref{prop:shadowMapSemilatticeMapRotation}, we need that the cuts appear at the level of the leaves of the left combs in~$\projDown(\PT)$.
Had we considered instead the shuffle of the $n$-associahedron with the $m$-permutahedron (or equivalently, the shuffle of the $m$-permutahedron with the anti-$n$-associahedron), the (left or right) shadow map would be neither a join nor a meet semilattice morphism.
\end{remark}

%\begin{remark}
%\label{rem:shadowMapSemilatticeMapRefinement}
%When~$m \le 1$, the shadow map is a surjective meet semilattice morphism from the \mbox{$m$-painted} $n$-tree refinement meet semilattice to the $m$-lighted $n$-shade refinement meet semilattice.
%Indeed, the minimal element~$\projDown(\PT)$ in the shadow class of an $m$-painted $n$-tree~$\PT$ is obtained by contracting all edges between two nodes that are not on the right branch of~$\PT$.
%It fails when~$m \ge 2$ as edges between two consecutive cuts cannot be contracted.
%\end{remark}

%%%%%%%%%

\subsection{Singletons}
\label{subsec:singletons}

We now study the fibers of the shadow map which consist in a single $m$-painted $n$-tree.
They are analoguous to the classical singleton permutations used to construct associahedra, see \cref{rem:similarities}.

\begin{definition}
An \defn{$(m,n)$-singleton} is a binary $m$-painted $n$-tree which is alone in its shadow congruence class.
\end{definition}

\begin{proposition}
\label{prop:singletons}
The following conditions are equivalent for a binary $m$-painted $n$-tree~$\PT$:
\begin{enumerate}[(i)]
\item $\PT$ is an $(m,n)$-singleton,
\item each binary node of~$\PT$ lies on the right branch, or its parent lies on the right branch if it is below the last cut,
\item each tuple of the shadow~$\shadow(\PT)$ is reduced to a single~$1$, or either to a single~$1$ or a single~$2$ if it is below the last cut.
\end{enumerate}
\end{proposition}

\begin{proof}
Assume that~$\PT$ has a binary node~$i$ which is not on the right branch, and let~$j$ be the parent of~$i$.
If~$j$ is a unary node contained in a cut~$c$, then moving~$c$ past $i$ preserves the shadow of~$\PT$.
If~$j$ is a binary node not on the right branch, then rotating the edge~$i \to j$ preserves the shadow of~$\PT$.
If~$j$ is on the right branch but above a cut~$c$, then moving~$c$ past a node in the left branch of~$j$ preserves the shadow of~$\PT$.
Finally, if~$j$ is on the right branch and below all cuts, then the only possible rotation in the left branch of~$j$ modifies the shadow of~$\PT$.
This proves that~(i)~$\iff$~(ii).
Finally, (ii) clearly translates to~(iii) via the shadow map.
\end{proof}

\newcommand{\Fib}[1]{\textsf{Fib}_{#1}} % Fibonacci number
\begin{corollary}
The number of singletons of the $(m,n)$-shadow map is
\[
m! \sum_{k = 0}^{n} \binom{m+k-1}{k} \Fib{n-k+1} ,
\]
where~$\Fib{k}$ denote the $k$th Fibonacci number (defined by~$\Fib{0} = \Fib{1} = 1$ and $\Fib{k+2} = \Fib{k+1} + \Fib{k}$ for~$k \ge 0$, see \OEIS{A000045}).
See \cref{table:singletons} in \cref{subsec:tableSingletons} for the first few numbers.
\end{corollary}

\begin{proof}
To count the number of singletons, we can simply count the number of shadows of singletons.
From their description in \cref{prop:singletons}\,(iii), we obtain that the number of such shades with $k$ entries above the last cut is given by~$m! \binom{m+k-1}{k} \Fib{n-k+1}$.
Namely, choose the order of the cuts (hence $m!$ choices), insert $k$ tuples reduced to a single~$1$ before the last cut (hence $\binom{m+k-1}{k}$ choices), and finish with a sequence of tuples reduced to a single~$1$ or a single~$2$, whose total sum is~$n-k$ (hence $\Fib{n-k+1}$ choices).
\end{proof}

\begin{example}
When~$m = 0$, the number of singletons is the Fibonacci~$\Fib{n}$ (\OEIS{A000045}).
When~$m = 1$, the number of singletons is~$\Fib{n+2}-1$ (\OEIS{A000071}).
\end{example}

%%%%%%%%%%%%%%%%%%%%%%%%%%%%%%%%%%%%%%

\section{Multiplihedra and Hochschild polytopes}
\label{sec:polytopes}

In this section, we construct polyhedral fans and polytopes whose face lattices are isomorphic to the coarsening lattices on $m$-painted $n$-trees and $m$-lighted $n$-shades respectively.
We start with a brief recollection on polyhedral geometry (\cref{subsec:polyhedralGeometry}).
We then present the vertex and facet of the $(m,n)$-multiplihedron (\cref{subsec:multiplihedra}) and of the $(m,n)$-Hochschild polytope (\cref{subsec:freehedra}).
We conclude by gathering all necessary proofs on Hochschild polytopes (\cref{subsec:proofs}).

%%%%%%%%%

\subsection{Recollection on polyhedral geometry}
\label{subsec:polyhedralGeometry}

We start with a brief reminder on fans and polytopes, with a particular attention to deformed permutahedra.
We invite the reader familiar with these notions to jump directly to \cref{subsec:multiplihedra}.

%%%

\subsubsection{Fans and polytopes}
\label{subsec:fansPolytopes}

A (polyhedral) \defn{cone} is the positive span~$\R_{\ge0}\b{R}$ of a finite set~$\b{R}$ of vectors of~$\R^d$ or equivalently, the intersection of finitely many closed linear half-spaces of~$\R^d.$ 
The \defn{faces} of a cone are its intersections with its supporting hyperplanes. 
The \defn{rays} (resp.~\defn{facets}) are the faces of dimension~$1$ (resp.~ codimension~$1$).
A cone is \defn{simplicial} if its rays are linearly independent.
A (polyhedral) \defn{fan}~$\c{F}$ is a set of cones such that any face of a cone of~$\c{F}$ belongs to~$\c{F}$, and any two cones of~$\c{F}$ intersect along a face of both. 
A fan is \defn{essential} if the intersection of its cones is the origin, \defn{complete} if the union of its cones covers~$\R^d$, and \defn{simplicial} if all its cones are simplicial.

A \defn{polytope} is the convex hull of finitely many points of~$\R^d$ or equivalently, a bounded intersection of finitely many closed affine half-spaces of~$\R^d$.
The \defn{faces} of a polytope are its intersections with its supporting hyperplanes.
The \defn{vertices} (resp.~\defn{edges}, resp.~\defn{facets}) are the faces of dimension~$0$ (resp.~dimension~$1$, resp.~codimension~$1$).

The \defn{normal cone} of a face~$\polytope{F}$ of a polytope~$\polytope{P}$ is the cone generated by the normal vectors to the supporting hyperplanes of~$\polytope{P}$ containing~$\polytope{F}$.
Said differently, it is the cone of vectors~$\b{c}$ of~$\R^d$ such that the linear form~$\b{x} \mapsto \dotprod{\b{c}}{\b{x}}$ on~$\polytope{P}$ is maximized by all points of the face~$\polytope{F}$.
The \defn{normal fan} of~$\polytope{P}$ is the set of normal cones of all its faces.

The \defn{Minkowski sum} of two polytopes~$\polytope{P}, \polytope{Q} \subseteq \R^n$ is the polytope~$\polytope{P} + \polytope{Q} \eqdef \set{p+q}{p \in \polytope{P}, q \in \polytope{Q}}$.
The normal fan of~$\polytope{P} + \polytope{Q}$ is the common refinement of the normal fans of~$\polytope{P}$ and~$\polytope{Q}$.
We write~${\polytope{P} = \polytope{Q} - \polytope{R}}$ when~$\polytope{P} + \polytope{R} = \polytope{Q}$.

A \defn{deformation} of a full-dimensional polytope~$\polytope{P}$ is a polytope~$\polytope{Q}$ satisfying the following equivalent conditions:
\begin{itemize}
\item the normal fan of~$\polytope{Q}$ coarsens the normal fan of~$\polytope{P}$,
\item $\polytope{Q}$ is a \defn{weak Minkowski summand} of~$\polytope{P}$, \ie there exists a polytope~$\polytope{R}$ and a positive real number~$\lambda$ such that~$\lambda \polytope{P} = \polytope{Q} + \polytope{R}$
\item $\polytope{Q}$ can be obtained from~$\polytope{P}$ by gliding each facet in the direction of its normal vector without passing a vertex.
\end{itemize}

%%%

\subsubsection{The braid fan, the permutahedron, and its deformations}
\label{subsec:permutahedron}

We denote by~$(\b{e}_i)_{i \in [d]}$ the canonical basis of~$\R^d$ and we define~$\one_X \eqdef \sum_{i \in X} \b{e}_i$ for~$X \subseteq [d]$, and~$\one \eqdef \one_{[d]}$.
All our polytopal constructions will lie in the affine subspace~$\HH_d \eqdef \bigset{\b{x} \in \R^d}{\dotprod{\one}{\b{x}} = \sum_{i \in [d]} x_i = \binom{d+1}{2}}$, and their normal fans will lie in the vector subspace~$\one^\perp \eqdef \set{\b{x} \in \R^d}{\dotprod{\one}{\b{x}} = 0}$.

The \defn{braid arrangement} is the arrangement formed by the hyperplanes~$\set{\b{x} \in \one^\perp}{x_i = x_j}$ for all~${1 \le i < j \le d}$.
Its regions (\ie the closures of the connected components of the complement of the union of its hyperplanes) are the maximal cones of the \defn{braid fan}~$\c{B}_d$.
This fan has a \mbox{$k$-dimen}\-sional cone for each ordered partition of~$[d]$ into~$k+1$ parts.
In particular, it has a region for each permutation of~$[d]$, and a ray for each proper nonempty subset of~$[d]$.
(Note that we work in the subspace~$\one^\perp$ of~$\R^d$ so that the braid arrangement is essential and indeed has rays.)

The \defn{permutahedron}~$\Perm$ is the polytope defined equivalently as
\begin{itemize}
\item the convex hull of the points~$\sum_{i \in [d]} i \, \b{e}_{\sigma(i)}$ for all permutations~$\sigma$ of~$[d]$, see \cite{Schoute},
\item the intersection of the hyperplane~$\HH_d$ with the halfspaces~$\smash{\bigset{\b{x} \in \R^d}{\sum_{i \in I} x_i \ge \binom{|I|+1}{2}}}$ for all~${\varnothing \ne I \subsetneq [d]}$, see \cite{Rado}.
\end{itemize}
The braid fan~$\c{B}_d$ is the normal fan of the permutahedron~$\Perm$.
When oriented in the direction~${\b{\omega}_d \eqdef (d,\dots,1) - (1,\dots,d) = \sum_{i \in [d]} (d+1-2i) \, \b{e}_i}$, the skeleton of the permutahedron~$\Perm$ is isomorphic to the Hasse diagram of the classical weak order on permutations of~$[d]$.

%%%

\subsubsection{Deformed permutahedra}
\label{subsec:deformedPermutahedra}

A \defn{deformed permutahedron} (\aka polymatroid~\cite{Edmonds}, or generalized permutahedron~\cite{Postnikov, PostnikovReinerWilliams}) is a deformation of the permutahedron.
The normal fan of a deformed permutahedron is a collection of preposet cones~\cite{PostnikovReinerWilliams}.
The \defn{preposet cone} of a preposet~$\preccurlyeq$ on~$[d]$ is the cone~$\set{\b{x} \in \R^d}{x_i \le x_j \text{ if } i \preccurlyeq j}$.
For instance, the cones of the braid fan are precisely the preposet cones of the total preposets, \ie those where~$i \preccurlyeq j$ or~$j \preccurlyeq i$ (or both) for any~$i,j \in [d]$.

There are two standard parametrizations of the deformed permutahedra.
Namely, for a deformed permutahedron~$\polytope{P}$ in~$\R^d$, we define:
\begin{itemize}
\item its \defn{Minkowski coefficients}~$\big( \b{y}_I(\polytope{P}) \big)_{\varnothing \ne I \subseteq [d]}$ such that~$\polytope{P}$ is the Minkowski sum and difference~$\sum_{\varnothing \ne I \subseteq [d]} \b{y}_I(\polytope{P}) \, \triangle_I$, where $\simplex_I \eqdef \conv\set{\b{e}_i}{i \in I}$ is the face of the standard simplex~$\simplex_{[d]} \eqdef \conv\set{\b{e}_i}{i \in [d]}$ corresponding to~$I$,
\item its \defn{tight right hand sides}~$\big( \b{z}_J(\polytope{P}) \big)_{\varnothing \ne J \subseteq [d]}$ such that $\b{z}_J(\polytope{P}) \eqdef \min\set{\dotprod{\one_J}{\b{p}}}{\b{p} \in \b{P}}$.
\end{itemize}
As proved in~\cite{Postnikov, ArdilaBenedettiDoker}, these two parametrizations are related by boolean M\"obius inversion:
\[
\b{z}_J(\polytope{P})  = \sum_{I \subseteq J} \b{y}_I(\polytope{P}) 
\qquad\text{and}\qquad
\b{y}_I(\polytope{P})  = \sum_{J \subseteq I} (-1)^{|I \ssm J|} \, \b{z}_J(\polytope{P}).
\]

For instance, for the classical permutahedron~$\Perm$,
\begin{itemize}
\item its Minkowski coefficients are~$\b{y}_I \big( \Perm \big) = 1$ if~$|I| \le 2$ and~$0$ otherwise,
\item its tight right hand sides are~$\b{z}_J \big( \Perm \big) = \binom{|J|+1}{2}$.
\end{itemize}

As another illustration, recall that the \defn{associahedron}~$\Asso$ is the deformed permutahedron defined equivalently as
\begin{itemize}
\item the convex hull of the points~$\sum_{i \in [d]} \ell(T,i) \, r(T,i) \, \b{e}_i$ for all binary trees~$T$ with $d$~internal nodes, where $\ell(T,i)$ and~$r(T,i)$ respectively denote the numbers of leaves in the left and right subtrees of the $i$th node of~$T$ in infix labeling, see~\cite{Loday},
\item the intersection of the hyperplane~$\HH_d$ with the halfspaces~$\smash{\bigset{\b{x} \in \R^d}{\sum_{i \le \ell \le j} x_\ell \ge \binom{j-i+2}{2}}}$ for all~$1 \le i \le j \le d$, see~\cite{ShniderSternberg}.
\end{itemize}
Moreover,
\begin{itemize}
\item its Minkowski coefficients are~$\b{y}_I \big( \Asso \big) = 1$ if~$I$ is an interval of~$[d]$ and~$0$ otherwise, see~\cite{Postnikov},
\item its tight right hand sides are~$\b{z}_J \big( \Asso \big) = \binom{|J_1|+1}{2} + \dots + \binom{|J_k|+1}{2}$ where~$J = J_1 \cup \dots \cup J_k$ is the decomposition of~$J$ into maximal intervals of~$[d]$, see \cite{Lange}.
\end{itemize}

%%%%%%%%%

\subsection{Multiplihedra}
\label{subsec:multiplihedra}

We now consider the $(m,n)$-multiplihedron which realize the $m$-painted $n$-tree coarsening lattice.
These polytopes are illustrated in \cref{fig:multiplihedronFreehedronLabeledOriented13,fig:multiplihedronFreehedronLabeledOriented3,fig:multiplihedronFreehedronLabeledOriented4}.
Although they were previously constructed when~$m = 1$ in~\cite{Stasheff-HSpaces, SaneblidzeUmble-diagonals, Forcey-multiplihedra, ForceyLauveSottile, MauWoodward, ArdilaDoker}, we use here the construction of~\cite[Sect.~3]{ChapotonPilaud}.
This construction is just an example of the shuffle product on deformed permutahedra, introduced in~\cite[Sect.~2]{ChapotonPilaud}.
However, we do not really need the generality of this operation and define here the $(m,n)$-multiplihedron using its vertex and facet descriptions.

\begin{definition}
\label{def:verticesPaintedTrees}
Consider a binary $m$-painted $n$-tree~$\PT \eqdef (T, C, \mu)$.
We associate to~$\PT$ a point~$\b{a}(\PT)$ whose $p$th coordinate is
\begin{itemize}
\item if~$p \le m$, the number of binary nodes and cuts weakly below the cut labeled by~$p$,
\item if~$p \ge m+1$, the number of cuts below plus the product of the numbers of leaves in the left and right subtrees of the node of~$T$ labeled by~$p-m$ in inorder.
\end{itemize}
See \cref{fig:verticesPaintedTrees} for some examples.
\end{definition}

\begin{definition}
\label{def:inequalitiesPaintedTrees}
Consider the hyperplane~$\HH_{m+n}$ of~$\R^{m+n}$ defined by the equality
\[
\dotprod{\b{x}}{\b{1}_{[m+n]}} = \binom{m+n+1}{2}.
\]
Moreover, for each rank~$m+n-2$ $m$-painted $n$-tree~$\PT \eqdef (T, C, \mu)$, consider the halfspace~$\b{H}(\PT)$ of~$\R^{m+n}$ defined by the inequality
\[
\dotprod{\b{x}}{\one_{A \cup B}} \ge \binom{|A|+1}{2} + \binom{|B_1|+1}{2} + \dots + \binom{|B_k|+1}{2} + |A| \cdot |B|,
\]
where
\begin{itemize}
\item $A$ denotes the set of elements of~$[m]$ which label the cut of~$C$ not containing the root of~$T$ (hence, $A = \varnothing$ if~$C$ has only one cut, which contains the root of~$T$), 
\item $B \eqdef B_1 \cup \dots \cup B_k$ where~$B_1, \dots, B_k$ are the inorder labels shifted by~$m$ of the non-unary nodes of~$T$ distinct from the root of~$T$.
\end{itemize}
See \cref{fig:inequalitiesPaintedTrees} for some examples.
\end{definition}

\begin{proposition}[{\cite[Props.~116, 122, 123]{ChapotonPilaud}}]
\label{prop:VHDescriptionsPaintedTrees}
The $m$-painted $n$-tree coarsening lattice is isomorphic to the face lattice of the \defn{$(m,n)$-multiplihedron}~$\Multiplihedron$, defined equivalently as
\begin{enumerate}[(i)]
\item the convex hull of the vertices~$\b{a}(\PT)$ for all binary $m$-painted $n$-trees~$\PT$,
\item the intersection of the hyperplane~$\HH_{m+n}$ with the halfspaces~$\b{H}(\PT)$ for all rank~$m+n-2$ \mbox{$m$-painted} $n$-trees~$\PT$.
\end{enumerate}
\end{proposition}

\begin{proposition}[{\cite[Prop.~118]{ChapotonPilaud}}]
\label{prop:fanPaintedTrees}
The normal fan of the $(m,n)$-multiplihedron~$\Multiplihedron$ is the fan whose cones are the preposet cones of the preposets~$\preccurlyeq_{\PT}$ of all $m$-painted $n$-trees~$\PT$.
%The preposet cones of the preposets~$\preccurlyeq_{\PT}$ of all $m$-painted $n$-trees~$\PT$ form a fan, called the \defn{$m$-painted $n$-tree fan}.
\end{proposition}

\begin{proposition}[{\cite[Prop.~119]{ChapotonPilaud}}]
\label{prop:graphPaintedTrees}
The skeleton of the $(m,n)$-multiplihedron~$\Multiplihedron$ oriented in the direction~$\b{\omega}_{m+n} \eqdef (m+n, \dots, 1) - (1, \dots, m+n)$ is isomorphic to the right rotation digraph on binary $m$-painted $n$-trees.
\end{proposition}

%\begin{proposition}[{\cite[Prop.~124]{ChapotonPilaud}}]
%\label{prop:MinkowskiDescriptionsPaintedTrees}
%The $(m,n)$-multiplihedron~$\Multiplihedron$ is the Minkowski sum of the faces ${\simplex_I \eqdef \conv\set{\b{e}_i}{i \in I}}$ of the standard simplex~$\simplex_{[m+n]}$ corresponding to all subsets~${I \subseteq [m+n]}$ such that~$|I| \le 2$ and~$|I \cap [n]^{+m}| \le 1$, or~$I$ is a subinterval of~$[n]^{+m}$.
%\end{proposition}

\begin{remark}
\label{rem:YZCoordinatesPaintedTrees}
As observed in \cite[Prop.~124]{ChapotonPilaud}, it is straightforward to obtain the $\b{y}$ and $\b{z}$ parametrizations of the $(m,n)$-multiplihedron~$\Multiplihedron$.
Namely, for~$I \subseteq [m+n]$, we have
\[
\b{y}_I \big( \Multiplihedron \big) = 
\begin{cases}
	1 & \text{ if~$|I| \le 2$ and~$|I \cap [n]^{+m}| \le 1$, or~$I$ is a subinterval of~$[n]^{+m}$} \\
	0 & \text{ otherwise}
\end{cases}
\]
and
\[
\b{z}_J \big( \Multiplihedron \big) = \binom{|A|+1}{2} + \binom{|B_1|+1}{2} + \dots + \binom{|B_k|+1}{2} + |A| \cdot |B|,
\]
where~$A \eqdef J \cap [m]$ and~$B \eqdef B_1 \cup \dots \cup B_k$ is the coarsest interval decomposition of~$J \ssm [m]$.
\end{remark}

\begin{example}
When~$m \! = \! 0$, the $(0,n)$-multiplihedron is Loday's associahedron~\cite{Loday}.
When~$m \! = \! 1$, the $(1,n)$-multiplihedron is the classical multiplihedron alternatively constructed in~\cite{Forcey-multiplihedra, ArdilaDoker}.
\end{example}

%%%%%%%%%

\begin{figure}
	\centerline{\includegraphics[scale=.9]{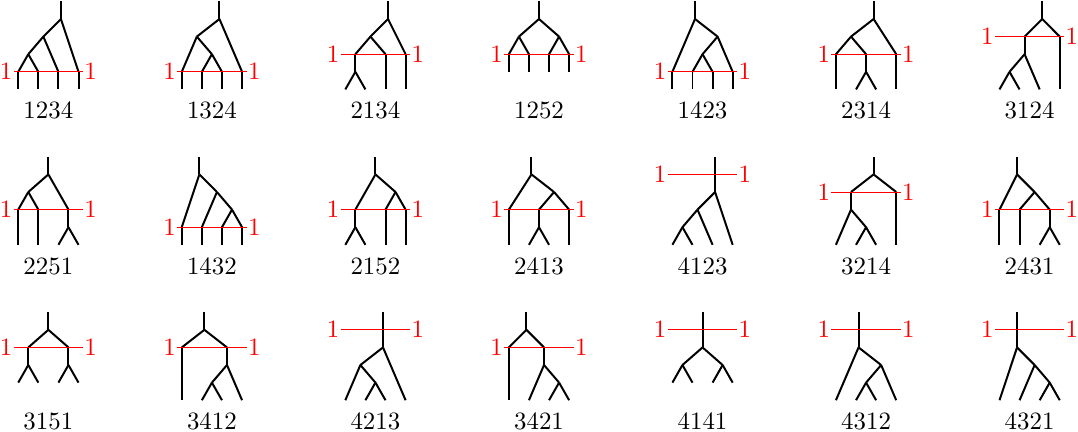}}
%	\centerline{
%	\begin{tabular}{c@{}c@{}c@{}c@{}c@{}c@{}c}
%		\paintedTree[1]{[[[[1[]][1[]]][1[]]][1[]]]}{1} &
%		\paintedTree[1]{[[[1[]][[1[]][1[]]]][1[]]]}{1} &
%		\paintedTree[1]{[[[1[[][]]][1[]]][1[]]]}{1} &
%		\paintedTree[1]{[[[1[]][1[]]][[1[]][1[]]]]}{1} &
%		\paintedTree[1]{[[1[]][[[1[]][1[]]][1[]]]]}{1} &
%		\paintedTree[1]{[[[1[]][1[[][]]]][1[]]]}{1} &
%		\paintedTree[1]{[[1[[[][]][]]][1[]]]}{1} \\
%		$1234$ & 
%		$1324$ & 
%		$2134$ & 
%		$1252$ & 
%		$1423$ & 
%		$2314$ & 
%		$3124$ \\
%		\paintedTree[1]{[[[1[]][1[]]][1[[][]]]]}{1} &
%		\paintedTree[1]{[[1[]][[1[]][[1[]][1[]]]]]}{1} &
%		\paintedTree[1]{[[1[[][]]][[1[]][1[]]]]}{1} &
%		\paintedTree[1]{[[1[]][[1[[][]]][1[]]]]}{1} &
%		\paintedTree[1]{[1[[[[][]][]][]]]}{1} &
%		\paintedTree[1]{[[1[[][[][]]]][1[]]]}{1} &
%		\paintedTree[1]{[[1[]][[1[]][1[[][]]]]]}{1} \\
%		$2251$ & 
%		$1432$ & 
%		$2152$ & 
%		$2413$ & 
%		$4123$ & 
%		$3214$ & 
%		$2431$ \\
%		\paintedTree[1]{[[1[[][]]][1[[][]]]]}{1} &
%		\paintedTree[1]{[[1[]][1[[[][]][]]]]}{1} &
%		\paintedTree[1]{[1[[[][[][]]][]]]}{1} &
%		\paintedTree[1]{[[1[]][1[[][[][]]]]]}{1} &
%		\paintedTree[1]{[1[[[][]][[][]]]]}{1} &
%		\paintedTree[1]{[1[[][[[][]][]]]]}{1} &
%		\paintedTree[1]{[1[[][[][[][]]]]]}{1} \\
%		$3151$ & 
%		$3412$ & 
%		$4213$ & 
%		$3421$ & 
%		$4141$ & 
%		$4312$ & 
%		$4321$
%	\end{tabular}
%	}
	\caption{Vertices of $\Multiplihedron[1][3]$.}
	\label{fig:verticesPaintedTrees}
\end{figure}

\begin{figure}
	\vspace{.5cm}
	\centerline{\includegraphics[scale=.9]{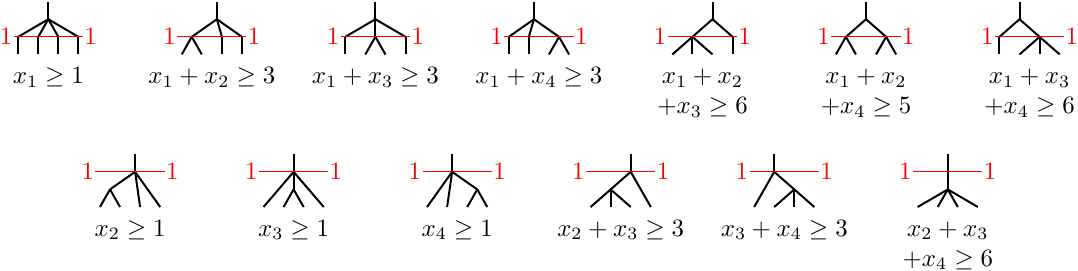}}
%	\centerline{
%	\begin{tabular}{c@{}c@{}c@{}c@{}c@{}c@{}c}
%		\paintedTree[1]{[[1[]][1[]][1[]][1[]]]}{1} &
%		\paintedTree[1]{[[1[][]][1[]][1[]]]}{1} &
%		\paintedTree[1]{[[1[]][1[][]][1[]]]}{1} &
%		\paintedTree[1]{[[1[]][1[]][1[][]]]}{1} &
%		\paintedTree[1]{[[1[][][]][1[]]]}{1} &
%		\paintedTree[1]{[[1[][]][1[][]]]}{1} &
%		\paintedTree[1]{[[1[]][1[][][]]]}{1} \\
%		$x_1 \ge 1$ & 
%		$x_1 + x_2 \ge 3$ & 
%		$x_1 + x_3 \ge 3$ & 
%		$x_1 + x_4 \ge 3$ & 
%		$x_1 + x_2$ & 
%		$x_1 + x_2$ & 
%		$x_1 + x_3$ \\
%		&
%		&
%		&
%		&
%		$+ x_3 \ge 6$ & 
%		$+ x_4 \ge 5$ &
%		$+ x_4 \ge 6$
%	\end{tabular}
%	}
%	\centerline{
%	\begin{tabular}{c@{}c@{}c@{}c@{}c@{}c}
%		\paintedTree[1]{[1[[][]][][]]}{1} &
%		\paintedTree[1]{[1[][[][]][]]}{1} &
%		\paintedTree[1]{[1[][][[][]]]}{1} &
%		\paintedTree[1]{[1[[][][]][]]}{1} &
%		\paintedTree[1]{[1[][[][][]]]}{1} &
%		\paintedTree[1]{[1[[][][][]]]}{1} \\
%		$x_2 \ge 1$ & 
%		$x_3 \ge 1$ & 
%		$x_4 \ge 1$ & 
%		$x_2 + x_3 \ge 3$ &
%		$x_3 + x_4 \ge 3$ &
%		$x_2 + x_3$ \\
%		&
%		&
%		&
%		&
%		&
%		$+ x_4 \ge 6$
%	\end{tabular}
%	}
	\caption{Facet defining inequalities of $\Multiplihedron[1][3]$.}
	\label{fig:inequalitiesPaintedTrees}
\end{figure}

\begin{figure}
	\vspace{.5cm}
	\centerline{\includegraphics[scale=.9]{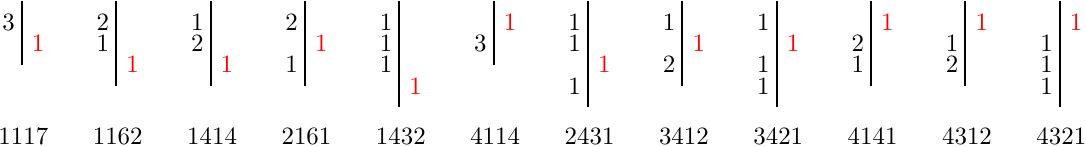}}
%	\centerline{
%	\begin{tabular}{c@{\hspace{-.2cm}}c@{\hspace{-.2cm}}c@{\hspace{-.2cm}}c@{\hspace{-.2cm}}c@{\hspace{-.2cm}}c@{\hspace{-.2cm}}c@{\hspace{-.2cm}}c@{\hspace{-.2cm}}c@{\hspace{-.2cm}}c@{\hspace{-.2cm}}c@{\hspace{-.2cm}}c}
%		\lightedShade[1]{[3 [, tier=1[]]]}{1} &
%		\lightedShade[1]{[2 [1 [, tier=1[]]]]}{1} &
%		\lightedShade[1]{[1 [2 [, tier=1[]]]]}{1} &
%		\lightedShade[1]{[2 [, tier=1[1 []]]]}{1} &
%		\lightedShade[1]{[1 [1 [1 [, tier=1[]]]]]}{1} &
%		\lightedShade[1]{[, tier=1[3 []]]}{1} &
%		\lightedShade[1]{[1 [1 [, tier=1[1 []]]]]}{1} &
%		\lightedShade[1]{[1 [, tier=1[2 []]]]}{1} &
%		\lightedShade[1]{[1 [, tier=1[1 [1 []]]]]}{1} &
%		\lightedShade[1]{[, tier=1[2 [1 []]]]}{1} &
%		\lightedShade[1]{[, tier=1[1 [2 []]]]}{1} &
%		\lightedShade[1]{[, tier=1[1 [1 [1 []]]]]}{1}
%		\\
%		$1117$ &
%		$1162$ &
%		$1414$ &
%		$2161$ &
%		$1432$ &
%		$4114$ &
%		$2431$ &
%		$3412$ &
%		$3421$ &
%		$4141$ &
%		$4312$ &
%		$4321$
%	\end{tabular}
%	}
	\caption{Vertices of $\HP[1][3]$.}
	\label{fig:verticesLightedShades}
\end{figure}

\begin{figure}
	\vspace{.5cm}
	\centerline{\includegraphics[scale=.9]{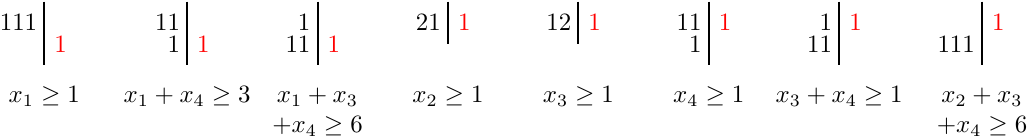}}
%	\centerline{
%	\begin{tabular}{c@{}c@{}c@{}c@{}c@{}c@{}c@{}c}
%		\lightedShade[1]{[111 [, tier=1 []]]}{1} &
%		\lightedShade[1]{[11 [1, tier=1 []]]}{1} &
%		\lightedShade[1]{[1 [11, tier=1 []]]}{1} &
%		\lightedShade[1]{[21, tier=1 []]}{1} &
%		\lightedShade[1]{[12, tier=1 []]}{1} &
%		\lightedShade[1]{[11, tier=1 [1 []]]}{1} &
%		\lightedShade[1]{[1, tier=1 [11 []]]}{1} &
%		\lightedShade[1]{[, tier=1 [111 []]]}{1} \\
%		$x_1 \ge 1$ & 
%		$x_1 + x_4 \ge 3$ & 
%		$x_1 + x_3$ &
%		$x_2 \ge 1$ & 
%		$x_3 \ge 1$ & 
%		$x_4 \ge 1$ & 
%		$x_3 + x_4 \ge 1$ &
%		$x_2 + x_3$ \\
%		&
%		&
%		$+ x_4 \ge 6$ &
%		&
%		&
%		&
%		&
%		$+ x_4 \ge 6$
%	\end{tabular}
%	}
	\caption{Facet defining inequalities of $\HP[1][3]$.}
	\label{fig:inequalitiesLightedShades}
\end{figure}

%%%%%%%%%

%\begin{figure}
%	\centerline{\includegraphics[scale=.6]{multiplihedronFreehedronLabeled3}}}
%%	\centerline{\input{multiplihedronFreehedronLabeled3}}
%	\caption{The multiplihedron~$\Multiplihedron$ (top) and the Hochschild polytope~$\HP$ (bottom) for $(m,n) = (0,3)$, $(1,2)$, $(2,1)$, and~$(3,0)$ (left to right).}
%	\label{fig:multiplihedronFreehedronLabeled3}
%\end{figure}

\afterpage{
\begin{figure}[p]
	\centerline{\includegraphics[scale=1.2]{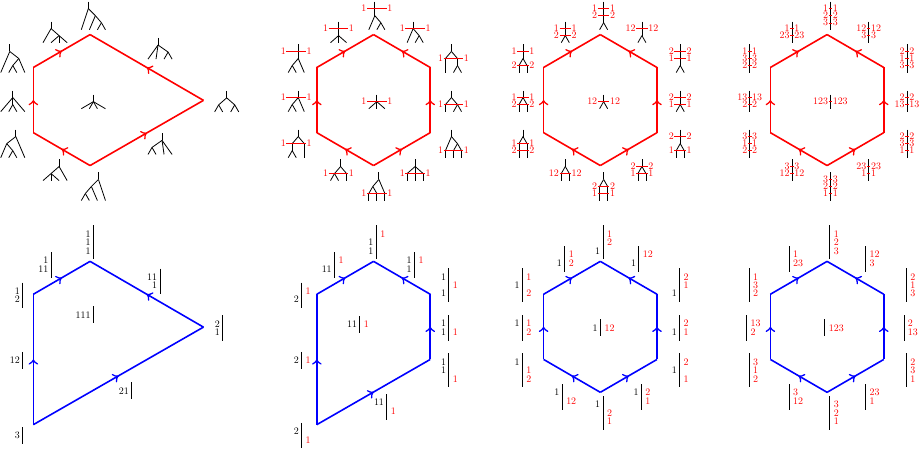}}
%	\centerline{\input{multiplihedronFreehedronLabeledOriented3}}
	\caption{The multiplihedron~$\Multiplihedron$ (top) and the Hochschild polytope~$\HP$ (bottom) for $(m,n) = (0,3)$, $(1,2)$, $(2,1)$, and~$(3,0)$ (left to right).}
	\label{fig:multiplihedronFreehedronLabeledOriented3}
\end{figure}
}

%\begin{figure}
%	\centerline{\includegraphics[scale=.4]{multiplihedronFreehedronLabeled4}}
%%	\centerline{\scalebox{.3}{\input{multiplihedronFreehedronLabeled4}}}
%	\caption{The multiplihedron~$\Multiplihedron$ (top) and the Hochschild polytope~$\HP$ (bottom) for $(m,n) = (0,4)$, $(1,3)$, $(2,2)$, $(3,1)$, and~$(4,0)$ (left to right).}
%	\label{fig:multiplihedronFreehedronLabeled4}
%\end{figure}

\afterpage{
\begin{figure}[p]
	\vspace{.3cm}
	\centerline{\includegraphics[scale=.38]{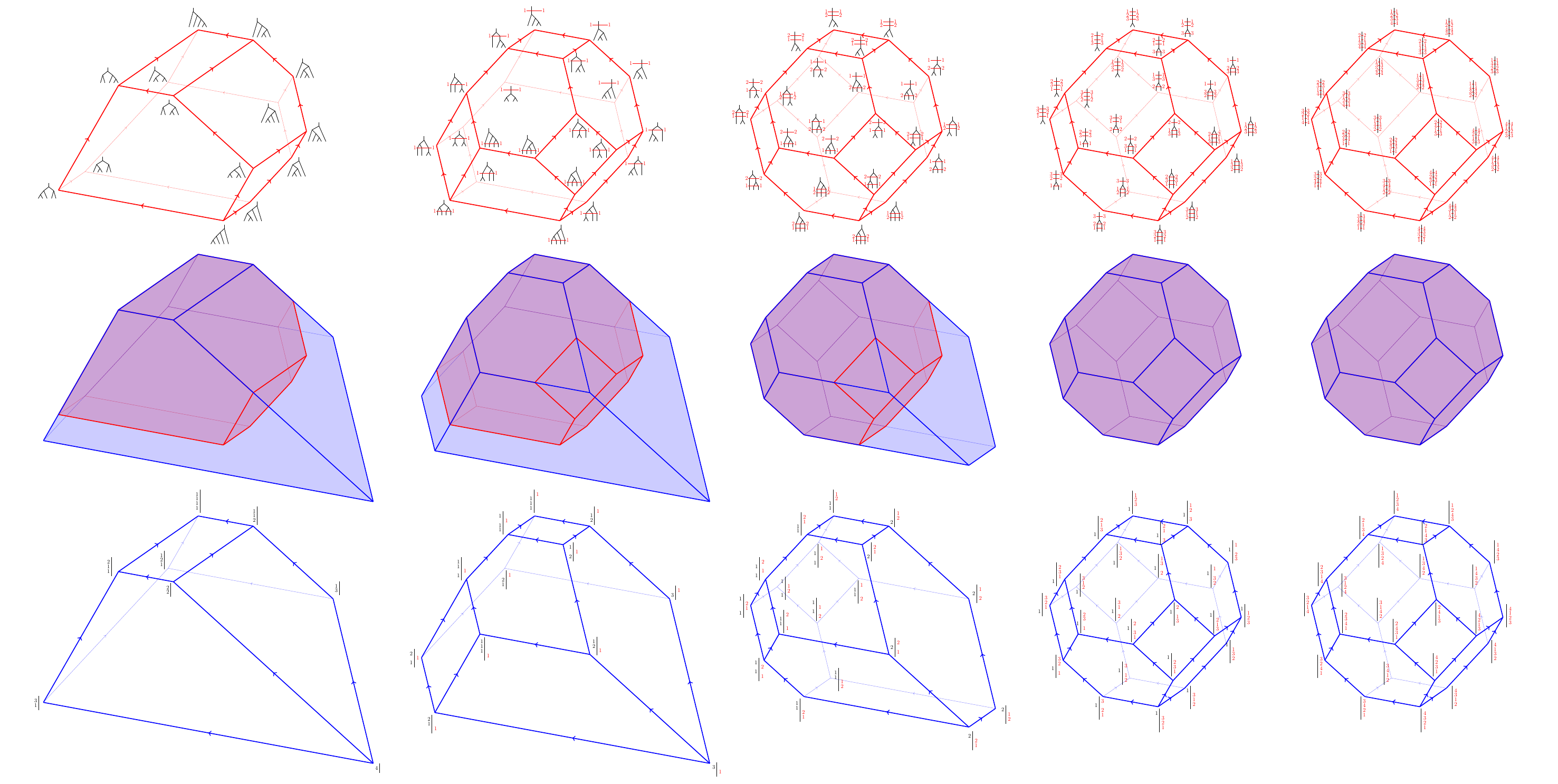}}
%	\centerline{\scalebox{.3}{\input{multiplihedronFreehedronLabeledOriented4}}}
	\caption{The multiplihedron~$\Multiplihedron$ (top) and the Hochschild polytope~$\HP$ (bottom) for $(m,n) = (0,4)$, $(1,3)$, $(2,2)$, $(3,1)$, and~$(4,0)$ (left to right).}
	\label{fig:multiplihedronFreehedronLabeledOriented4}
\end{figure}
}

\subsection{Hochschild polytopes}
\label{subsec:freehedra}

We now construct the $(m,n)$-Hochschild polytope which realize the $m$-lighted $n$-shade coarsening lattice.
These polytopes are illustrated in \cref{fig:multiplihedronFreehedronLabeledOriented13,fig:multiplihedronFreehedronLabeledOriented3,fig:multiplihedronFreehedronLabeledOriented4}.
Recall that we denote by~$ps(x)$ the preceeding sum of an entry~$x$ in an $m$-lighted $n$-shade (see \cref{def:preorderLightedShade}).

\begin{definition}
\label{def:verticesLightedShades}
Consider a unary $m$-lighted $n$-shade~$\LS \eqdef (S, C, \mu)$ and denote by~$s_1, s_2, \dots, s_r$ the values of the singleton tuples of~$S$ from left to right (\ie downwards in the pictures).
We associate to~$\LS$ a point~$\b{a}(\LS)$ whose $p$th coordinate is
\begin{itemize}
\item if $p \le m$, then the number of cuts plus the sum of the entries~$s_i$ which are weakly below the cut labeled~$p$,
\item if there is~$j \in [r]$ such that~$p = ps(s_j)$, then~$1 + s_j \big( m + n - p + c_p \big) + \binom{s_j}{2}$ where~$c_p$ is the number of cuts below $s_j$,
\item $1$ otherwise.
\end{itemize}
See \cref{fig:verticesLightedShades} for some examples.
\end{definition}

\begin{definition}
\label{def:inequalitiesLightedShades}
We still denote by~$\HH_{m+n}$ the hyperplane of~$\R^{m+n}$ defined by the equality
\[
\dotprod{\b{x}}{\b{1}_{[m+n]}} = \binom{m+n+1}{2}.
\]
Moreover, for each rank~$m+n-2$ $m$-lighted $n$-shade~$\LS \eqdef (S, C, \mu)$, consider the halfspace~$\b{H}(\LS)$ of~$\R^{m+n}$ defined by the inequality
\[
\dotprod{\b{x}}{\one_{A \cup B}} \ge \binom{|A|+|B|+1}{2},
\]
where
\begin{itemize}
\item $A$ denotes the set of elements of~$[m]$ which label the cut of~$C$ not containing the first tuple of~$S$ (hence, $A = \varnothing$ if~$C$ has only one cut, which contains the first tuple of~$S$), 
\item $B = \{m+q\}$ if $S$ is a single tuple with the $2$ in position~$q$, and~$B = \{m+q+1, \dots, m+n\}$ if~$S = (s_1, s_2)$ is a pair of tuples with~$|s_1| = q$.
\end{itemize}
See \cref{fig:inequalitiesLightedShades} for some examples.
\end{definition}

\begin{remark}
\label{rem:removahedron}
The inequalities of \cref{def:inequalitiesLightedShades} form a subset of the inequalities of \cref{def:inequalitiesPaintedTrees}.
\end{remark}

We postpone the proofs of the next three statements to \cref{subsec:proofs}.

\begin{proposition}
\label{prop:VHDescriptionsLightedShades}
The $m$-lighted $n$-shade coarsening lattice is isomorphic to the face lattice of the \defn{$(m,n)$-Hochschild polytope}~$\HP$, defined equivalently as
\begin{enumerate}[(i)]
\item the convex hull of the vertices~$\b{a}(\LS)$ for all unary $m$-lighted $n$-shades~$\LS$,
\item the intersection of the hyperplane~$\HH_{m+n}$ with the halfspaces~$\b{H}(\LS)$ for all rank~$m+n-2$ $m$-lighted $n$-shades~$\LS$.
\end{enumerate}
\end{proposition}

\begin{proposition}
\label{prop:fanLightedShades}
The normal fan of the $(m,n)$-Hochschild polytope~$\HP$ is the fan whose cones are the preposet cones of the preposets~$\preccurlyeq_{\LS}$ of all $m$-lighted $n$-shades~$\LS$.
%The preposet cones of the preposets~$\preccurlyeq_{\LS}$ of all $m$-lighted $n$-shades~$\LS$ form a fan, called the \defn{$m$-lighted $n$-shade fan}.
\end{proposition}

\begin{proposition}
\label{prop:graphLightedShades}
The skeleton of the $(m,n)$-Hochschild polytope~$\HP$ oriented in the direction~$\b{\omega}_{m+n} \eqdef (m+n, \dots, 1) - (1, \dots, m+n)$ is isomorphic to the right rotation digraph on unary $m$-lighted $n$-shades.
\end{proposition}

\begin{remark}
It follows from \cref{rem:rotationGraphLightedShadesRegular} that the $(m,n)$-Hochschild polytope is simple and the $m$-lighted $n$-shade fan is simplicial.
This will simplify our proofs in \cref{subsec:proofs}.
\end{remark}

\begin{remark}
\label{rem:YZCoordinatesLightedShades}
As in \cref{rem:YZCoordinatesPaintedTrees}, one can compute the $\b{y}$ and $\b{z}$ parametrizations of the \mbox{$(m,n)$-Hoch}\-schild polytope~$\HP$.
Namely, for~$I \subseteq [m+n]$, we have
\[
\b{y}_I \big( \HP \big) = 
\begin{cases}
	1 & \text{ if~$|I| = 1$, or~$|I| = 2$ and~$I \subseteq [m]$,} \\
	   & \text{ or~$I = \{i, m+j, m+j+1, \dots, m+n\}$ for some~$i \in [m]$ and~$j \in [n]$} \\
	n-j & \text{ if~$I = \{m+j, m+j+1, \dots, m+n\}$ for some~$j \in [n]$} \\
	0 & \text{ otherwise}
\end{cases}
\]
and
\[
\b{z}_J \big( \HP \big) = \binom{|A|+|C|+1}{2} + |B|,
\]
where~$A \eqdef J \cap [m]$, and~$B \cup C \eqdef J \ssm [m]$ such that~$C$ is the largest interval of~$J \ssm [m]$ containing~$m+n$.
\end{remark}

\begin{remark}
\label{rem:similarities}
As mentioned in the introduction, there are deep similarities between the behaviors~of
\begin{itemize}
\item the permutahedron~$\Perm$ and the associahedron~$\Asso$,
\item the multiplihedron~$\Multiplihedron$ and the Hochschild polytope~$\HP$.
\end{itemize}
We conclude with a few comments on the behavior of the latter for the reader familiar with the behavior of the former:
\begin{itemize}
\item As observed in \cref{rem:removahedron}, the $(m,n)$-Hochschild polytope~$\HP$ can be obtained by deleting inequalities in the facet description of the $(m,n)$-multiplihedron~$\Multiplihedron$.
\item The common facet defining inequalities of~$\Multiplihedron$ and~$\HP$ are precisely those that contain a common vertex of~$\Multiplihedron$ and~$\HP$ (the singletons of \cref{subsec:singletons}).
\item In contrast, the vertex barycenters of the $(m,n)$-multiplihedron~$\Multiplihedron$ and of the \mbox{$(m,n)$-Hochschild} polytope~$\HP$ do not coincide.
\item When~$m = 0$, the $(0,n)$-Hochschild polytope~$\HP[0][n]$ is a skew cube distinct from the parallelepiped obtained by considering the canopy congruence on binary trees (which is a lattice congruence, in contrast to the shadow meet semilattice congruence).
\end{itemize}
\end{remark}

\begin{example}
\label{exm:badFreehedron}
When~$m = 0$, the $(0,n)$-Hochschild polytope is a skew cube.
Note that it is distinct from the parallelotope~$\sum_{i \in [n-1]} [\b{e}_i, \b{e}_{i+1}]$.
When~$m = 1$, the $(1,n)$-Hochschild polytope gives a realization of the Hochschild lattice~\cite{Chapoton-Dyck, Combe, Muhle}.
Note that the unoriented rotation graph on $1$-lighted $n$-shades was already known to be isomorphic to the unoriented skeleton of a deformed permutahedron called \defn{freehedron} and obtained as a truncation of the standard simplex~\cite{Saneblidze}, or more precisely as the Minkowski sum~$\sum_{i \in [n]} \simplex_{\{1, \dots, i\}} + \sum_{i \in [n]} \simplex_{\{i, \dots, n\}}$ of the faces of the standard simplex corresponding to initial and final intervals, see \cref{fig:badFreehedron}.
However, orienting the skeleton of the freehedron in direction~$\b{\omega}_{m+n}$, we obtain a poset different from the Hochschild lattice, and which is not even a lattice.
Indeed, in \cref{fig:badFreehedron}\,(left) the two blue vertices have no join while the two red vertices have no meet.
In fact, the Hasse diagram of the Hochschild lattice cannot be obtained as a Morse orientation given by a linear functional on the freehedron.
Finally, observe that the freehedron cannot be obtained by removing inequalities in the facet description of the permutahedron or of the multiplihedron.
See \cref{fig:badFreehedron}\,(middle and right) where the resulting removahedra have the wrong combinatorics (look at the $4$-valent vertex on the right of the polytopes).
\begin{figure}
	\centerline{
		\includegraphics[scale=.6]{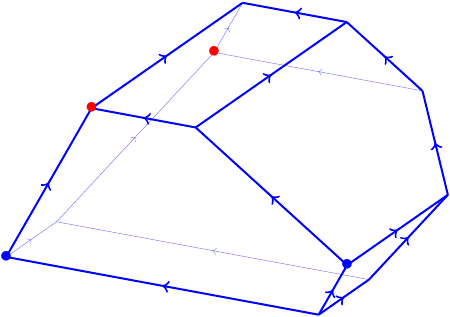}
		\quad
		\includegraphics[scale=.6]{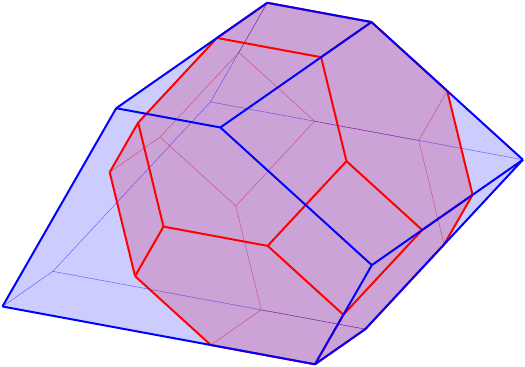}
		\includegraphics[scale=.6]{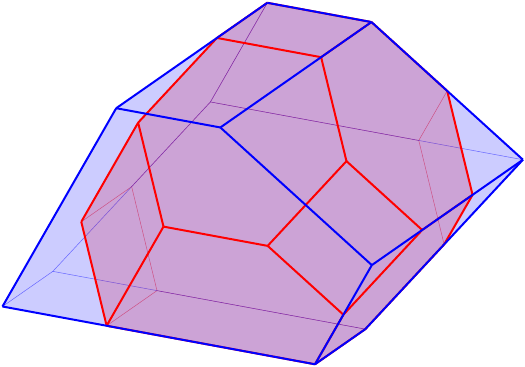}
	}
%	\centerline{
%		\input{figures/badFreehedronNotLattice}
%		\input{figures/badFreehedronNotRemovahedronPermutahedron}
%		\input{figures/badFreehedronNotRemovahedronMultiplihedron}
%	}
	\caption{The freehedron obtained as Minkowski sum of the faces of the standard simplex corresponding to initial or final intervals (left), and failed attempts to obtain it as a removahedron of the permutahedron (middle) or of the multiplihedron (right).}
	\label{fig:badFreehedron}
\end{figure}
\end{example}

%%%%%%%%%

\pagebreak
\subsection{Proof of \cref{prop:VHDescriptionsLightedShades,,prop:fanLightedShades,,prop:graphLightedShades}}
\label{subsec:proofs}

Our proof strategy follows that of~\cite[Sect.~4]{HohlwegLangeThomas}.
First, we will prove that the collection of cones described in \cref{prop:fanLightedShades} indeed defines a fan.

\begin{proposition}
\label{prop:fanLightedShades2}
The preposet cones of the preposets~$\preccurlyeq_{\LS}$ for all $m$-lighted $n$-shades~$\LS$ define a complete simplicial fan realizing the $m$-lighted $n$-shade coarsening lattice.
\end{proposition}

\begin{proof}
By \cref{def:preorderLightedShade}, the Hasse diagram of the preposets~$\preccurlyeq_{\LS}$ of each $m$-lighted $n$-shade~$\LS$ is a tree, so that the corresponding preposet cone is simplicial.
Moreover, contracting any edge in this tree gives rise to the Hasse diagram of the preposet~$\preccurlyeq_{\LS'}$ of an $m$-lighted $n$-shade~$\LS'$ refined by~$\LS$, so that this collection of cones is closed by faces.
Finally, by \cref{prop:shadowPreposets}, the shadow map~$\shadow$ sends any $m$-painted $n$-tree~$\PT$ to the coarsening maximal $m$-lighted $n$-shade such that~${\preccurlyeq_{\PT}} \supseteq {\preccurlyeq_{\LS}}$ , hence such that the preposet cone of~$\preccurlyeq_{\PT}$ is contained in the preposet cone of~$\preccurlyeq_{\LS}$.
Since the preposet cones of the preposets~$\preccurlyeq_{\PT}$ for all $m$-painted $n$-trees~$\PT$ form a complete fan~$\c{F}$, we conclude that the preposet cones of the preposets~$\preccurlyeq_{\LS}$ for all $m$-lighted $n$-shades~$\LS$ also form a complete fan refined by~$\c{F}$.
\end{proof}

Next, we apply the following characterization to realize a complete simplicial fan as the normal fan of a convex polytope. A proof of this statement can be found \eg in~\cite[Theorem~4.1]{HohlwegLangeThomas}.

\begin{theorem}[\protect{\cite[Thm~4.1]{HohlwegLangeThomas}}]
\label{thm:HohlwegLangeThomas}
Consider a complete simplicial fan~$\c{F}$ in~$\R^d$, and choose
\begin{itemize}
\item a point~$\b{a}(\polytope{C})$ for each maximal cone~$\polytope{C}$ of~$\c{F}$ (not necessarily in~$\polytope{C}$),
\item a half-space~$\b{H}(\b{\rho})$ of~$\R^d$ orthogonal to~$\b{\rho}$ and containing the origin for each ray~$\b{\rho}$ of~$\c{F}$,
\end{itemize}
such that $\b{a}(\polytope{C})$ belongs to the hyperplane defining~$\b{H}^=(\b{\rho})$ when~$\b{\rho} \in \polytope{C}$.
Then the following assertions are equivalent:
\begin{itemize}
\item for any two adjacent maximal cones~${\polytope{C}, \polytope{C}'}$~of~$\c{F}$, the vector~$\b{a}(\polytope{C}') - \b{a}(\polytope{C})$ points from~$\polytope{C}$ to~$\polytope{C}'$, 
(meaning that $\dotprod{\b{a}(\polytope{C}') - \b{a}(\polytope{C})}{\b{v}'-\b{v}}$ for any~$\b{v} \in \polytope{C}$ and~$\b{v}' \in \polytope{C}'$),
\item the polytopes
\[
\conv\set{\b{a}(\polytope{C})}{\polytope{C} \text{ maximal cone of } \c{F}}
\qquad\text{ and }\qquad
\bigcap_{\b{\rho} \text{ ray of } \c{F}} \b{H}(\b{\rho})
\]
coincide and their normal fan is~$\c{F}$.
\end{itemize}
\end{theorem}

In the next two lemmas, we check the conditions of application of \cref{thm:HohlwegLangeThomas}.
%To apply this theorem, we start by checking that our point~$\b{a}(\LS)$ of \cref{def:verticesPaintedTrees} is indeed the intersection of the hyperplanes defining the half-spaces~$\b{H}(\LS')$ of \cref{def:inequalitiesPaintedTrees} for all~$\LS'$ refined~by~$\LS$.

\begin{lemma}
\label{lem:intersectionPoint}
For any $m$-lighted $n$-shades~$\LS$ and~$\LS'$, of rank~$0$ and~$m+n-2$ respectively, such that~$\preccurlyeq_{\LS}$ refines~$\preccurlyeq_{\LS'}$, the point~$\b{a}(\LS)$ belongs to the hyperplane defining~$\b{H}(\LS')$.
\end{lemma}

\begin{proof}
Denote by~$s_1, \dots, s_r$ the values of the singleton tuples of~$\LS$.
We distinguish two cases:
\begin{itemize}
\item Assume first that~$\LS'$ contains a single tuple with the $2$ in position~$q$, so that~$A = \varnothing$ and~$B = \{m+q\}$ in \cref{def:inequalitiesPaintedTrees}.
Since~$\LS$ refines~$\LS'$, there is no~$j$ so that~$m+q = ps(s_j)$, so that~$\b{a}(\LS)_{m+q} = 1$ in \cref{def:verticesPaintedTrees}.
We conclude that
\[
\dotprod{\b{a}(\LS)}{\one_{A \cup B}} = \b{a}(\LS)_{m+q} = 1 = \binom{|A|+|B|+1}{2}.
\]
\item Assume now that~$\LS'$ is a pair of tuples~$(s'_1,s'_2)$ with~$|s'_1| = q$, so that~$A \subseteq [m]$ are the labels of the cut containing~$s'_2$, and~$B = \{m+q+1, \dots, m+n\}$ in \cref{def:inequalitiesPaintedTrees}.
Since~$\LS$ refines~$\LS'$, there is~$j$ such that~$q = ps(s_j)$.
We conclude that
\begin{align*}
\dotprod{\b{a}(\LS)}{\one_{A \cup B}} 
& = \binom{|A|+1}{2} + |A||B| + \sum_{i = j+1}^r \big( s_i - 1+ s_i \big( n-ps(s_i) \big) + \textstyle\binom{s_i}{2} \big) \\
& = \binom{|A|+1}{2} + |A||B| + \binom{|B|+1}{2} = \binom{|A|+|B|+1}{2}.
\qedhere
\end{align*}
\end{itemize}
\end{proof}

We now check that for a rotation sending~$\LS$ to~$\LS'$, the direction between the two points~$\b{a}(\LS)$ and~$\b{a}(\LS')$ of \cref{def:verticesPaintedTrees} points from the poset cone~$\preccurlyeq_{\LS}$ to the poset cone~$\preccurlyeq_{\LS'}$ of \cref{def:preorderLightedShade}.

\begin{lemma}
\label{lem:goodDirection}
For any unary $m$-lighted $n$-shades~$\LS$ and~$\LS'$ related by a rotation, the vector~${\b{a}(\LS') - \b{a}(\LS)}$ points from the poset cone~$\preccurlyeq_{\LS}$ to the poset cone~$\preccurlyeq_{\LS'}$.
\end{lemma}

\begin{proof}
We distinguish three cases according to~\cref{def:rotationLightedShades}.
Namely, if we obtain~$\LS'$ from~$\LS$ by:
\begin{enumerate}[(i)]
\item replacing a singleton~$(r)$ by two singletons~$(s), (t)$ with~$r = s + t$, then
\[
\b{a}(\LS') - \b{a}(\LS) = \big( s(m+n-p+t+c_p) + \textstyle\binom{s}{2} \big)(\b{e}_{p-t} - \b{e}_p),
\]
and we have~$p-t \prec_{\LS} p$ while~$p \prec_{\LS'} p-t$, where~$p \eqdef ps(r)$ is the preceeding sum of~$r$ in~$\LS$.
\item exchanging a singleton~$(s)$ with a cut~$c$ (with~$(s)$ above~$c$ in~$\LS$), then~$\b{a}(\LS') - \b{a}(\LS) = \b{e}_c - \b{e}_p$ and we have~$c \prec_{\LS} p$ while~$p \prec_{\LS'} c$, where~$p \eqdef ps(s)$.
\item exchanging the labels of two consecutive cuts~$c,c'$ with no singleton in between them (with~$c$ above~$c'$ in~$\LS$), then ${\b{a}(\LS') - \b{a}(\LS) = \b{e}_{c'} - \b{e}_{c}}$ and we have~$c' \prec_{\LS} c$ while~$c \prec_{\LS'} c'$.
\end{enumerate}
In all cases, the vector ${\b{a}(\LS') - \b{a}(\LS)}$ points from the poset cone~$\preccurlyeq_{\LS}$ to the poset cone~$\preccurlyeq_{\LS'}$.
\end{proof}

\begin{proof}[Proof of \cref{prop:VHDescriptionsLightedShades,,prop:fanLightedShades,,prop:graphLightedShades}]
We have seen in \cref{prop:fanLightedShades2} that the preposet cones of the preposets~$\preccurlyeq_{\LS}$ for all $m$-lighted $n$-shades~$\LS$ define a complete simplicial fan.
By \cref{thm:HohlwegLangeThomas}, whose conditions of application are checked in \cref{lem:intersectionPoint,lem:goodDirection}, we thus obtain \cref{prop:VHDescriptionsLightedShades,prop:fanLightedShades}.
Finally, \cref{prop:graphLightedShades} is a direct consequence of \cref{lem:goodDirection}, since~$\dotprod{\b{a}(\LS') - \b{a}(\LS)}{\b{\omega}_{m+n}} > 0$ for~$\LS$ and~$\LS'$ related by a right rotation.
\end{proof}

%%%%%%%%%%%%%%%%%%%%%%%%%%%%%%%%%%%%%%

\section{Cubic realizations}
\label{sec:cubicRealizations}

In this section we give an alternative description of the $m$-painted $n$-tree and $m$-lighted $n$-shade rotation lattices, generalizing the triword description of the Hochschild lattice~\cite{Chapoton-Dyck, Combe, Muhle}. 
We also construct the cubic subdivisions realizing the face poset of the \mbox{$(m,n)$-multipli}\-hedron and of the $(m,n)$-Hochschild polytope, generalizing the original construction of \cite{Saneblidze, RiveraSaneblidze}.
We first fix our conventions and give examples of cubic realizations (\cref{subsec:cubicRealizations}), then recall the cubic $(m,n)$-multiplihedron (\cref{subsec:cubicMultiplihedron}) and finally construct the cubic $(m,n)$-Hochschild polytope (\cref{subsec:cubicHochschildPolytope}).

%%%%%%%%%

\subsection{Cubic realizations of posets}
\label{subsec:cubicRealizations}

We first propose formal definitions of two types of cubic realizations of posets.
\cref{def:subcubePoset} is the cubic analogue of the face lattice of a polytope while \cref{def:cubicRealization} is the cubic analogue of the oriented skeleton of a polytope.
These definitions are illustrated in \cref{exm:LehmerCode,exm:bracketVector} with the permutahedron and the associahedron.

\begin{definition}
\label{def:subcubePoset}
We call \defn{cube} any axis aligned parallelepiped in~$\R^d$.
If~$\b{x}, \b{y} \in \R^d$ are such that~$x_i \le y_i$ for all~$i \in [d]$, we denote by~$\cube(\b{x},\b{y})$ the cube~$\prod_{i \in [d]} [x_i,y_i]$.
A \defn{subcube} of~$C$ is a cube included in~$C$ whose vertices all lie on the boundary of~$C$.
A \defn{cubic subdivision} of~$C$ is a collection~$\c{D}$ of subcubes of~$C$ such that
\begin{itemize}
\item The boundary of~$C$ is the union of all the subcubes of~$\c{D}$,
\item for any subcubes~$C', C'' \in \c{D}$, the intersection~$C' \cap C''$ is either empty, or~$C$ or~$C'$ themselves, or a subcube of~$\c{D}$ with~$\dim(C' \cap C'') < \min(\dim(C), \min(C'))$.
\end{itemize}
The \defn{subcube poset} of the cubic subdivision~$\c{D}$ is the poset on~$\c{D} \cup \{C\}$ ordered by inclusion.
%A cubic subdivision \defn{realizes} a poset $P$ if its subcube poset is isomorphic to $P$.
\end{definition}

\begin{definition}
\label{def:cubicRealization}
A \defn{cubic realization} of a poset~$P$ is a map~$\gamma : P \to \R^d$ such that
\begin{itemize}
\item for any cover relation~$p \lessdot q$ in~$P$, the difference~$\gamma(p) - \gamma(q)$ is a positive multiple of some basis vector~$\b{e}_i$,
\item $\gamma(P)$ lies on the boundary of $\cube \big( \gamma(\min(P)), \gamma(\max(P)) \big)$.
\end{itemize}
\end{definition}

Note that our conventions are slightly unusual: we require that the cubic coordinates are decreasing along the poset, so that the maximum of the poset~$P$ has minimal cubic coordinates.
Our choice is driven by the fact that we want our cubic coordinates for $1$-lighted $n$-shades to coincide with the triwords of~\cite{Saneblidze, RiveraSaneblidze, Chapoton-Dyck, Combe, Muhle}.
We next illustrate these two notions of cubic realizations with the Lehmer code of a permutation and the bracket vector of a binary tree (or we should say adaptations of them, in order to stick with our conventions).
The fact that these indeed provide cubic realizations was established in \cite{SaneblidzeUmble-diagonals} with slightly different conventions, and will be generalized and proved with our conventions in \cref{prop:cubicalRealizationsMultiplihedra}.

\begin{figure}
	\centerline{\includegraphics[scale=.9]{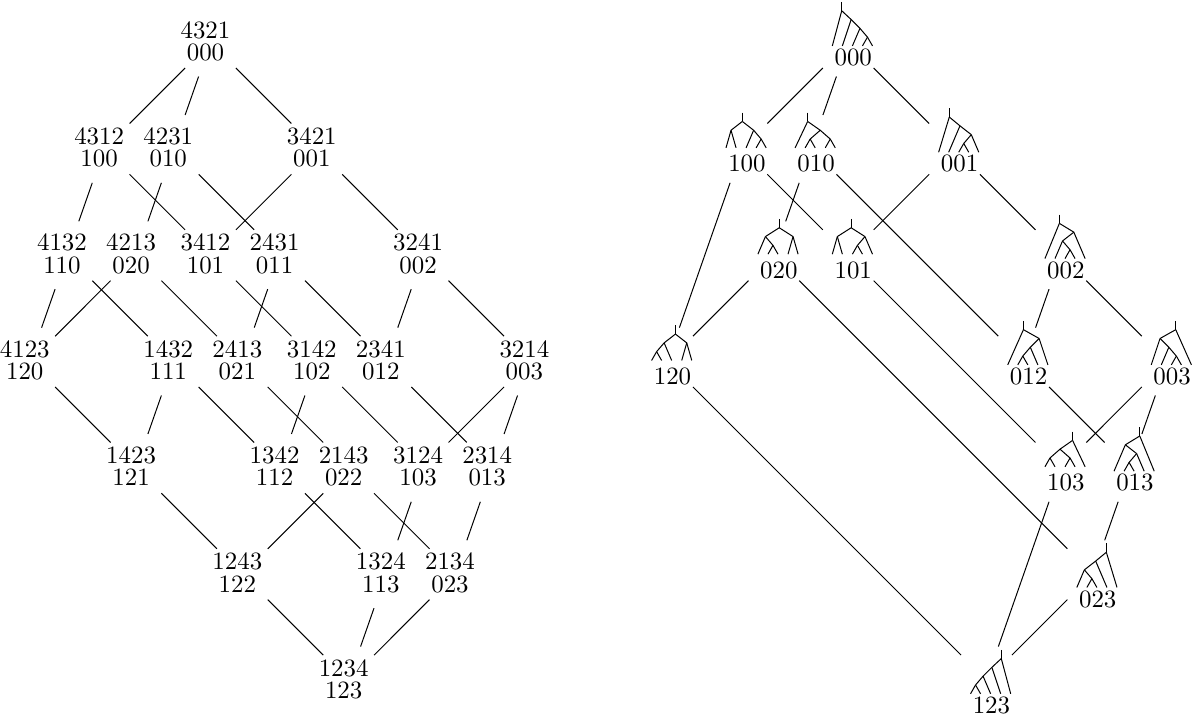}}
%	\centerline{\input{cubicRealizations}}
	\caption{Cubic realizations of the weak order via the Lehmer codes of permutations (left) and of the Tamari lattice via the bracket vectors of binary trees (right).}
	\label{fig:cubicRealizations}
\end{figure}

\pagebreak
\begin{example}
\label{exm:LehmerCode}
The \defn{Lehmer code}~\cite{Lehmer} of a permutation~$\sigma$ of~$[m]$ is the vector~$\b{L}(\sigma) \eqdef \smash{\big( L_j(\sigma) \big)_{j \in [m]}}$ where $L_j(\sigma) = \# \set{i < j}{\sigma^{-1}(i) < \sigma^{-1}(j)}$.
Note that ${L_j(\sigma) \in \{0, \dots, j-1\}}$, so that it is standard to forget the first coordinate (which is always~$0$). See \cref{fig:cubicRealizations}\,(left) for illustration.
\end{example}

The following proposition is due to S.~Saneblidze and R.~Umble~\cite[Sect.~2]{SaneblidzeUmble-diagonals}, but we reproduce the proof for convenience.

\begin{proposition}
\label{prop:LehmerCode}
The Lehmer codes of the permutations of~$[m]$ belong to the boundary of the cube~$[0,1] \times [0,2] \times \dots \times [0,m-1]$ and define
\begin{itemize}
\item a cubic realization of the weak order on the permutations of~$[m]$,
\item a cubic subdivision~$\set{\cube \big( \b{L}(\sigma), \b{L}(\tau) \big)}{\text{$\sigma \le \tau$ defining a face of~$\Perm[m]$}}$ whose subcube poset is isomorphic to the face lattice of the permutahedron~$\Perm[m]$.
\end{itemize} 
\end{proposition}

\begin{proof}
The proof works by induction on~$m$. The case $\Perm [1]$ is trivial. Assume that we have constructed the cubic subdivision~$\c{D}_{m-1}$ of~$[0,1] \times [0,2] \times \dots \times [0, m-2]$ for $\Perm[m-1]$.

Let~$\pi \eqdef \pi_1|\dots|\pi_k$ be an ordered partition of~$[m]$.
Let~$i \in [k]$ be such that~$m \in \pi_i$ and let~$p_\pi \eqdef |\pi_1| + \dots + |\pi_{i-1}|$ and~$q_\pi \eqdef |\pi_1| + \dots + |\pi_i| - 1$.
Let~$\pi'$ denote the ordered partition of~$[m-1]$ obtained by deleting~$m$ from~$\pi$ (and removing the potential empty part if~$m$ was alone in its part), and let~$C_{\pi'}$ denote the cube of~$\c{D}_{m-1}$ corresponding to~$\pi'$.
Finally, define~$C_\pi \eqdef C_{\pi'} \times [p_\pi, q_\pi]$.

We claim that the set~$\c{D}_m \eqdef \set{C_\pi}{\pi \text{ ordered partition of } [m]}$ defines a cubic subdivision of the cube~$[0,1] \times [0,2] \times \dots \times [0,m-1]$.
Indeed, we just need to prove by induction that the coarsening of ordered partitions corresponds to the inclusion of the corresponding cubes.
Let $\mu$ and~$\nu$ be two ordered partitions of~$[m]$ and $\mu'$ and~$\nu'$ be the ordered partitions of~$[m-1]$ obtained by deleting~$m$.
If~$\mu$ coarsens~$\nu$, then~$\mu'$ (weakly) coarsens~$\nu'$ so that~$C_{\mu'} \supseteq C_{\nu'}$ by induction, and~$[p_\mu, q_\mu] \supseteq [p_\nu, q_\nu]$, so that~$C_\mu \supseteq C_\nu$.

Finally, we observe by induction that for a permutation~$\sigma$, the  $0$-dimensional cube~$C_\sigma$ is at coordinate given by the Lehmer code~$\b{L}(\sigma)$.
\end{proof}

\begin{example}
\label{exm:bracketVector}
The \defn{bracket vector}~\cite{HuangTamari} of a binary $n$-tree~$T$ is the vector~$\b{B}(T) \eqdef \smash{\big( B_j(T) \big)_{j \in [n]}}$ where~$B_j(T)$ is the number of descendants of~$j$ which are smaller than~$j$ (for the usual inorder labeling of~$T$).
Equivalently, $B_j(T)$ is the number of leaves minus $1$ in the left subtree of~$j$.
Note that $B_j(\sigma) \in \{0, \dots, j-1\}$, so that it is standard to forget the first coordinate (which is always~$0$). See \cref{fig:cubicRealizations}\,(right) for illustration.   
\end{example}

The following proposition is again due to S.~Saneblidze and R.~Umble~\cite[Sect.~5]{SaneblidzeUmble-diagonals}, and we again reproduce the proof for convenience. The proof is very similar to that of \cref{prop:LehmerCode}.

\begin{proposition}
\label{prop:bracketVector}
The bracket vectors of the binary $n$-trees belong to the boundary of the cube~$[0,1] \times [0,2] \times \dots \times [0,n-1]$ and define
\begin{itemize}
\item a cubic realization of the Tamari lattice on the binary $n$-trees,
\item a cubic subdivision~$\set{\cube \big( \b{B}(S), \b{B}(T) \big)}{\text{$S \le T$ defining a face of~$\Asso[n]$}}$ whose subcube poset is isomorphic to the face lattice of the associahedron~$\Asso[n]$.
\end{itemize}
\end{proposition}

\begin{proof}
The proof works by induction on~$n$. The case $\Asso [1]$ is trivial. Assume that we have constructed the cubic subdivision~$\c{D}_{n-1}$ of~$[0,1] \times [0,2] \times \dots \times [0, n-2]$ for $\Asso[n-1]$.

Recall that the faces of the associahedron~$\Asso[n]$ are labeled by \defn{Schr\"oder $n$-trees}, that is, rooted plane trees with $n+1$ leaves where each node has at least two children. The inclusion of faces of the associahedron then corresponds to the edge contraction on Schr\"oder $n$-trees.

Let~$T$ be a Schr\"oder $n$-tree.
Let $s$ be the parent of the rightmost leaf. Let $p_T+1$ be the number of leaves on the second from the right branch growing from $s$, and let $q_T+1$ be the number of leaves on all the branches growing from $s$ except on the rightmost branch. Let~$T'$ denote the Schr\"oder $(n-1)$-tree obtained by deleting the rightmost leaf from~$T$ (and removing the node $s$ if it becomes unary after the deletion), and let~$C_{T'}$ denote the cube of~$\c{D}_{n-1}$ corresponding to~$T'$.
Finally, define~$C_T \eqdef C_{T'} \times [p_T, q_T]$.

We claim that the set~$\c{D}_n \eqdef \set{C_T}{T \text{ Schr\"oder $n$-tree}}$ defines a cubic subdivision of the cube $[0,1] \times [0,2] \times \dots \times [0,n-1]$.
Indeed, we just need to prove by induction that the contraction of inner edges in Schr\"oder trees corresponds to the inclusion of the corresponding cubes.
Let $S$ and~$T$ be two Schr\"oder $n$-trees, and let $S'$ and~$T'$ be the Schr\"oder $(n-1)$-trees obtained by deleting their rightmost leaves.
If~$S$ is obtained from~$T$ by contracting some set of its inner edges, then the same holds for~$S'$ and~$T'$ so that~$C_{S'} \supseteq C_{T'}$ by induction, and~$[p_S, q_S] \supseteq [p_T, q_T]$, so that~$C_S \supseteq C_T$.

Finally, we observe by induction that for a binary tree~$T$, the $0$-dimensional cube~$C_T$ is at coordinate given by the bracket vector $\b{B}(T)$.
\end{proof}

%%%%%%%%%

\subsection{Cubic $(m,n)$-multiplihedron}
\label{subsec:cubicMultiplihedron}

We now briefly present the cubic realizations of the $m$-painted $n$-tree coarsening lattice and rotation lattice.
These are a mixture of the Lehmer codes of the permutations of~$[m]$ and the bracket vectors of the binary $n$-trees.
The case~$m = 1$ was already discussed in \cite[above Figure 11]{SaneblidzeUmble-diagonals}.
It is convenient to use the poset~$\prec_{\PT}$ of~\cref{def:preorderPaintedTree} to define the cubic vector of~$\PT$.

\begin{definition}
The \defn{cubic vector} of a binary $m$-painted $n$-tree~$\PT$ is the vector $\b{C}(\PT) \eqdef \smash{\big( C_j(\PT) \big)_{j \in [m+n]}}$ where~$C_j(\PT) \eqdef \# \set{i < j}{i \prec_{\PT} j}$.
Note that $C_j(\PT) \in \{0, \dots, j-1\}$, so that it is standard to forget the first coordinate (which is always~$0$).
\end{definition}

\begin{example}
Observe that 
\begin{itemize}
\item when~$n = 0$, we have the Lehmer code of a permutation presented in \cref{exm:LehmerCode}, 
\item when~$m = 0$, we have the bracket vector of a binary tree presented in \cref{exm:bracketVector}.
\end{itemize}
\end{example}

The following statement is illustrated in \cref{fig:multiplihedronFreehedronLabeledCubic13,fig:multiplihedronFreehedronLabeledCubic3,fig:multiplihedronFreehedronLabeledCubic4}\,(top).
We skip its proof as it is a straightforward generalization of that of \cref{prop:LehmerCode,prop:bracketVector}.

\pagebreak
\begin{proposition}
\label{prop:cubicalRealizationsMultiplihedra}
The cubic vectors of $m$-painted $n$-trees belong to the boundary of the cube~$[0,1] \times [0,2] \times \dots \times [0,m+n-1]$ and define
\begin{itemize}
\item a cubic realization of the right rotation lattice on $m$-painted $n$-trees,
\item a cubic subdivision~$\set{\cube \big( \b{C}(\PT), \b{C}(\PT') \big)}{\text{$\PT \le \PT'$ defining a face of~$\Multiplihedron$}}$ whose subcube poset is isomorphic to the face lattice of the $(m,n)$-multiplihedron~$\Multiplihedron$.
\end{itemize}
\end{proposition}

%%%%%%%%%

\begin{figure}
	\centerline{\includegraphics[scale=.9]{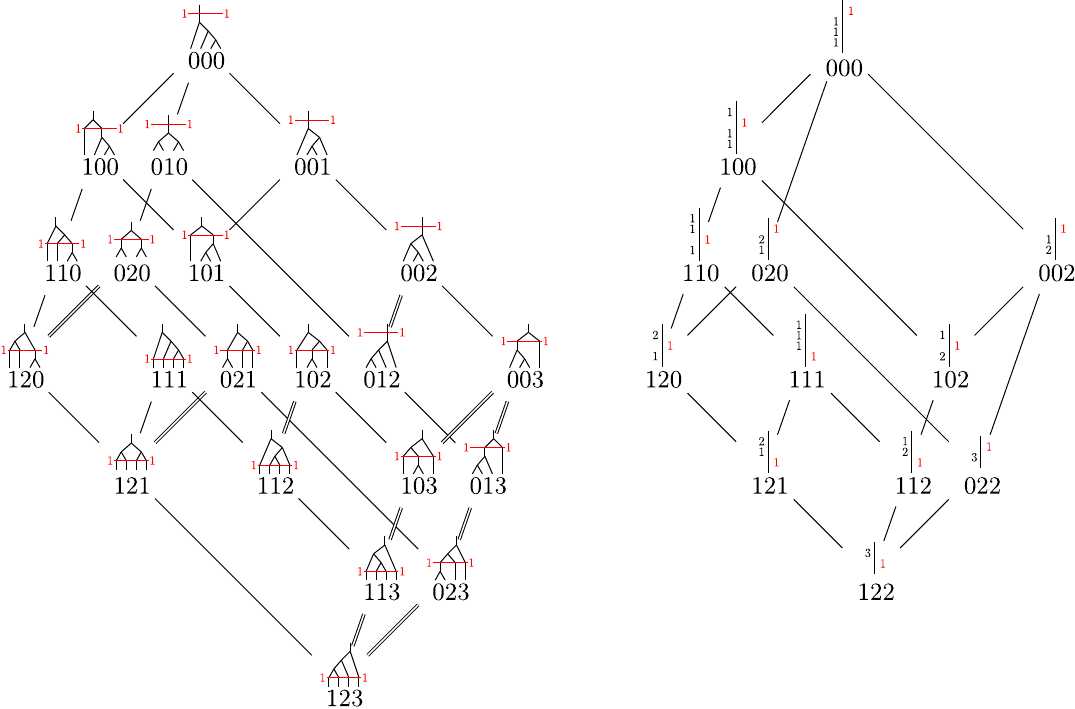}}
%	\centerline{\input{multiplihedronFreehedronLabeledCubicDoubled13}}
	\caption{Cubic realizations of the $1$-painted $3$-tree (left) and $1$-lighted $3$-shade (right) rotation lattices.}
	\label{fig:multiplihedronFreehedronLabeledCubic13}
\end{figure}

\afterpage{
\begin{figure}
	\vspace{.5cm}
	\centerline{\includegraphics[scale=.9]{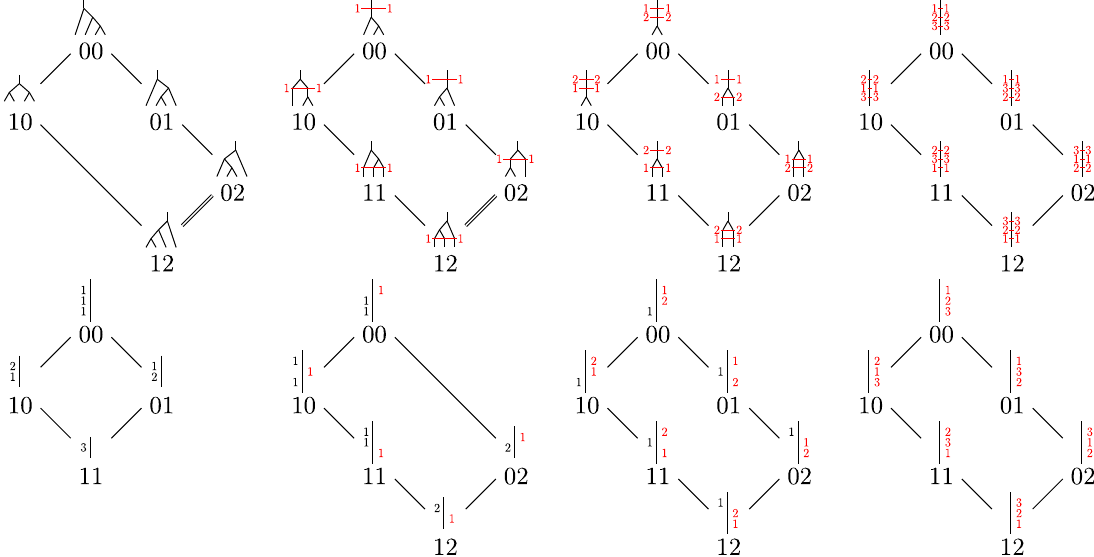}}
%	\centerline{\input{multiplihedronFreehedronLabeledCubicDoubled3}}
	\caption{Cubic realizations of the $m$-painted $n$-tree rotation lattice (top) and the $m$-lighted $n$-shade rotation lattice (bottom) for $(m,n) = (0,3)$, $(1,2)$, $(2,1)$, and~$(3,0)$ (left to right).}
	\label{fig:multiplihedronFreehedronLabeledCubic3}
\end{figure}
}

\afterpage{
\begin{figure}
	\vspace{.7cm}
	\centerline{\includegraphics[scale=.9]{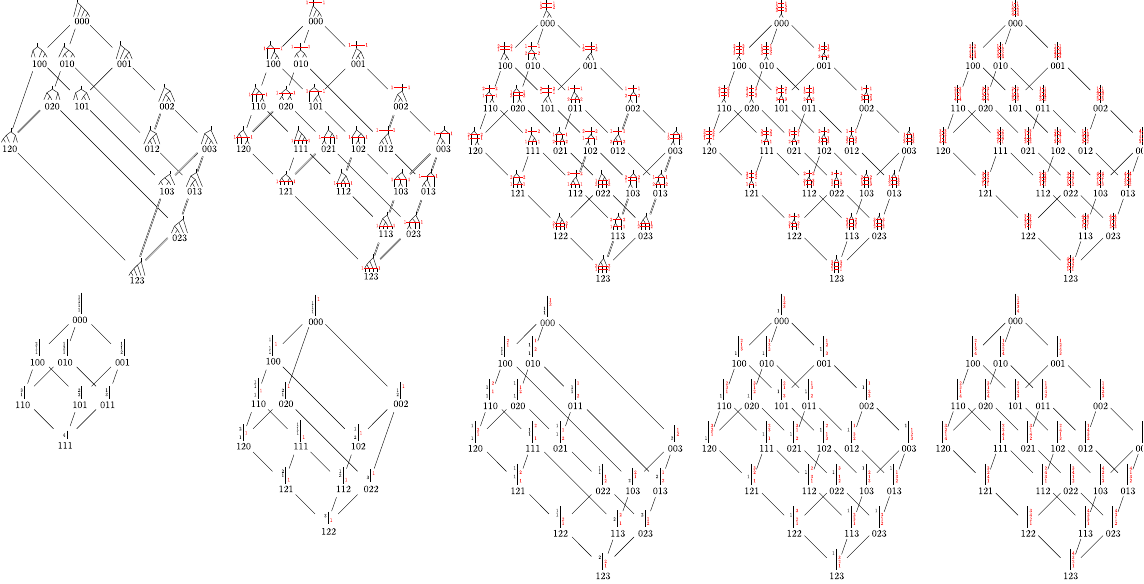}}
%	\centerline{\scalebox{.45}{\input{multiplihedronFreehedronLabeledCubicDoubled4}}}
	\caption{Cubic realizations of the $m$-painted $n$-tree rotation lattice (top) and the $m$-lighted $n$-shade rotation lattice (bottom) for $(m,n) = (0,4)$, $(1,3)$, $(2,2)$, $(3,1)$, and~$(4,0)$ (left to right).}
	\label{fig:multiplihedronFreehedronLabeledCubic4}
\end{figure}
}

%%%%%%%%%

%\begin{figure}[h]
%	\centerline{\includegraphics[scale=.9]{multiplihedronFreehedronLabeledCubic04}}}
%%	\centerline{\input{multiplihedronFreehedronLabeledCubic04}}
%	\caption{Cubic realizations of the $0$-painted $4$-tree (left) and $0$-lighted $4$-shade (right) rotation lattices.}
%\end{figure}

%\begin{figure}[h]
%	\centerline{\includegraphics[scale=.9]{multiplihedronFreehedronLabeledCubic13}}}
%%	\centerline{\input{multiplihedronFreehedronLabeledCubic13}}
%	\caption{Cubic realizations of the $1$-painted $3$-tree (left) and $1$-lighted $3$-shade (right) rotation lattices.}
%\end{figure}

%\begin{figure}[h]
%	\centerline{\includegraphics[scale=.9]{multiplihedronFreehedronLabeledCubic22}}}
%%	\centerline{\input{multiplihedronFreehedronLabeledCubic22}}
%	\caption{Cubic realizations of the $2$-painted $2$-tree (left) and $2$-lighted $2$-shade (right) rotation lattices.}
%\end{figure}

%\begin{figure}[h]
%	\centerline{\includegraphics[scale=.9]{multiplihedronFreehedronLabeledCubic31}}}
%%	\centerline{\input{multiplihedronFreehedronLabeledCubic31}}
%	\caption{Cubic realizations of the $3$-painted $1$-tree (left) and $3$-lighted $1$-shade (right) rotation lattices.}
%\end{figure}

%\begin{figure}[h]
%	\centerline{\includegraphics[scale=.9]{multiplihedronFreehedronLabeledCubic40}}}
%%	\centerline{\input{multiplihedronFreehedronLabeledCubic40}}
%	\caption{Cubic realizations of the $4$-painted $0$-tree (left) and $4$-lighted $0$-shade (right) rotation lattices.}
%\end{figure}

%%%%%%%%%

\subsection{Cubic $(m,n)$-Hochschild polytope}
\label{subsec:cubicHochschildPolytope}

We now provide cubic realizations for the $(m,n)$-Hoch\-schild polytope.
Unfortunately, the formula for the cubic coordinates of an $m$-lighted $n$-shade~$\LS$ is not just obtained by counting non-inversions in~$\prec_{\LS}$ (\ie pairs~$i < j$ with~$i \prec_{\LS} j$).
We thus first introduce a bijection between the $m$-lighted $n$-shades and the $(m,n)$-Hochschild words, generalizing the triwords of~\cite{Saneblidze, RiveraSaneblidze, Chapoton-Dyck, Combe, Muhle}.
We then use these $(m,n)$-Hochschild words to obtain cubic realizations.

%%%

\subsubsection{$(m,n)$-Hochschild words}

We start with $(m,n)$-words, defined as follows.

\begin{definition}
\label{def:multiwords}
A \defn{$(m,n)$-word} is a word~$w \eqdef w_1 \dots w_n$ of length~$n$ on the alphabet $\{0,1, ... , m+1\}$ such that
\begin{itemize}
    \item $w_1 \neq m+1$
    \item for $s \in [1,m]$, $w_i = s$ implies $w_j \ge s$ for all $j<i$
\end{itemize}
We denote by~$\multiwords$ the poset of $(m,n)$-words ordered reverse componentwise (\ie $w \le w'$ if and only if~$w_i \ge w'_i$ for all~$i \in [n]$).
\end{definition}

\pagebreak
\begin{example}
When~$m = 0$, the second condition is empty, so that the $(0,n)$-words are binary words of length~$n$ starting with a~$0$, and $\multiwords[0][n]$ is isomorphic to the boolean lattice on~$n-1$ letters.
When~$m = 1$, the $(1,n)$-words are precisely the triwords of~\cite{Saneblidze, RiveraSaneblidze, Chapoton-Dyck, Combe, Muhle}, and $\multiwords[1][n]$ is isomorphic to the Hochschild lattice.
\end{example}

\begin{definition}
A \defn{$(m,n)$-Hochschild word}  is a pair of $(\sigma, w)$ where $\sigma$ is a permutation of~$[m]$ and $w$ is an $(m,n)$-word.
\end{definition}

We now define a bijection between the $m$-lighted $n$-shades and the $(m,n)$-Hochschild words.
Recall that we denote by~$ps(x)$ the preceeding sum of an entry~$x$ in an $m$-lighted $n$-shade (see \cref{def:preorderLightedShade}).

\begin{definition}
\label{def:lightedShadesToHochschildWords}
Consider a unary $m$-lighted $n$-shade $\LS \eqdef (S, C, \sigma)$ and denote by~$s_1, \dots, s_r$ the values of the singleton tuples of~$S$ read from top to bottom.
We associate to~$\LS$ an $(m,n)$-Hochschild word~$(\sigma, w)$, where the permutation is the permutation~$\sigma$ of the labels of the cuts of~$\LS$ read from bottom to top, and the $(m,n)$-word~$w$ has $p$th entry~$w_p$ given by
\begin{itemize}
\item if there is~$j \in [r]$ such that~$p = ps(s_j)-m-s_j+1$, then the number of cuts below~$s_j$,
\item $m+1$ otherwise.
\end{itemize}
In other words, for each~$s_j$, we write the number of cuts below~$s_j$ followed by~$s_j-1$ copies of~$m+1$.
See \cref{fig:HochschildWords24,fig:HochschildWords13} for some examples.
\afterpage{
\begin{figure}[t]
	\centerline{\includegraphics[scale=.9]{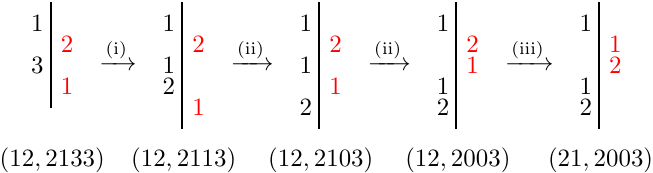}}
%	\centerline{
%	\begin{tabular}{c@{}c@{}c@{}c@{}c@{}c@{}c@{}c@{}c}
%		\lightedShade[1]{[1 [, tier=2 [3 [, tier=1 []]]]]}{1,2} &
%		\hspace{-.1cm}\raisebox{-1.2cm}{$\xrightarrow{\text{(i)}}$}\hspace{-.1cm} &
%		\lightedShade[1]{[1 [, tier=2 [1 [2 [, tier=1 []]]]]]}{1,2} &
%		\hspace{-.1cm}\raisebox{-1.2cm}{$\xrightarrow{\text{(ii)}}$}\hspace{-.1cm} &
%		\lightedShade[1]{[1 [, tier=2 [1 [, tier=1 [2 []]]]]]}{1,2} &
%		\hspace{-.1cm}\raisebox{-1.2cm}{$\xrightarrow{\text{(ii)}}$}\hspace{-.1cm} &
%		\lightedShade[1]{[1 [, tier=2 [, tier=1 [1 [2 []]]]]]}{1,2} &
%		\hspace{-.1cm}\raisebox{-1.2cm}{$\xrightarrow{\text{(iii)}}$}\hspace{-.1cm} &
%		\lightedShade[1]{[1 [, tier=1 [, tier=2 [1 [2 []]]]]]}{2,1} \\
%		$(12, 2133)$ &&
%		$(12, 2113)$ &&
%		$(12, 2103)$ &&
%		$(12, 2003)$ &&
%		$(21, 2003)$
%	\end{tabular}
%	}
	\caption{Some unary $2$-lighted $4$-shades and their $(2,4)$-Hochschild words.}
	\label{fig:HochschildWords24}
\end{figure}
}
\afterpage{
\begin{figure}
	\centerline{\includegraphics[scale=.9]{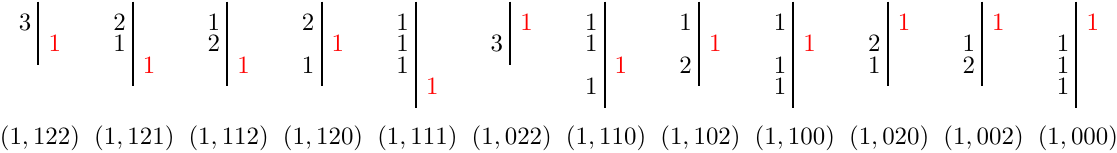}}
%	\centerline{
%	\begin{tabular}{c@{\hspace{-.2cm}}c@{\hspace{-.2cm}}c@{\hspace{-.2cm}}c@{\hspace{-.2cm}}c@{\hspace{-.2cm}}c@{\hspace{-.2cm}}c@{\hspace{-.2cm}}c@{\hspace{-.2cm}}c@{\hspace{-.2cm}}c@{\hspace{-.2cm}}c@{\hspace{-.2cm}}c}
%		\lightedShade[1]{[3 [, tier=1[]]]}{1} &
%		\lightedShade[1]{[2 [1 [, tier=1[]]]]}{1} &
%		\lightedShade[1]{[1 [2 [, tier=1[]]]]}{1} &
%		\lightedShade[1]{[2 [, tier=1[1 []]]]}{1} &
%		\lightedShade[1]{[1 [1 [1 [, tier=1[]]]]]}{1} &
%		\lightedShade[1]{[, tier=1[3 []]]}{1} &
%		\lightedShade[1]{[1 [1 [, tier=1[1 []]]]]}{1} &
%		\lightedShade[1]{[1 [, tier=1[2 []]]]}{1} &
%		\lightedShade[1]{[1 [, tier=1[1 [1 []]]]]}{1} &
%		\lightedShade[1]{[, tier=1[2 [1 []]]]}{1} &
%		\lightedShade[1]{[, tier=1[1 [2 []]]]}{1} &
%		\lightedShade[1]{[, tier=1[1 [1 [1 []]]]]}{1}
%		\\
%		$(1, 122)$ &
%		$(1, 121)$ &
%		$(1, 112)$ &
%		$(1, 120)$ &
%		$(1, 111)$ &
%		$(1, 022)$ &
%		$(1, 110)$ &
%		$(1, 102)$ &
%		$(1, 100)$ &
%		$(1, 020)$ &
%		$(1, 002)$ &
%		$(1, 000)$
%	\end{tabular}
%	}
	\caption{All unary $1$-lighted $3$-shades and $(1,3)$-Hochschild words.}
	\label{fig:HochschildWords13}
\end{figure}
}
\end{definition}

\begin{definition}
\label{def:HochschildWordsTolightedShades}
Conversely, we associate to an $(m,n)$-Hochschild word~$(\sigma, w)$ a unary $m$-lighted $n$-shade $\LS \eqdef (S, C, \sigma)$ where the labels of the cuts of~$\LS$ is given by the permutation~$\sigma$, and the $n$-shade~$S$ is the sequence of (either singleton or empty) tuples
\[
S \eqdef (s_{m,1}) \dots (s_{m,k_m})(\varnothing) \dots (\varnothing) (s_{i,1}) \dots (s_{i,k_i})(\varnothing) \cdots (\varnothing)(s_{0,1}) \dots (s_{0,k_0}),
\]
where the~$s_{i,j} \ge 1$ are such that
\[
w = m(m+1)^{s_{m,1}-1} \dots i(m+1)^{s_{1,k_i}-1} \dots i(m+1)^{s_{i,k_i}-1} \dots  0(m+1)^{s_{0,k_0}-1}.
\]
In other words, we place the $m$ cuts-to-be, and place a tuple~$(s)$ before the $(m-i+1)$st cut for each maximal subword of~$w$ of the form~$i(m+1)^{s-1}$.
See \cref{fig:HochschildWords24,fig:HochschildWords13} for some examples.
\end{definition}

%We leave to the reader the immediate verification of the following lemma.

\begin{lemma}
The maps of \cref{def:lightedShadesToHochschildWords,def:HochschildWordsTolightedShades} are inverse bijections between the unary $m$-lighted $n$-shades and the $(m,n)$-Hochschild words.
\end{lemma}

\pagebreak
\begin{proof}
First, the word associated to a unary $m$-lighted $n$-shade is an $(m,n)$-word.
Indeed, 
\begin{itemize}
\item the first letter is not $m+1$, because there are only $m$ cuts, 
\item as we are reading the shade from top to bottom, the numbers written before the number $s \in [1,m]$ come from higher entries that have at least $s$ cuts below them, so these numbers are at least~$s$.
\end{itemize}

Conversely, the sequence of tuples associated to an $(m,n)$-Hochschild word is a unary $m$-lighted $n$-shade.
Indeed, the total sum is the length of~$w$, and each tuple is either empty or a singleton contained in a singleton cut.

Finally, it is immediate to check that the two maps are inverse to each other.
\end{proof}

\begin{remark}
\label{rem:rotationHochschildWords}
Through the bijection of \cref{def:lightedShadesToHochschildWords}, we can thus transport  the rotation lattice on unary $m$-lighted $n$-shades to a lattice on $(m,n)$-Hochschild words.
The relation between \mbox{$(m,n)$-Hochschild} words can be described as follows.
For two Hochschild words~$(\sigma,v)$ and~$(\tau,w)$, we have~$(\sigma,v) \leq (\tau,w)$ if and only if
\begin{itemize}
\item $\sigma \leq \tau$ in weak order,
\item $v \leq w$ in~$\multiwords$ (meaning reverse coordinatewise),
\item there exists a reduced expression $\sigma^{-1} \circ \tau = \tau_{i_1} \circ \ldots \circ \tau_{i_k}$ (\ie a path in the permutahedron from $\sigma$ to $\tau$) and a sequence of $(m,n)$-words $v = h_0 \leq h_1 \leq h_2 \ldots \leq h_k \leq h_{k+1} = w$ such that the word~$h_l$ does not have a letter~$i_l$.   
\end{itemize}
See \cref{fig:rotationHochschildWords}.
The lattice is thus a subposet in the Cartesian product between the weak order on permutations of~$[m]$ and the $(m,n)$-word poset~$\multiwords$.
It would be nice to have a more explicit formulation of the last condition in the description of the relation~$(\sigma,v) \leq (\tau,w)$, but we were not able to find~it.
\afterpage{
\begin{figure}[t]
	\centerline{\includegraphics[scale=1.8]{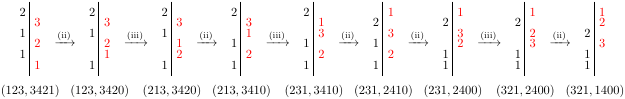}}
	\caption{Illustration of \cref{rem:rotationHochschildWords} on the transportation of the Hochschild lattice from unary $3$-lighted $4$-shades to the corresponding $(3,4)$-Hochschild words. Here, $(\sigma, v) = (123, 3421)$ and~$(\tau, w) = (321, 1400)$, the reduced expression is~$\sigma^{-1} \circ \tau = 321 = \tau_1 \circ \tau_2 \circ \tau_1$ and the sequence of $(3,4)$-words is~$v = h_0 = 3421 \le h_1 = 3420 \le h_2 = 3410 \le h_3 = 2400 \le h_4 = w = 1400$.}
	\label{fig:rotationHochschildWords}
\end{figure}
}
\end{remark}

Finally, as it is a fiber of the lattice morphism~$(\sigma, w) \mapsto \sigma$ from the $(m,n)$-Hochschild word rotation lattice to the weak order, we obtain that the $(m,n)$-word poset $\multiwords$ is a lattice.
As mentioned in \cref{rem:latticeProperties}, it seems to have much more interesting properties than the \mbox{$(m,n)$-Hochschild} word rotation lattice (for instance, it seems to be extremal, and its Coxeter polynomial seems to be a product of cyclotomic polynomials).

\begin{corollary}
The $(m,n)$-word poset $\multiwords$ is a lattice.
\end{corollary}

\begin{remark} 
The lattice $\multiwords[1][n]$ has a geometric interpretation in the context of homotopical algebra.
Specifically, the Hochschild polytope $\HP[1][n]$ has a polytopal subdivision whose directed $1$-skeleton is $\multiwords[m][n]$.
The Hochschild polytopes $\HP[1][n]$ form an operadic bimodule over the operad of skew cubes $\HP[0][n]$ in the category of CW-spaces~\cite{Poliakova}, and tensor powers of this bimodule over the operad are CW-isomorphic to this subdivision.
Algebraically this allows for the composition of sequences of morphisms of $A_{\infty}$-modules over DG-algebras (or of representations up to homotopy~\cite{AbadCrainicDherin}).
\end{remark}

%%%

\subsubsection{Cubic realizations}

Passing from the unary $m$-lighted $n$-shades to the $(m,n)$-Hochschild words allows us to construct cubic realizations for the $m$-lighted $n$-shade coarsening lattice and rotation lattice.

\begin{definition}
The \defn{cubic vector} of an $(m,n)$-Hochschild word is the vector obtained by the concatenation of the Lehmer code of~$\sigma$ (forgetting the first coordinates which is always~$0$) with the $(m,n)$-word~$w$.
The \defn{cubic vector}~$\b{C}(\LS)$ of a unary $m$-lighted $n$-shade~$\LS$ is the cubic vector of the associated $(m,n)$-Hochschild word via the bijection of \cref{def:lightedShadesToHochschildWords}.
\end{definition}

The following statement is illustrated in \cref{fig:multiplihedronFreehedronLabeledCubic13,fig:multiplihedronFreehedronLabeledCubic3,fig:multiplihedronFreehedronLabeledCubic4}\,(top).

\begin{proposition}
The cubic vectors of $m$-lighted $n$-shades belong to the boundary of the cube~$[0,1] \times [0,2] \times \dots \times [0,m-1] \times [0,1] \times [0,2] \times \dots \times [0,n-1]$ and define
\begin{itemize}
\item a cubic realization of the right rotation lattice on $m$-lighted $n$-shades,
\item a cubic subdivision~$\set{\cube \big( \b{C}(\LS), \b{C}(\LS') \big)}{\text{$\LS \le \LS'$ defining a face of~$\HP$}}$ whose subcube poset is isomorphic to the face lattice of the $(m,n)$-Hochschild polytope~$\HP$.
\end{itemize}
\end{proposition}

\begin{proof}
We proceed by induction on~$n$, starting from the cubic realization of the permutahedron from \cref{prop:LehmerCode} for the case $\HP[m][0]$. 
Assume that we have constructed the cubic subdivision~$\c{D}_{m,n-1}$ of~$[0,1] \times [0,2] \times \dots \times [0,m-1] \times [0,1] \times [0,2] \times \dots \times [0,n-2]$ for $\HP[m][n-1]$.

Let~$\LS \eqdef (S,C,\mu)$ be an $m$-lighted $n$-shade.
Let~$p_{\LS}$ be the sum of the sizes~$|\mu_i|$ for~$c_i$ strictly below the last entry of~$S$.
Let~$q_{\LS}$ be $n-1$ if the last entry of~$S$ is not the singleton~$\{1\}$, and $q_{\LS}$ be the sum of the sizes~$|\mu_i|$ for~$c_i$ weakly below the last entry of~$S$ otherwise.
Let~$\LS'$ denote the $m$-lighted $(n-1)$-shade obtained by decrementing the last entry of the last entry of~$\LS$ (and removing it if it becomes empty), and let~$C_{\LS'}$ denote the cube of~$\c{D}_{m,n-1}$ corresponding to~$\LS'$.
Finally, define~$C_{\LS} \eqdef C_{\LS'} \times [p_{\LS}, q_{\LS}]$.

We claim that the set~$\c{D}_{m,n} \eqdef \set{C_{\LS}}{\LS \text{ $m$-lighted $n$-shade}}$ defines a cubic subdivision of the cube~$[0,1] \times [0,2] \times \dots \times [0,m-1] \times [0,1] \times [0,2] \times \dots \times [0,n-1]$.
Indeed, we just need to prove by induction that coarsening in $m$-lighted $n$-shades corresponds to the inclusion of the corresponding cubes.
Let $\LS$ and~$\LS[T]$ be two $m$-lighted $n$-trees, and let $\LS'$ and~$\LS[T]'$ be the $m$-lighted $(n-1)$-shades obtained by decrementing their last entries.
If~$\LS$ coarsens~$\LS[T]$, then~$\LS'$ (weakly) coarsens~$\LS[T]'$ so that~$C_{\LS'} \supseteq C_{\LS[T]'}$ by induction, and~$[p_{\LS}, q_{\LS}] \supseteq [p_{\LS[T]}, q_{\LS[T]}]$, so that~$C_{\LS} \supseteq C_{\LS[T]}$.

Finally, we observe by induction that for a unary $m$-lighted $n$-shade~$\LS$, the $0$-dimensional cube~$C_{\LS}$ is at coordinate given by the cubic vector $\b{C}(\LS)$.
\end{proof}

%%%%%%%%%%%%%%%%%%%%%%%%%%%%%%%%%%%%%%

\section*{Acknowledgements}

We thank Frédéric Chapoton for suggesting to look for polytopal realizations of the Hochschild lattices.
This work started at the workshop ``Combinatorics and Geometry of Convex Polyhedra'' held at the Simons Center for Geometry and Physics in March 2023.
We are grateful to the organizers (Karim Adiprasito, Alexey Glazyrin, Isabella Novik, and Igor Pak) for this inspiring event, and to all participants for the wonderful atmosphere.
Finally, we are indebted to an anonymous referee, whose thorough and constructive reports largely improved the final quality and presentation of the paper.

%%%%%%%%%%%%%%%%%%%%%%%%%%%%%%%%%%%%%%

\bibliographystyle{alpha}
\bibliography{freehedron}
\label{sec:biblio}

%%%%%%%%%%%%%%%%%%%%%%%%%%%%%%%%%%%%%%

\clearpage
\appendix

\section{Enumeration tables}
\label{sec:tables}

\enlargethispage{.5cm}
All references like~\OEIS{A000142} are entries of the Online Encyclopedia of Integer Sequences~\cite{OEIS}. \\ [-.3cm]

%%%%%%%%%%

\subsection{Multiplihedra} ~
\label{subsec:tablesMultiplihedra}

\begin{table}[h]
	\centerline{\begin{tabular}{r|rrrrrrrrrr|l}
		$m \backslash n$ & 0 & 1 & 2 & 3 & 4 & 5 & 6 & 7 & 8 & 9 & \\
		\hline
		0 & . & 1 & 2 & 5 & 14 & 42 & 132 & 429 & 1430 & 4862 & \OEIS{A000108} \\
		1 & 1 & 2 & 6 & 21 & 80 & 322 & 1348 & 5814 & 25674 & & \OEIS{A121988} \\
		2 & 2 & 6 & 24 & 108 & 520 & 2620 & 13648 & 72956 & & & $2 \cdot \OEIS{A158826}$ \\
		3 & 6 & 24 & 120 & 660 & 3840 & 23220 & 144504 & & & & ? \\
		4 & 24 & 120 & 720 & 4680 & 31920 & 225120 & & & & \\
		5 & 120 & 720 & 5040 & 37800 & 295680 & & & & & \\
		6 & 720 & 5040 & 40320 & 342720 & & & & & & \\
		7 & 5040 & 40320 & 362880 & & & & & & & \\
		8 & 40320 & 362880 & & & & & & & & \\
		9 & 362880 & & & & & & & & & \\
		\hline
		& \OEIS{A000142} & \OEIS{A000142} & \OEIS{A000142} & \OEIS{A084253} & ? & & & & & & $m! \cdot \OEIS{A158825}$
	\end{tabular}}
	\caption{Number of vertices of the multiplihedra~$\Multiplihedron$. See~$\OEIS{A158825}$.}
	\label{table:verticesMultiplihedra}
\end{table}

\begin{table}[h]
	\centerline{\begin{tabular}{r|rrrrrrrrrr|l}
		$m \backslash n$ & 0 & 1 & 2 & 3 & 4 & 5 & 6 & 7 & 8 & 9 & \\
		\hline
		0 & . & 1 & 2 & 5 & 9 & 14 & 20 & 27 & 35 & 44 & \OEIS{A000096} \\
		1 & 1 & 2 & 6 & 13 & 25 & 46 & 84 & 155 & 291 & & \OEIS{A335439} \\
		2 & 2 & 6 & 14 & 29 & 57 & 110 & 212 & 411 & & & ? \\
		3 & 6 & 14 & 30 & 61 & 121 & 238 & 468 & & & & \\
		4 & 14 & 30 & 62 & 125 & 249 & 494 & & & & & \\
		5 & 30 & 62 & 126 & 253 & 505 & & & & & & \\
		6 & 62 & 126 & 254 & 509 & & & & & & & \\
		7 & 126 & 254 & 510 & & & & & & & & \\
		8 & 254 & 510 & & & & & & & & & \\
		9 & 510 & & & & & & & & & & \\
		\hline
		& \OEIS{A000918} & \OEIS{A000918} & \OEIS{A000918} & \OEIS{A036563} & \OEIS{A048490} & ? & & & & &
	\end{tabular}}
	\caption{Number of facets of the multiplihedra~$\Multiplihedron$.}
	\label{table:facetsMultiplihedra}
\end{table}

\begin{table}[h]
	\centerline{\begin{tabular}{r|rrrrrrrrrr|l}
		$m \backslash n$ & 0 & 1 & 2 & 3 & 4 & 5 & 6 & 7 & 8 & 9 & \\
		\hline
		0 & . & 1 & 3 & 11 & 45 & 197 & 903 & 4279 & 20793 & 103049 & \OEIS{A001003} \\
		1 & 1 & 3 & 13 & 67 & 381 & 2311 & 14681 & 96583 & 653049 & & ? \\
		2 & 3 & 13 & 75 & 497 & 3583 & 27393 & 218871 & 1810373 & & & \\
		3 & 13 & 75 & 541 & 4375 & 38073 & 349423 & 3341753 & & & & \\
		4 & 75 & 541 & 4683 & 44681 & 454855 & 4859697 & & & & & \\
		5 & 541 & 4683 & 47293 & 519847 & 6055401 & & & & & & \\
		6 & 4683 & 47293 & 545835 & 6790697 & & & & & & & \\
		7 & 47293 & 545835 & 7087261 & & & & & & & & \\
		8 & 545835 & 7087261 & & & & & & & & & \\
		9 & 7087261 & & & & & & & & & & \\
		\hline
		& \OEIS{A000670} & \OEIS{A000670} & \OEIS{A000670} & ? & & & & & & &
	\end{tabular}}
	\caption{Total number of faces of the multiplihedra~$\Multiplihedron$. The empty face is not counted, but the polytope itself is.}
	\label{table:facesMultiplihedra}
\end{table}

%%%%%%%%%%

\clearpage
\subsection{Hochschild polytopes} ~
\label{subsec:tablesHochschildPolytope}

\begin{table}[h]
	\centerline{\begin{tabular}{r|rrrrrrrrrr|l}
		$m \backslash n$ & 0 & 1 & 2 & 3 & 4 & 5 & 6 & 7 & 8 & 9 & \\
		\hline
		0 & . & 1 & 2 & 4 & 8 & 16 & 32 & 64 & 128 & 256 & \OEIS{A000079} \\
		1 & 2 & 2 & 5 & 12 & 28 & 64 & 144 & 320 & 704 & & \OEIS{A045623} \\
		2 & 6 & 6 & 18 & 50 & 132 & 336 & 832 & 2016 & & & ? \\
		3 & 24 & 24 & 84 & 264 & 774 & 2160 & 5808 & & & & \\
		4 & 120 & 120 & 480 & 1680 & 5400 & 16344 & & & & & \\
		5 & 720 & 720 & 3240 & 12480 & 43560 & & & & & & \\
		6 & 5040 & 5040 & 25200 & 105840 & & & & & & & \\
		7 & 40320 & 40320 & 221760 & & & & & & & & \\
		8 & 362880 & 362880 & & & & & & & & & \\
		9 & 3628800 & & & & & & & & & & \\
		\hline
		& \OEIS{A000142} & \OEIS{A000142} & \OEIS{A038720} & ? & & & & & & 
	\end{tabular}}
	\caption{Number of vertices of the Hochschild polytope~$\HP$. See~$\OEIS{A158825}$.}
	\label{table:verticesHochschildPolytope}
\end{table}

\begin{table}[h]
	\centerline{\begin{tabular}{r|rrrrrrrrrr|l}
		$m \backslash n$ & 0 & 1 & 2 & 3 & 4 & 5 & 6 & 7 & 8 & 9 & \\
		\hline
		0 & . & 0 & 2 & 4 & 6 & 8 & 10 & 12 & 14 & 16 & \OEIS{A005843} \\
		1 & 0 & 2 & 5 & 8 & 11 & 14 & 17 & 20 & 23 & & \OEIS{A016789} \\
		2 & 2 & 6 & 11 & 16 & 21 & 26 & 31 & 36 & & & \OEIS{A016861} \\
		3 & 6 & 14 & 23 & 32 & 41 & 50 & 59 & & & & \OEIS{A017221} \\
		4 & 14 & 30 & 47 & 64 & 81 & 98 & & & & & ? \\
		5 & 30 & 62 & 95 & 128 & 161 & & & & & & \\
		6 & 62 & 126 & 191 & 256 & & & & & & & \\
		7 & 126 & 254 & 383 & & & & & & & & \\
		8 & 254 & 510 & & & & & & & & & \\
		9 & 510 & & & & & & & & & & \\
		\hline
		& \OEIS{A000918} & \OEIS{A000918} & \OEIS{A055010} & \OEIS{A000079} & \OEIS{A083575} & \OEIS{A164094} & \OEIS{A164285} & \OEIS{A140504} & ? &
	\end{tabular}}
	\caption{Number of facets of the Hochschild polytope~$\HP$.}
	\label{table:facetsHochschildPolytope}
\end{table}

\begin{table}[h]
	\centerline{\begin{tabular}{r|rrrrrrrrrr|l}
		$m \backslash n$ & 0 & 1 & 2 & 3 & 4 & 5 & 6 & 7 & 8 & 9 & \\
		\hline
		0 & . & 1 & 3 & 9 & 27 & 81 & 243 & 729 & 2187 & 6561 & \OEIS{A000244} \\
		1 & 1 & 3 & 11 & 39 & 135 & 459 & 1539 & 5103 & 16767 & & ? \\
		2 & 3 & 13 & 57 & 233 & 909 & 3429 & 12609 & 45441 & & & \\
		3 & 13 & 75 & 383 & 1767 & 7635 & 31491 & 125415 & & & & \\
		4 & 75 & 541 & 3153 & 16169 & 76437 & 341205 & & & & & \\
		5 & 541 & 4683 & 30671 & 172839 & 885795 & & & & & & \\
		6 & 4683 & 47293 & 343857 & 2110313 & & & & & & & \\
		7 & 47293 & 545835 & 4362383 & & & & & & & & \\
		8 & 545835 & 7087261 & & & & & & & & & \\
		9 & 7087261 & & & & & & & & & \\
		\hline
		& \OEIS{A000670} & \OEIS{A000670} & ? & & & & & & & &
	\end{tabular}}
	\caption{Total number of faces of the Hochschild polytope~$\HP$. The empty face is not counted, but the polytope itself is.}
	\label{table:facesHochschildPolytope}
\end{table}

%%%%%%%%%%

\clearpage
\subsection{Singletons} ~
\label{subsec:tableSingletons}

\begin{table}[h]
	\centerline{\begin{tabular}{r|rrrrrrrrrr|l}
		$m \backslash n$ & 0 & 1 & 2 & 3 & 4 & 5 & 6 & 7 & 8 & 9 & \\
		\hline
		0 &. & 1 & 2 & 3 & 5 & 8 & 13 & 21 & 34 & 55 & \OEIS{A000045} \\
		1 & 1 & 2 & 4 & 7 & 12 & 20 & 33 & 54 & 88 & & \OEIS{A000071} \\
		2 & 2 & 6 & 14 & 28 & 52 & 92 & 158 & 266 & & & ? \\
		3 & 6 & 24 & 66 & 150 & 306 & 582 & 1056 & & & & \\
		4 & 24 & 120 & 384 & 984 & 2208 & 4536 & & & & & \\
		5 & 120 & 720 & 2640 & 7560 & 18600 & & & & & & \\
		6 & 720 & 5040 & 20880 & 66240 & & & & & & & \\
		7 & 5040 & 40320 & 186480 & & & & & & & & \\
		8 & 40320 & 362880 & & & & & & & & & \\
		9 & 362880 & & & & & & & & & \\
		\hline
		& \OEIS{A000142} & \OEIS{A000142} & ? & & & & & & & &
	\end{tabular}}
	\caption{Number of shadow singletons, \ie common vertices of the $(m,n)$-multiplihedron~$\Multiplihedron$ and the $(m,n)$-Hochschild polytope~$\HP$.}
	\label{table:singletons}
\end{table}

\end{document}